%% file: main.tex
\documentclass[letterpaper,11pt]{article}
\input{pkgs_and_defns}

\usepackage{algorithm}
\usepackage{algpseudocode}
\usepackage[margin=1in]{geometry}  
\usepackage{graphicx}
\usepackage{amsthm}
\usepackage{amssymb}
\usepackage{authblk}
\usepackage[title]{appendix}

\usepackage{extarrows}
\usepackage{amsmath,mathrsfs}
\makeatletter
\let\over=\@@over \let\overwithdelims=\@@overwithdelims
\let\atop=\@@atop \let\atopwithdelims=\@@atopwithdelims
\let\above=\@@above \let\abovewithdelims=\@@abovewithdelims
\makeatother
\usepackage{rotating}

\usepackage{ifpdf}

\usepackage{subfig}
\usepackage{psfrag}

\usepackage{prettyref}
\usepackage{enumitem}

\usepackage{tikz}
\usetikzlibrary{arrows}
\tikzstyle{int}=[draw, fill=blue!20, minimum size=2em]
\tikzstyle{dot}=[circle, draw, fill=blue!20, minimum size=2em]
\tikzstyle{init} = [pin edge={to-,thin,black}]

\usepackage[
colorlinks=true,
citecolor=red,
linkcolor=blue,
anchorcolor=red,
urlcolor=blue,
pdfauthor={Anzo Teh}
]{hyperref}

\usepackage[all]{xy}

\usepackage{mathtools}

\usepackage{ifthen}
\newboolean{aos}
\setboolean{aos}{FALSE}

\usepackage{amssymb, amsthm}
\newtheorem{theorem}{Theorem}
\newtheorem{definition}{Definition}
\newtheorem{lemma}{Lemma}
\newtheorem{corollary}[lemma]{Corollary}
\newtheorem{assump}{Assumption}
\newtheorem{remark}{Remark}
\newtheorem{prop}{Proposition}
\begin{document}
	
\ifpdf
	\DeclareGraphicsExtensions{.pgf,.jpg}
	\graphicspath{{figures/}{plots/}}
	\fi

    \title{Function estimation in the empirical Bayes setting}

\author{Benjamin Kang, Yury Polyanskiy and Anzo Teh\thanks{
            B.K., Y.P., and A.T. are with the Department of EECS, MIT, Cambridge,
		MA, emails: \url{benkang@alum.mit.edu}, \url{yp@mit.edu} and \url{anzoteh@mit.edu}. }}

	\maketitle

\begin{abstract}
    We study \emph{function estimation} in the empirical Bayes setting for Poisson and normal
    means. Specifically, given observations $X_i\sim f(\cdot; \theta_i)$ with latent parameters $\theta_i\sim \pi$, the goal is to estimate $\mathbb{E}_{\pi}[\ell(\theta)|X = x]$. 
    This task lies between classical deconvolution (recovering the full prior $\pi$), and standard empirical Bayes mean estimation. 
    While the minimax risk for estimating $\pi$ in the Wasserstein distance is known to decay only logarithmically, 
    we show that estimating the corresponding posterior smooth functionals admits dramatically faster rates.
    In particular, for polynomial functions of degree $k$ in the Poisson model, 
    we establish a tight total regret bound of $\Theta((\frac{\log n}{\log \log n})^{k+1})$ and $\Theta((\log n)^{2k+1})$ for bounded and subexponential priors, respectively, attainable by estimators mimicking those that achieve optimal regret for the mean estimation problem (Robbins, minimum distance, ERM). 
    In the normal means model, we establish tight total regret bound of $\Theta((\frac{\log n}{\log \log n})^{k+1})$ for bounded priors, and bounds that match up to a polylogarithmic factor for subgaussian priors. 
    Our analysis identifies the approximation-theoretic origin of this improvement: smooth functions can be well-approximated by low-degree polynomials, whereas Lipschitz functions have only $O(\frac{1}{k})$ degree-$k$ polynomial approximation error. The results reveal a sharp hierarchy in the difficulty of empirical Bayes problems: ranging from slow, logarithmic deconvolution to near-parametric convergence for smooth posterior functionals, and establish new connections between nonparametric empirical Bayes theory, polynomial approximation, and statistical inverse problems.
\end{abstract}
 
\tableofcontents

\maketitle

\section{Introduction}

\begin{figure}[H]
	\centering
	\subfloat[RMSE\label{fig:rmse}]{
		\includegraphics[width=0.48\linewidth]{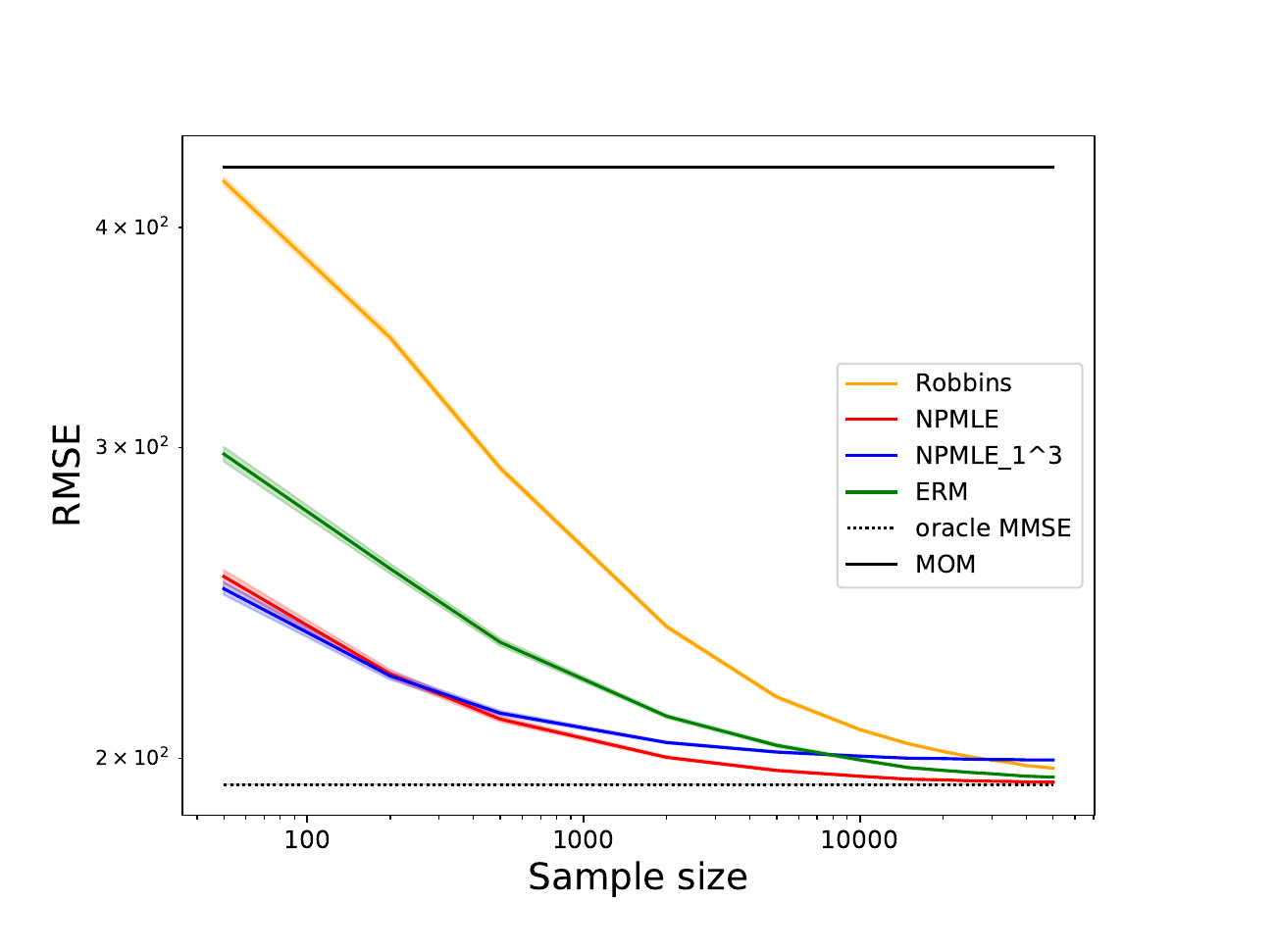}
	}
	\subfloat[Regret\label{fig:regret}]{
		\includegraphics[width=0.48\linewidth]{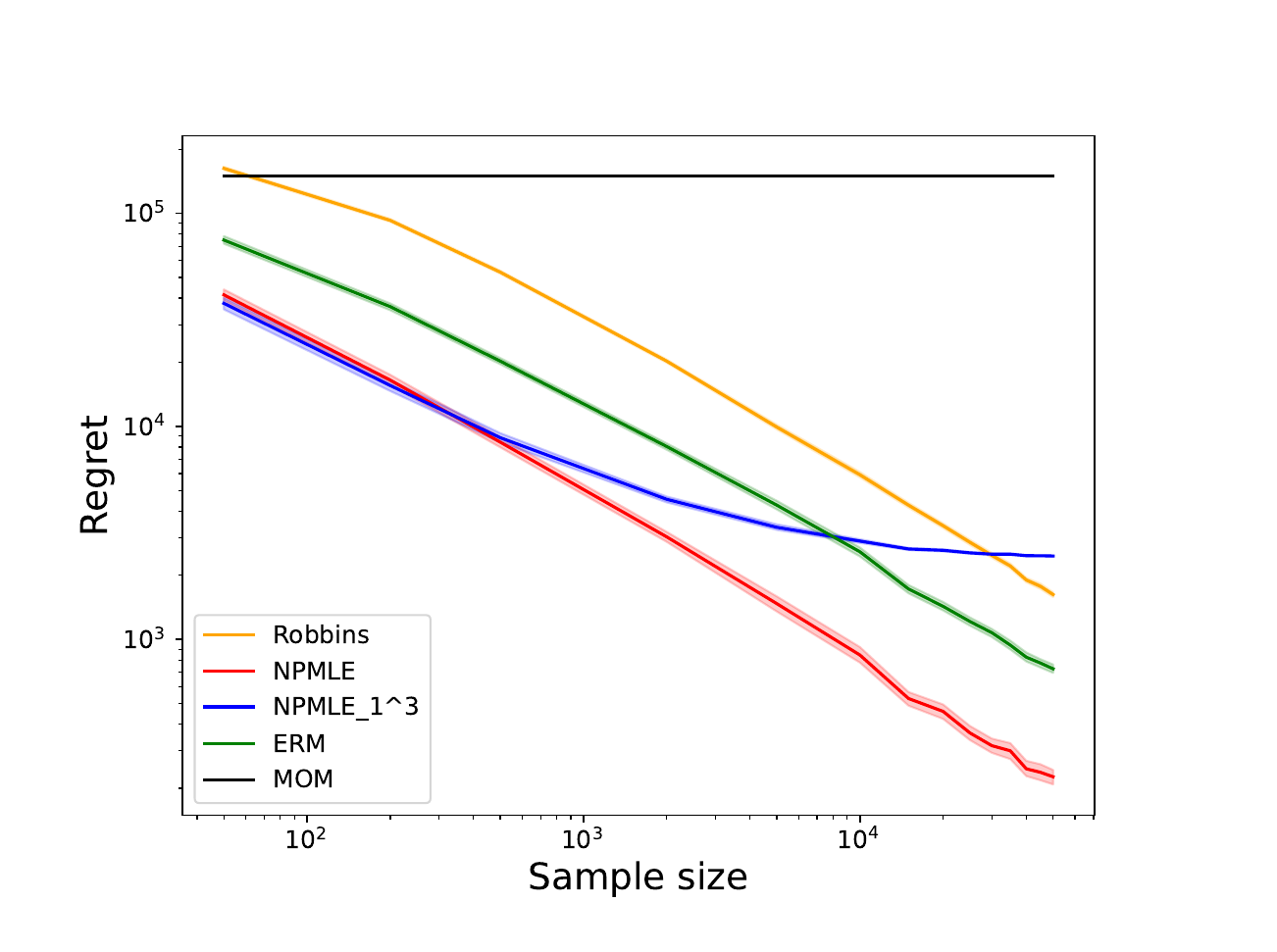}
	}
	\caption{RMSE and regret (gap to oracle-MMSE as defined in \prettyref{def:mmse}) for estimating $\ell(\theta) = \theta^3$ in the Poisson model with $\pi = \mathsf{Unif}([0, 10])$.
		The proper estimators $\Trobk, \Tnpmlek, \Termk$ (defined later in \prettyref{sec:poisson-estimators} along with $\hat{T}_{\mathsf{MOM}, k}$) all achieve vanishing (minimax optimal) \emph{per-coordinate} regret; their differences may be attributable to estimator-specific constants.
		The naïve plugin $\widehat{\theta^3} = (\hat{\theta}_{\mathsf{NPMLE}})^3$ incurs significant bias and is eventually dominated by even $\Trobk$.
		Despite slow $W_1$-convergence of $\hat{\pi}_{\mathsf{NPMLE}}$ to $\pi$, the proper
		NPMLE-based estimator $\Tnpmlek$ achieves fast (near-parametric) regret decay.
		Details on experimental setup can be found in \prettyref{app:exp-details}.
	}
	\label{fig:intro_unif10}
\end{figure}

Empirical Bayes (EB), introduced by Robbins \cite{Rob51, Rob56}, exploits the shared structure among independent observations to reduce estimation risk.
It has found applications in computational biology \cite{efron2001empirical}, sports prediction \cite{brown2008season}, and false discovery rate control \cite{benjamini1995controlling}; see \cite{efron2012large} for a comprehensive treatment.

The EB setup is as follows: latent parameters $\theta_1, \ldots, \theta_n \simiid \pi$ (unknown prior), observations $X_i \sim f(\cdot; \theta_i)$ (known channel).
Given $(X_1, \ldots, X_n)$, the classical task of EB is to estimate $(\theta_1,\ldots,\theta_n)$. In
this paper, we consider a variation in which the task is to estimate a function of $\theta_j$'s, namely
$(\ell(\theta_1), \ldots, \ell(\theta_n))$.
When $\pi$ is known a priori, the optimal estimator under squared loss is $\hat{t}_{\pi,\ell}(x) =
\mathbb{E}_\pi[\ell(\theta) \mid X = x]$, whose mean square error (MSE) is known as the Bayes risk or
\emph{oracle minimum MSE (MMSE)}.
Without knowing $\pi$, performance is measured by the \emph{total regret}: excess MSE over the Bayes risk summed over $n$ observations.

For the mean estimation ($\ell(\theta) = \theta$) of the Poisson model, the minimax regret is well understood:
\begin{itemize}
	\item Poisson, bounded priors: $\Theta\bigl((\tfrac{\log n}{\log \log n})^2\bigr)$ \cite{BGR13, polyanskiy2021sharp}
	\item Poisson, subexponential priors: $\Theta((\log n)^3)$ \cite{polyanskiy2021sharp, JPW24, JPTW23}
\end{itemize}
These rates are achieved by Robbins' estimator, nonparametric maximum likelihood (NPMLE), and
empirical risk minimization (ERM)-based methods.

At the opposite extreme is deconvolution --- recovering $\pi$ under the Wasserstein metric $W_1$. (We
call it the opposite extreme because by Kantorovich duality, estimating $\pi$ is equivalent to estimating all Lipschitz
functions of $\theta$, as opposed to only estimating a linear function as in the original EB).
Here convergence is much slower: $\Theta(\tfrac{\log\log n}{\log n})$ for Poisson \cite{miao2024fisher} and $\Theta(1/\sqrt{\log n})$ for normal means \cite{dedecker2013minimax}.
This contrast is surprising: mean estimation achieves near-parametric rates \footnote{$O(\mathsf{polylog(n)})$ total regret is near parametric; refer to \prettyref{def:regret} and \prettyref{lmm:totregretind} for its connection with individual (single-letter) regret.}, while prior recovery is logarithmically slow.

This paper studies the intermediate regime of polynomial functionals $\ell(\theta) = \theta^k$.
A key observation is that the naive plugin estimator $(\hat{\theta}_{\mathsf{NPMLE}})^k$ performs poorly due to bias (\prettyref{fig:intro_unif10}), while the proper estimators based on the generalized Tweedie formula
\begin{equation}\label{eq:tweedie_intro}
	\mathbb{E}_\pi[\theta^k \mid X = x] = \frac{(x+1)_k \, f_\pi(x+k)}{f_\pi(x)}
\end{equation}
achieve optimal rates (see Fig.~\ref{fig:intro_unif10}, where $(x)_k = x(x+1)\cdots(x+k-1)$ denotes the Pochhammer symbol).

\textbf{Contributions.} In this paper, we show the following results:
\begin{enumerate}
	\item \textit{Poisson model, tight bounds for $\ell(\theta) = \theta^k$:} We establish matching upper and lower bounds:
	\begin{itemize}
		\item Bounded priors $\calP([0,h])$: $\Theta\bigl((\tfrac{\log n}{\log \log n})^{k+1}\bigr)$
		\item Subexponential priors $\subexpo(s)$: $\Theta\bigl((\log n)^{2k+1}\bigr)$
	\end{itemize}
	The upper bounds are achieved by Robbins-type, NPMLE, and ERM estimators (\prettyref{sec:upper_bound}).
	
	\item \textit{Normal means model, bounds on $\ell(\theta) := \theta^k$ up to polylog factors}. We prove the following upper and lower bounds: 
	\begin{itemize}
		\item Bounded priors $\calP([-h, h])$: 
		$\Theta\bigl((\tfrac{\log n}{\log \log n})^{k+1}\bigr)$; 
		
		\item Subgaussian priors $\SubG(s)$: 
		$\Omega\bigl((\log n)^{k+1}\bigr)$ and 
		$O\bigl((\log n)^{k+2}(\log \log n)^{2k}\bigr)$. 
	\end{itemize}
	
	\item \textit{Both models, bounded priors, smooth functionals:} 
	For $\ell \in C^p$ with $\ell^{(p)} \in \mathrm{Lip}_\alpha$ where $0<\alpha\le 1$, we establish matching bounds of $\Theta(n(\frac{\log n}{\log \log n})^{2\beta})$ where $\beta = p + \alpha$.
		We also provide upper bounds for analytic functions. 
\end{enumerate}

On the technical side, for upper bounds, we extend the offset Rademacher complexity framework of \cite{JPTW23} and polynomial approximation analysis in \cite{chen2026sharp} to handle the modified loss for the $\theta^k$-estimation.
For lower bounds, we analyze higher-order derivatives of the basis functions from \cite{polyanskiy2021sharp}, connecting posterior moment estimation to the analytic structure of these functions.

\textbf{Transport maps between $\hat \pi$ and $\pi$.}
The slow deconvolution rate reflects poor polynomial approximation of 1-Lipschitz functions:
degree-$k$ polynomial approximation incurs $\Omega(1/k)$ error and thus requires estimating high-degree
polynomials $\ell(\theta)$, whose rate is exponential in $k$. This provides us with a perspective on the so-called 
``EB miracle'': despite $W_1(\hat{\pi}, \pi) = \Omega(\tfrac{\log\log n}{\log n})$,
the plugin Bayes estimator $\hat{t}_{\hat{\pi}}$ estimates
$\EE_{\pi}[\theta|X]$ at almost parametric rate. 
The key observation is that the posterior mean is a smooth functional
of $\pi$. Thus, integrating over $\hat{\pi}$ to obtain $\hat{t}_{\hat{\pi}}$ smooths out all
small-scale mismatches between $\pi$ and $\hat{\pi}$. Consequently, we make an important
qualitative conclusion (under
Gaussian and Poisson models at least): the large transport distance $W_1(\hat{\pi},\pi)$ is a result of
needing to move a lot of small masses to small distances. Indeed, we show in \prettyref{prop:transport-map} that if a large transport over a fixed gap was needed
somewhere, for example, if
$$\hat{\pi}((-\infty,1/2])>\pi((-\infty,2/3])+\Omega(\tfrac{\log \log n}{\log n})\,,$$ then there would
exist a polynomial of degree $O(\log \log n)$ whose expectations under $\hat{\pi}$ and $\pi$ differ by $\Omega(\frac{\log \log n}{\log n})$, thus
violating our results in \prettyref{thm:poisson_bound}  together with the reduction from posterior-functional estimation to prior-functional estimation in \prettyref{prop:reduction}.

\textbf{Paper organization.} 
The rest of the paper is organized as follows. 
We detail our main results in \prettyref{sec:main}. 
In particular, in \prettyref{sec:poisson_model} we describe the estimators of the Poisson model; in \prettyref{sec:poisson_bound} and \prettyref{sec:normal_bound} we focus on the case where $\ell$ is polynomial, and provide regret bounds for the Poisson and normal means models, respectively. 
In \prettyref{sec:main-result-smooth} we present a unified result for the case where $\ell$ is a smooth function.
In \prettyref{sec:preliminaries} we provide some background context needed for the rest of the paper. 
In \prettyref{sec:lower-bound} we provide proofs for lower bounds. 
For $\ell(\theta)$ polynomial, we prove the upper bounds for Poisson and normal means in \prettyref{sec:upper_bound} and \prettyref{sec:upper_bound_gsn}, respectively.
In \prettyref{sec:smooth}, we provide proofs on smooth functions. 
Technical auxiliary proofs, as well as commentary on the compound decision version of the problem, are in the appendices.

\section{Main results}\label{sec:main}

\subsection{Empirical Bayes task}

We formally define the empirical Bayes task. 
Recall that $\theta_1, \cdots, \theta_n$ are i.i.d. from $\pi$, 
and $X_i\sim f(\cdot; \theta_i)$ from some known conditional distribution $f(\cdot; \theta)$ parametrized by $\theta$. 
Upon seeing samples $X_1^n := (X_1, \cdots, X_n)$, 
one produces an estimate $\hat{T}(X_1,\cdots, X_n) = (\hat{T}_1(X_1^n), \cdots, \hat{T}_n(X_1^n))$ such that the MSE 
$\frac 1n \sum_{i = 1}^n \bbE[(\hat{T}_i(X_1^n) - \ell(\theta_i))^2]$ is minimized. 
If $\pi$ is known, the minimizer of this MSE is the Bayes estimator, 
$\hat T_{\pi, \ell}(X_1^n) = (\hat{t}_{\pi, \ell}(X_1), \cdots, \hat{t}_{\pi, \ell}(X_n))$ where $(\hat{t}_{\pi, \ell}(x) = \bbE[\ell(\theta)|X = x])$.  
In practice, we do not have access to the true distribution $\pi(\cdot)$. 
The goal now is to output an approximation $\hat T$ of the Bayes estimator. 
We evaluate the quality of such an estimator via the \textit{regret}, which captures the difference between the mean squared error of the estimator $\hat{T}$ and the Bayes estimator $\Tpiell$. 

\begin{definition}[mmse of a functional $\ell$]\label{def:mmse}
	Let the mmse of a prior distribution $\pi$ and a functional $\ell$ be the expected squared error of the Bayes estimator, i.e.
	\begin{align*}
		\mmse_{\ell}(\pi)\overset{\Delta}{=}
		\min_{\hat{t}} \E_{\pi}\left[(\hat{t}(X)-\ell(\theta))^2\right]
		=\E_{\pi}[(\hat{t}_{\pi, \ell}(X)-\ell(\theta))^2].
	\end{align*}
\end{definition}

Next, we consider the following definition of total regret. 
\begin{definition}[Total regret]\label{def:totregret}
	The total regret of an estimator $\hat{T}:\mathbb{R}^n\to\mathbb{R}^n$ on a prior $\pi$ and functional $\ell$ is defined as 
	\[\mathsf{TotRegret}_{\pi, \ell, n}(\hat{T}) = \mathbb{E}\left[\sum_{j=1}^n (\hat{T}_j(X_1^n) -
	\ell(\theta_j))^2\right] - n\cdot \mathsf{mmse}_{\ell}(\pi)\]
	Correspondingly, the minimax optimal regret of a class $\mathcal{G}$ of priors is defined by optimizing $\hat T$'s performance uniformly over all priors in $\mathcal{G}$:
	\[
	\mathsf{TotRegret}_{\ell, n}(\mathcal{G}) = \inf_{\hat{T}} \sup_{\pi\in\mathcal{G}} \mathsf{TotRegret}_{\pi, \ell, n}(\hat{T})
	\]
\end{definition}

Below, whenever $\ell$ denotes the $k$-th moment functional, i.e. $\ell(\theta) = \theta^k$, we will instead use the subscript $k$ in place of $\ell$ in both the mmse and regret notation, e.g. $\mathsf{TotRegret}_{k,n}(\mathcal{G})$ etc.

Next, we define the classes of priors over which we will bound the regret. 
These are priors satisfying certain light-tailedness assumptions (in particular, priors of bounded support). 
Regret bounds for the mean estimation problem have generally been established under those settings 
\cite{BGR13, polyanskiy2021sharp, JPW24, JPTW23}. 
We recall that the class of bounded distributions is defined as 
$\calP([a,b]) = \{\pi: \pi([a,b])=1\}$; 
the class of subgaussian distributions $\SubG(s) = \{\pi: \pi[|\theta| \ge t]\le 2e^{-t^2/2s} \,\,\, \forall t>0\}$; 
and the class of subexponential distributions (defined only on $\mathbb{R}_+$) as $\subexpo(s) = \{\pi: \pi[\theta \ge t]\le 2e^{-t/s}\,\,\, \forall t>0\}$. 

We now describe the models which we will be working on. These are the normal means and Poisson models. 
We recall that the normal means model is where $f(\cdot; \theta)\triangleq \calN(\theta, 1)$, 
and the Poisson model is where $f(\cdot; \theta)\triangleq \Poi(\theta)$. 
We will continue the discussion of the Poisson model in \prettyref{sec:poisson_model}. 

\subsection{Poisson model}\label{sec:poisson_model}
We now describe certain properties of the Poisson model that we can manipulate. 
Recall that the marginal distribution $f_{\pi} = \text{Poi}\circ \pi$ is given by 
\begin{align}\label{eq1}f_{\pi}(x) = \int e^{-\theta}\frac{\theta^x}{x!}d\pi(\theta).\end{align}
Thanks to the relatively straightforward form of Tweedie's formula, the standard estimators can also be formulated rather easily. 

\subsubsection{Tweedie's formula}
We start the discussion by outlining Tweedie's formula for $\ell(\theta)=\theta^k$. 
(c.f. \cite[Section 2]{nichols1972empirical}). 
\begin{equation}\label{eq:bayes_polyk}
	\tpiell(x) = \E[\theta^k|X=x] 
	= \frac{\int\theta^k \left(e^{-\theta}\frac{\theta^x}{x!}\right)d\pi(\theta)}{f_\pi(x)}
	= \frac{(x+1)_{k}\int e^{-\theta}\frac{\theta^{x+k}}{(x+k)!}d\pi(\theta)}{f_\pi(x)}
	= \frac{(x+1)_{k} f_\pi(x+k)}{f_\pi(x)}
\end{equation}
where: 
\begin{equation}\label{eq:pochhammer}
	(x)_n = x(x + 1) \cdots (x + n - 1) = \frac{\Gamma(x + n)}{\Gamma(x)}
\end{equation}
denotes the Pochhammer symbol. 

\subsubsection{Estimators}\label{sec:poisson-estimators}
Next, we introduce some of the estimators for the Poisson model with the goal of approximating \prettyref{eq:bayes_polyk}, as adapted from the mean estimation cases. 

\paragraph{Method of moments estimator.} 
We first introduce a non-empirical Bayes baseline, based on the following observation: 
for $X\sim \Poi(\theta)$, we have 
\[
\EE[X(X - 1)\cdots (X - k + 1)] = \theta^k,
\]
thus suggesting a natural (unbiased) method of moments estimator \[\hat{t}_{\mathsf{MOM}, k}(X) = X(X - 1)\cdots (X - k + 1) =: (X-k+1)_k.\]
Note that for $k = 1$ this also coincides with the (frequentist's) maximum likelihood estimator. 
By Lehmann–Scheffé theorem, this is the unique minimum variance unbiased estimator (UMVUE) for this Poisson moment estimation problem
\cite[(3.17)]{lehmann1998theory}. 

\paragraph{Modified Robbins estimator ($f$-modeling).} This estimator is naturally extended from the Robbins estimator \cite{Rob51, Rob56}. That is, we estimate $f_\pi$ using the empirical distribution of the observations $X_1, \cdots, X_n$'s and plug this estimate into Tweedie's formula~\eqref{eq:bayes_polyk}. This way, we obtain estimator 
\begin{align}\label{eq:robbins}
	\Trobk(X_1, \ldots, X_n) =\left(\frac{(X_i+1)_kN_n(X_i+k)}{N_n(X_i)}\right)_{i = 1, \cdots, n}, 
	\qquad 
	N_n(x) = \sum_{j=1}^n \indc{X_j = x}.
\end{align}

\paragraph{Minimum distance method ($g$-modeling).} Another approach to EB problems is to first derive an estimator $\hat \pi$ of the unknown latent $\pi$ and then compute the Bayes estimator $\hat T_{\hat \pi,\ell}$. One of the surprises of EB is that despite $\hat \pi$ converging very slowly to $\pi$ \cite[Theorem 6.2]{miao2024fisher} (due to the nonparametric nature of the class $\mathcal{G}$), the large mismatch between $\hat \pi$ and $\pi$ is almost irrelevant for the computation of the Bayes conditional expectation. In this paper, we consider estimators $\hat{\pi}$ that minimize $d(p_n^{\mathsf{emp}} || f_{\hat{\pi}})$, where $p_n^{\mathsf{emp}}$ denotes the empirical distribution of sample $X_1^n$, and $d$ is a divergence. That is, we define $\hat{\pi}_d$ as 
\begin{equation}\label{eq:mindist-def}
	\hat{\pi}_d = \argmin_{\hat{\pi}} d(p_n^{\mathsf{emp}}|| f_{\hat{\pi}})
\end{equation}
The most important special case is $d = D_{KL}$ (the Kullback-Leibler divergence \cite[(2.1)]{polyanskiy2025information}), in which case $\hat{\pi}$ becomes the nonparametric maximum likelihood estimator (NPMLE) \cite{KW56} and we give it a special name: 
\begin{equation}\label{eq:npmle-def}
	\hat{\pi}_{\mathsf{NPMLE}} = \argmax_{Q\in\mathbb{R}^+} \prod_{i=1}^n f_Q(x_i)
\end{equation}
Then we use the following plugin estimator: 
\[\hat T \triangleq \hat{T}_{\hat{\pi}} = (\hat{t}_{\hat{\pi}}(X_i))_{i = 1, \cdots, n}
\qquad 
\hat{t}_{\hat{\pi}}(x) = \frac{(x+1)_kf_{\hat{\pi}}(x + k)}{f_{\hat\pi}(x)}.
\]
Recall that this approach (based on estimating the distribution of $\theta$) is called $g$-modeling~\cite{Efr14}. 

\paragraph{ERM on monotone functions.} 
The ERM is an idea in statistical learning theory first introduced in \cite{Vap91}, where an optimal hypothesis is learned by finding the hypothesis with the smallest loss over the empirical data. 
In the mean estimation setting, this is first done by \cite{BZ22} for normal means, and then by \cite{JPTW23} for Poisson. We now extend the latter algorithm to estimating $\theta^k$ for the Poisson model. 

This is motivated by the following: each component $\tpik$ of the Bayes estimator $\Tpik$ is the unique minimizer of the following objective function: 
\begin{align}\label{eq:erm_obj_pre}
	\tpik = \argmin_{\hat{t}}\E\left[\left(\hat{t}(X)-\ell(\theta)\right)^2\right] = \argmin_{\hat{t}}\E\left[\hat{t}(X)^2-2\ell(\theta)\hat{t}(X)\right].
\end{align}
where the last equality holds because $\E[\ell(\theta)^2]$ does not depend on the estimator $\hat{T}$. 
For a given $\hat{T}$ we denote $\hat{t}_i(\theta; x) := \mathbb{E}[\hat{T}_i(X_1^n) \mid X_i = x, \theta_i = \theta]$ 
and $\hat{t}_i(x) = \mathbb{E}[\hat{T}_i(X_1^n) \mid X_i = x]$. 
Then the cross term can be further decomposed into: 
\begin{align}\label{eq:erm_obj_theta}
	\E\left[\theta_i^k \hat{T}_i(X_1^n)\right] &= \E_{\theta_i\sim\pi}[\theta_i^k\E_{X_i\sim \Poi(\theta_i)}[\hat{T}_i(X_1^n)]]
	\nonumber\\
	&=\E_{\theta\sim\pi}[\theta^k\E_{X\sim \Poi(\theta)}[\hat{t}_i(\theta; x)]]
	\nonumber\\
	&=\int\sum_{x=0}^\infty e^{-\theta}\frac{\theta^x}{x!} 
	\hat{t}_i(\theta; x)\theta^kd\pi(\theta) \nonumber\\
	&= \int\sum_{x=0}^\infty e^{-\theta}\frac{\theta^{x+k}}{(x+k)!}(x+1)_k \hat{t}_i(\theta; x) d\pi(\theta) \nonumber\\
	&\stepa{=} \int\sum_{x=0}^\infty e^{-\theta}\frac{\theta^x}{x!}(x-k+1)_k 
	\hat{t}_i(\theta; x - k) d\pi(\theta) \nonumber\\
	&= \E[(X - k + 1)_k \hat{t}_i(X-k)].
\end{align}
where (a) follows from index shift $x\to x -k$, 
and including terms where $x=0, 1, \cdots, k - 1$ as $(x-k+1)_k=0$ for $x=0, 1, \cdots, k - 1$. Substituting \prettyref{eq:erm_obj_theta} into \prettyref{eq:erm_obj_pre}, we obtain 
\begin{align}\label{eq:erm_obj}
	\tpik = 
	\argmin_{\hat{t}} R_{\pi,k}(\hat{t})
	\qquad 
	R_{\pi, k}(\hat{t}) \triangleq \E_{X\sim \text{Poi}\circ \pi}\left[\hat{t}(X)^2-2(X - k + 1)_k \hat{t}(X - k)\right].
\end{align}
and in a similar manner, we may define the empirical version $\hat{R}_k := \hat{R}_{X_1,\cdots, X_n, k}$ of $R$ defined in 
\prettyref{eq:erm_obj} as follows: 
\begin{equation}\label{eq:erm_obj_emp}
	\hat{R}_k(\hat{t}) \triangleq \hat{\E}\left[\hat{t}(X)^2-2(X - k + 1)_k \hat{t}(X - k)\right]
	=\frac{1}{n} \sum_{i=1}^n \left[\hat{t}(X_i)^2-2(X_i - k + 1)_k \hat{t}(X_i - k)\right]
\end{equation}

Like \cite{JPTW23}, this motivates finding an estimator $\hat{T}$ that minimizes the empirical version of the objective function in \prettyref{eq:erm_obj} over some function class $\calF$. 
Later in \prettyref{lmm:bayes-monotone} we show that the Bayes estimator $\hat{t}_{\pi, k}$ is monotone. 
	This allows us to consider the following monotone class $\mathcal{F}_{\uparrow}$ defined as follows: 
\begin{equation}\label{eq:monotone}
	\mathcal{F}_{\uparrow}= \{\hat{t}:\bbZ_+\to\bbR_+: \hat{t}(x) \le \hat{t}(x + 1), \forall x\in\bbZ_+\}
\end{equation}
We then define the ERM estimator as follows: 
\begin{align}\label{eq:erm_obj_mon}
	\Termk(X_1^n) = (\hat{t}(X_1), \cdots, \hat{t}(X_n)), 
	\qquad 
	\hat{t} \in \argmin_{t\in \mathcal{F}_{\uparrow}} \hat{R}_{k}(t)
\end{align}

The ERM can be viewed as a compromise of the accuracy-efficiency tradeoff between the $f$- and $g$-modeling. 
The $f$-modeling technique suffers from multiple numerical instability problems (e.g., when $x$ is large and the counts are small, small changes in counts greatly affect the estimated values)\cite{Mar68}. Furthermore, the Robbins' estimator may not be monotone, which is often desired \cite{HS83}. 
The $g$-modeling technique reconciles this by preserving the Bayesian structure (and hence desired monotonicity and stability of the Bayes estimator). 
For this reason, $g$-modeling has historically attracted a lot of attention, see \cite{koenker2014convex, koenker2024empirical, gu2022nonparametric, gu2023invidious} for recent examples. 

Despite all of this, $g$-modeling has a major computational issue: solving the optimization problem in \prettyref{eq:mindist-def} is computationally challenging even in the one-dimensional setting considered here (and essentially impossible for higher-dimensional settings as runtime becomes exponential in dimension \cite{Lin83}).
The asymptotic computational complexity of NPMLE remains poorly understood, and such studies have only recently begun \cite{polyanskiy2025nonparametric}. 

The ERM method preserves regularity by imposing the monotonicity constraint, while escaping the computationally intensive step of prior estimation. 
This is partly inspired by some works (c.f. \cite{VH77}) that impose monotonicity to $f$-modeling estimators without increasing regret.

\subsection{Regret bound (Poisson)}\label{sec:poisson_bound}

We are now ready to explicitly state our main results. Subsequent sections contain the proofs.

\begin{theorem}[Poisson problem: $\ell$ polynomial.]\label{thm:poisson_bound}
	Consider the Poisson model $f(\cdot; \theta) = \Poi(\theta), \theta\in\bbR^+$. 
	\begin{itemize}
		\item There exist  $c_1 = c_1(h, k) > 0, c_2 = c_2(h, k) > 0$ and $n_0 = n_0(k)$ such that for all $n\ge n_0$, 
		$$
		c_1\left(\left(\frac{\log n}{\log\log n}\right)^{k+1}\right)
		\le 
		\TotRegret_{k,n}(\calP([0,h])) \le c_2\left(\left(\frac{\log n}{\log\log n}\right)^{k+1}\right)$$
		
		\item There exist constants $c_3 = c_3(s, k) > 0, c_4 = c_4(s, k) > 0$ and $n_0 = n_0(k)$ such that for all $n\ge n_0$, 
		$$
		c_3\left((\log n)^{2k+1}\right)
		\le \TotRegret_{k, n}(\subexpo(s)) 
		\le c_4\left((\log n)^{2k+1}\right)
		$$
	\end{itemize}
	
	Furthermore, the upper bounds are achieved by any of the following estimators: $\Trobk$ and estimators based on $\Termk$ and $\hat{T}_{\hat{\pi}}$ where $\hat{\pi}$ is estimated by minimum distance estimators satisfying \cite[Assumptions 1, 2]{JPW24}. 
\end{theorem}
The upper bound in \prettyref{thm:poisson_bound} can be established by analyzing just one estimator. 
Instead, we do so for three estimators to show that they are minimax optimal for a more general class of estimation tasks (i.e. $\ell(\theta)=\theta^k$), and to show that the techniques developed in prior works can be generalized to our setting. In addition, we illustrate in \prettyref{fig:intro_unif10} how these empirical Bayes estimators achieve diminishing per-coordinate regret as sample size $n$ increases,  
and at which $n$ the performance gap of estimators that correspond to $(\mathbb{E}[\theta | X])^k$ and $\mathbb{E}[\theta^k | X]$ becomes visible.

\subsection{Regret bound (normal means)}\label{sec:normal_bound}
For the normal means model, in the mean estimation problem ($k = 1$) the recent work in \cite{chen2026sharp} establishes an upper bound of $O((\frac{\log n}{\log \log n})^{2})$, which matches the lower bound of $\Omega((\frac{\log n}{\log \log n})^2)$ given by \cite{polyanskiy2021sharp}.
	The previous state-of-the-art regret of $O(\log^5 n)$ was attained by \cite{JZ09} based on the minimum distance method. 
	We now establish minimax regret bounds for polynomial estimation, that are tight up to constant and polylogarithmic factors for bounded and subgaussian priors, respectively.

\begin{theorem}[Normal means problem]\label{thm:gsn_bounds}
	Consider the normal means model: $P_{\theta} = \calN(\theta, 1), \theta\in\bbR$. 
	\begin{itemize}
		\item There exist constants $c_1 = c_1(h, k) > 0$, $c_2 = c_2(h, k) > 0$ and $n_0 = n_0(k)$, such that for all $n \ge n_0$,
			$$c_1\left(\left(\frac{\log n}{\log\log n}\right)^{k+1}\right)\le \TotRegret_{k, n}(\calP([-h,h])) \le c_2\left(\left(\frac{\log n}{\log\log n}\right)^{k+1}\right)$$
		
		\item There exist constants $c_1 = c_1(s, k) > 0$, $c_2 = c_2(s, k) > 0$ and $n_0 = n_0(k)$ such that for all $n\ge n_0$, 
			$$c_1\left((\log n)^{k+1}\right)\le \TotRegret_{k, n}(\SubG(s)) \le c_2((\log n)^{k+2}(\log \log n)^{2k})$$
	\end{itemize}
\end{theorem}

\subsection{Generalization to smooth function estimation}\label{sec:main-result-smooth}
	Next, we generalize \prettyref{thm:poisson_bound} and \prettyref{thm:gsn_bounds} to obtain an upper bound over a more general class of smooth functions.
Here, we remind ourselves that the class $\text{Lip}_{\alpha}(h, L)$ of H\"older continuous functions with exponent $\alpha$ ($0 < \alpha\le 1$) and constant $L$ are functions that obey 
\begin{equation}\label{eq:lipchitz}
	|\ell(\theta_1) - \ell(\theta_2)|\le L|\theta_1 - \theta_2|^{\alpha}, \quad \forall \theta_1, \theta_2\in [-h, h]
\end{equation}
More generally, we consider a higher degree of smoothness of a function. For a real number $\beta > 0$, we decompose uniquely into $\beta = p+\alpha$, with $p\in\mathbb{Z}_{+}, 0<\alpha \le 1$. 
Then we denote by $\mathcal{F}_{\beta}(h, L)$ the following class of smooth functions: 
\begin{equation}\label{eq:beta-smooth}
	\mathcal{F}_{\beta}(h, L) = \{\ell: 
	\ell^{(p)}\in \text{Lip}_{\alpha}(h, L) 
	\}
	\qquad 
	\beta=p+\alpha
\end{equation}

We also consider functions analytic in a neighborhood of the origin, which is a more restricted form of the class of smooth functions. We say that $\ell$ is analytic around $\theta = 0$ if 
$\ell$ can be written in the form of $\ell(\theta) = \sum_{m = 0}^{\infty} c_m\theta^m$, 
and the radius $R$ of convergence is defined as 
\begin{equation*}
	R = \sup_{r\ge 0}\{r: \sum_{m = 0}^{\infty} |c_m|r^m \text{ converges}\}.
\end{equation*}

\begin{theorem}[Poisson / normal means problem: smooth and analytic functionals.]\label{thm:eb_cts}
	Let $h > 0$ be given, and denote the interval $I(h) := [0, h]$ in the case of Poisson model, and $[-h, h]$ in the case of normal means model.
	Consider a function $\ell$ continuous on $I(h)$. 
	Let $M\triangleq\sup_{\theta\in I(h)}|\ell(\theta)|$. 
	Then the minimax regret of $\mathcal{P}(I(h))$ has the following bounds on both the Poisson and normal means model:
	\begin{itemize}
		\item 
		Let $\beta > 0$ be a real number. 
		
			Then there exist $C_1\triangleq C_1(L, M, h, \beta)$ and $C_2\triangleq C_2(L, M, h, \beta)$ such that 
			\[
			C_1n\left(\frac{\log \log n}{\log n}\right)^{2\beta}
			\sup_{\substack{\ell\in \mathcal{F}_{\beta}(h, L)\\ \sup_{\theta\in I(h)} |\ell(\theta)|\le M}}\TotRegret_{\ell, n}(\mathcal{P}(I(h)))\le C_2n\left(\frac{\log \log n}{\log n}\right)^{2\beta}.
			\]

		\item 
		Suppose that $\ell$ is analytic around $\theta = 0$ with radius of convergence $R > h$. 
		Then there exists a constant $C_1\triangleq C_1(R, h, M)$ and $C_2\triangleq C_2(R, h, M) > 0$ such that 
		\[
		\TotRegret_{\ell, n}(\mathcal{P}(I(h)))\le C_1n^{1-\frac{C_2}{\log \log n}}.
		\]
		
		\item Finally, suppose that $\ell$ is analytic with coefficient $c_m$ of $x^m$ satisfying $|c_m|\le (D_1m)^{-D_2m}$ for some $D_2 > 0$. Then there exist constants $C_1, C_2$ depending on $D_1, D_2, h$ such that 
		\[
		\TotRegret_{\ell, n}(\mathcal{P}(I(h)))\le C_1 n^{1-\min\{2D_2, 1\} + \frac{C_2}{\log \log n}}.
		\]
	\end{itemize}
\end{theorem}

	\begin{remark}[Connection to deconvolution bound]
		In the case of 1-Lipchitz functions, the minimax total regret bound is $\Theta(n\left(\frac{\log \log n}{\log n}\right)^2)$. 
		For the lower bound, the proof idea uses  
		\prettyref{prop:reduction} and establishes a stronger result:
		\[\sup_{\ell\in Lip(1)}\inf_{\hat{L}}\sup_{\pi\in \mathcal{P}([0, h])} \mathbb{E}\left[\left|\hat{L}(X_1^n) - \mathbb{E}_{\pi}[\ell(\theta)]\right|\right] 
		= \Omega_h\left(\frac{\log \log n}{\log n}\right).\] 
		In particular, for any estimate $\hat{\pi}$ of $\pi$, one has 
		\begin{flalign*}
			\sup_{\pi\in \mathcal{P}([0, h])}\mathbb{E}[W_1(\hat{\pi}_d, \pi)]
			&=\sup_{\pi\in \mathcal{P}([0, h])}\mathbb{E}[\sup_{\ell\in\text{Lip}(1)}\mathbb{E}_{\hat{\pi}}[\ell(\theta)] - \mathbb{E}_{\pi}[\ell(\theta)]]
			\\&\ge \sup_{\ell\in\text{Lip}(1)}\sup_{\pi\in \mathcal{P}([0, h])}\mathbb{E}[\mathbb{E}_{\hat{\pi}}[\ell(\theta)] - \mathbb{E}_{\pi}[\ell(\theta)]]
			\ge \Omega_h\left(\frac{\log \log n}{\log n}\right)
		\end{flalign*}
		and thus recovers the deconvolution lower bound in \cite[Theorem 6.2(a)]{miao2024fisher}.
	\end{remark}

\begin{remark}
	We give some examples of functions that satisfy some of the smoothness classes above. 
	Functions like $\ell(\theta) := (1 + \theta)^{\alpha}$ for any $\alpha\in \mathbb{R}$ have a power series expansion for $|\theta| < 1$, 
	hence within $[0, h]$ where $h < 1$ we can achieve a minimax total regret of $O(n^{1 - \frac{C}{\log \log n}})$. 
	One family of functions $\ell$ with coefficients falling into the  ``rapidly decaying'' regime is where 
	there exists a constant $h'$ such that the $k$-th derivative $\ell^{(k)}$ satisfies $\sup_{[0, h]}|\ell^{(k)}(\theta)|\le (h')^k$. 
	This is in particular the case for $\ell(\theta)=e^{\theta}$ and the trigonometric polynomial 
	$T(\theta)\triangleq a_0 + \sum_{m = 1}^N a_m \cos(m\theta) + b_m\sin (m\theta)$. Here we may choose any $D_2 < 1$, and therefore achieve a total regret of $n^{o(1)}$. 
\end{remark}

\subsection{Related work}

\textbf{Mean estimation in empirical Bayes setting}. 
There has been a wealth of literature on  bounding regret in the mean estimation problem in the empirical Bayes setting, 
particularly in the normal means and Poisson means models. 
In normal means, estimators include the NPMLE estimator \footnote{NPMLE's definition is given in \prettyref{eq:npmle-def}.} ($g$-modeling) and the GEB estimator based on approximation of the posterior density $f_{\pi}$ and its derivative $f'_{\pi}$ using a kernel ($f$-modeling). 
\cite{zhang1997empirical} considers the \emph{truncated} Bayes estimator, and shows that the GEB estimator achieves an extra total MSE of $O(\log^{3/2} n)$ compared to the truncated Bayes estimator.
\cite[Example 1]{LGL05} establishes an upper bound on total regret of $O(\log^8 n)$ using polynomial kernels. 
For $g$-modeling, regret bounds are typically shown using results on posterior density estimation 
\cite{ghosal2001entropies, zhang_generalized_2009}, and \cite{JZ09} establishes an upper bound on total regret of $O(\log^5 n)$ using the NPMLE estimator, which was the best known upper bound  until the very recent improvement by \cite{chen2026sharp} to $O((\frac{\log n}{\log\log n})^{2})$. This matches exactly the lower bound of $\Omega((\frac{\log n}{\log \log n})^2)$ in \cite[Theorem 1]{polyanskiy2021sharp}. 

In the Poisson model, 
both the $f$- and $g$-modeling methods have been shown to achieve minimax total regret of $O((\frac{\log n}{\log \log n})^2)$ and $O((\log n)^3)$ in compact and subexponential priors, respectively \cite{BGR13, JPW24}. 
For lower bounds on total regret, 
\cite[p.891]{singh1979empirical} conjectured that bounded minimax total regret in any exponential family is impossible. 
\cite[Theorem 2.2]{LGL05} shows a lower bound of $\Omega(1)$ for the normal means using two-point methods. 
\cite[Theorems 1, 2]{polyanskiy2021sharp} made a substantial improvement using Assouad’s lemma, which resolves the conjecture in \cite{singh1979empirical} for normal means and Poisson models.  
In particular, \cite{polyanskiy2021sharp} shows that the aforementioned bounds for Poisson are tight up to a constant factor, 
while for normal means there is a polylog factor between the best known upper and lower bounds. 
In view of the stability-efficiency tradeoff presented in the $f$- and $g$-modeling methods, 
\cite{BZ22} proposed an estimator based on ERM \footnote{ERM for Poisson model is defined in \prettyref{eq:erm_obj_emp}; for the normal means model it is a minimizer of $\frac 1n\sum_{i=1}^n [\hat{t}(X_i)^2 - 2X_i\hat{t}(X_i) + 2\hat{t}'(X_i)]$} on a restricted class of estimators for normal means model to add regularity to the estimator. However, they only achieved a slow rate of $O(\sqrt{n}\mathsf{polylog}(n))$ total regret. 
This was later improved to $O(\log^6 n)$ in \cite[Theorem 7]{ghosh2025stein} which optimizes the Stein's unbiased risk estimate (SURE) among all the possible Bayes estimators of normal means. 
For the Poisson model, \cite{JPTW23} derived a fast algorithm implementing ERM on the monotone class, and showed that it achieves asymptotically optimal regret (up to a multiplicative constant) for two classes of priors (bounded and subexponential). In order to show the fast rate, ERM analysis required adaptation of a learning theory method known as offset Rademacher complexity \cite{liang2015learning}. In a different setting, \cite{jaffe2025constrained} studies the problem of estimating the Bayes estimator with the constraints of matching either the moments or distributions of the latent estimator $\theta$, in an attempt to reverse the shrinkage phenomenon of typical EB estimators. This and earlier works \cite{garcia2024new} establish connection between EB and optimal transport.  

Finally, in the case of higher dimensional priors, for the Poisson model, the upper bound of $O(\log^{O(d)}(n))$ total regret has been established by \cite{JPW24, JPTW23} for a generalization of the problem, in which each $\theta$ is allowed to take values in $\R^d$. For the normal means model (with identity covariance matrix), the upper bound of $O(\log^{O(d)}(n))$ total regret is established in \cite[Corollary 3.2]{saha2020nonparametric}, 
while for density estimation a tight bound of $\Theta(\frac{(\log n)^d}{n})$ squared $L^2$ distance is shown in \cite[Theorem 2.1]{kim2022minimax}. To our knowledge, there is no known lower bound for regret estimation in both the Poisson and normal means models for $d > 1$. 

\textbf{Function estimation other than the mean}. 
We now discuss some works outside of the mean estimation problem. For the higher moment estimation problem, 
\cite{nichols1972empirical} formulated an estimator based on Tweedie's formula for a certain class of exponential families, to which Poisson belongs. 
For the variance estimation problem, 
\cite{wang2008effect} established a minimax bound for the variance estimation in the heteroscedastic nonparametric regression model. \cite{xie2012sure} established a consistency result of the SURE estimator on the variance of the normal prior in the heteroscedastic normal means model. 
Several works studied the problem of variance estimation in the normal means model (i.e. $X\sim N(0, \sigma^2)$), 
e.g. \cite{wu2020optimal} when priors have bounded number of atoms, and \cite{subhodh2024variance} establishes a tight bound for the variance estimation of the normal means model in a nonparametric setting. 
An example of a more general posterior inference in the empirical Bayes is the quantile estimation problem, done via maximization of the likelihood function under the asymmetric Laplace density, as formalized in \cite{yu2001bayesian}. 
\cite{yang2012bayesian} showed that the resultant posterior is asymptotically normal. 
Estimation of confidence intervals in the empirical Bayes setting is done in \cite{laird1987empirical} via bootstrap sampling, and \cite{carlin1990approaches} via bias correction. 
We stress that for many of these estimation tasks (particularly the higher moment estimation problem), the minimax rate in the nonparametric setting is not nearly as well-studied as the mean estimation problem. 

\textbf{Compound decision theory and empirical Bayes}. Although we focus on the empirical Bayes problems here, it would be appropriate to recall that those problems (and our results) go hand in hand with compound decision theory since the emergence of both in~\cite{Rob51}. This association is 
based on the premise that the risk of a statistical procedure can be improved by allowing estimation of each individual parameter $\theta_i$ to depend on the entire sequence. The stark demonstration of this effect is in the Stein's problem on the mean estimation of the normal means model \cite{stein1956inadmissibility}, 
where the James-Stein estimator \cite{james1961estimation} demonstrates a strictly lower risk than the maximum likelihood estimator. 
In the first introduction of the concept of empirical Bayes, 
\cite{Rob56} proposed the Robbins estimator to approximate the Bayes estimator of the Poisson model in the mean estimation problem by approximating the mixture density of the samples with the empirical distribution. This informal association between two problems was finally quantified by \cite{greenshtein2009asymptotic}, who demonstrated that under certain conditions the total regret in the compound decision problem (against a permutation oracle) is within $O(1)$ of the regret in the equivalent empirical Bayes problem. This result was generalized and strengthened in~\cite{han2024approximate} which generalizes \cite{greenshtein2009asymptotic} to more general loss functions and models by showing a $\chi^2$ bound between permutation and product mixtures that is independent of sample size. 
\cite{han2025besting} improves the dependence on support size from exponential to linear compared to \cite{greenshtein2009asymptotic, han2024approximate}. 
Finally, the follow-up work by \cite{liang2025sharp} unifies the techniques from both works and improves upon them. 

In our context of functional estimation, we comment on some relations between empirical Bayes and compound decision setting in \prettyref{app:compound}, where in particular we show that the lower bound on the former implies a lower bound on the latter. Determination of the exact regret (up to a constant factor) remains open, however.

\subsection{Notation}
In the course of this work, we denote by $n$ the sample size, by $k$ the exponent of the polynomial $\theta^k$, 
by $f_{\pi}$ the mixture density induced by prior $\pi$ under either the Poisson or normal-means channel. 
The Bayes estimator of $\ell(\theta)$ with prior $\pi$ is denoted $\tpiell$, 
which is also denoted by $\tpik$ in the case where $\ell(\theta) := \theta^k$. 
We denote $\hat{T}$ as an $n$-letter to $n$-letter estimator 
(i.e. $\mathbb{R}^n\to\mathbb{R}^n$ for normal means model, and $\mathbb{Z}_+^n\to\mathbb{R}_+^n$ for Poisson model) 
and $\hat{t}$ as a scalar to scalar estimator (i.e. $\mathbb{R}\to\mathbb{R}$ for normal means model, and $\mathbb{Z}_+\to\mathbb{R}_+$ for Poisson model), respectively. 
Estimators like $\thetarobell, \Trobk, \Termell,$ and $\Termk$ are also defined similarly for Robbins and ERM estimator, 
respectively. 

\section{Preliminaries}\label{sec:preliminaries}
\subsection{Relationship between total regret and individual regret}
In addition to total regret, we define individual regret by treating $X_1, \cdots, X_{n - 1}$ as training data and estimating $\ell(\theta_n)$ given the unseen observation $X_n$. 
\begin{definition}[Individual Regret]\label{def:regret}
	The individual regret of an estimator $\hat{T}: \mathbb{R}^n\to\mathbb{R}^n$ on a prior $\pi$ and functional $\ell$ is defined as  
	\[\mathsf{Regret}_{\pi, \ell, n}(\hat{T}) = \mathbb{E}\left[(\hat{T}_n(X_1^n) - \ell(\theta_n))^2\right] - \mathsf{mmse}_{\ell}(\pi)\]
	
	The individual regret of a collection $\mathcal{G}$ of priors is the minimax regret achieved by any estimator over the class $\mathcal{G}$
	\[
	\mathsf{Regret}_{\ell, n}(\mathcal{G}) = \inf_{\hat{T}} \sup_{\pi\in\mathcal{G}} \mathsf{Regret}_{\pi, \ell, n}(\hat{T})
	\]
\end{definition}
Similar to \cite[Lemma 5]{polyanskiy2021sharp} for mean estimation, 
we show that the minimax total regret over a prior class is exactly $n$ times the minimax individual regret. 

\begin{lemma}\label{lmm:totregretind}
	For any class $\mathcal{G}$, sample size $n$ and functional $\ell$, 
	we have $n\cdot \Regret_{\ell, n}(\mathcal{G}) = \TotRegret_{\ell, n}(\mathcal{G})$. 
\end{lemma}
\begin{proof}
	The proof follows closely the proof of \cite[Lemma 5]{polyanskiy2021sharp}, broken down into two parts. 
	
	\paragraph{$\TotRegret_{\ell, n}(\mathcal{G})\le n\Regret_{\ell, n}(\mathcal{G})$.}
	Given any estimator $\hat{T}$, we define $\tilde{T}$ with $\tilde{T}(X_1^n) = (\tilde{T}_1(X_1^n), \cdots, \tilde{T}_n(X_1^n))$, 
	where $\tilde{T}_i(X_1^n) = \hat{T}_n(X_{\backslash i}, X_i)$, 
	where $X_{\backslash i} = (X_1, \cdots, X_{i-1}, X_{i+1}, \cdots, X_n)$. 
	Then 
	\begin{flalign*}
		\TotRegret_{\pi, \ell, n}(\tilde{T})
		&=\sum_{i=1}^n \mathbb{E}[(\tilde{T}_i(X_1^n) - \ell(\theta_i))^2] - n\cdot \mathsf{mmse}_{\ell}(\pi)
		\\&=\sum_{i=1}^n \mathbb{E}[(\hat{T}_n(X_{\backslash i}, X_i) - \ell(\theta_i))^2]- n\cdot \mathsf{mmse}_{\ell}(\pi)
		\\&=n\mathbb{E}[(\hat{T}_n(X_1^n) - \ell(\theta_n))^2] - n\cdot \mathsf{mmse}_{\ell}(\pi)
		=n\Regret_{\pi, \ell, n}(\hat{T})
	\end{flalign*}
	The claim then follows by choosing $\hat{T}$ to approach the minimax individual regret. 
	
	\paragraph{$\Regret_{\ell, n}(\mathcal{G})\le \frac 1n\TotRegret_{\ell, n}(\mathcal{G})$.} 
	For the converse direction, given an estimator $\hat{T}$, we construct the following randomized estimator 
	\[
	\tilde{T}_n(X_1^n) 
	:= \hat{T}_i(X_1^{i-1}, X_n, X_{i+1}^{n-1}, X_i), \text{  w.p. } \frac 1n,
	i = 1, \cdots, n. 
	\]
	Then for each prior $\pi$ we have 
	\[
	\Regret_{\pi, \ell, n}(\tilde{T})
	=\mathbb{E}[(\tilde{T}_n(X_1^n)  - \ell(\theta_n))^2] - \mathsf{mmse}_{\ell}(\pi)
	\]\[
	=\frac 1n \sum_{i=1}^n \mathbb{E}[\hat{T}_i(X_1^{i-1}, X_n, X_{i+1}^{n-1}, X_i) - \ell(\theta_n))^2] - \mathsf{mmse}_{\ell}(\pi)
	=\frac 1n\TotRegret_{\pi, \ell, n}(\hat{T})
	\]
	and therefore $\Regret_{\ell, n}(\mathcal{G})\le \frac 1n\TotRegret_{\ell, n}(\mathcal{G})$. 
\end{proof}
We note that the proof also establishes that an estimator achieving a minimax total regret of $\alpha$ (for some $\alpha > 0$) can be adapted to produce an estimator achieving a minimax individual regret of $\frac{\alpha}{n}$, and vice versa. 

\subsection{Truncation of priors}
Here, we show that for both the upper and lower bound arguments, 
we may reduce regret bounds of priors with uniform tail bounds 
to those with compact support, 
effectively generalizing \cite[Lemma 6]{polyanskiy2021sharp}. 

\begin{lemma}\label{lmm:6_generalized}
	Given $h>0$, let $\calG$ be a collection of priors on $\R$ such that $\sup_{\pi\in\calG}\PP(|\theta| > h)\le\epsilon\le\frac12$ for some $\epsilon$ and $\sup_{\pi\in\calG}\E_\pi[\theta^{4k}]\le M$. Then 
	\begin{align}\label{eq:lemma_6_generalized}
		\inf_{\hat{T}}\sup_{\pi\in\calP([-h,h])}\Regret_{\pi,k, n}(\hat{T}) \ge \inf_{\hat{T}}\sup_{\pi\in\calG}\Regret_{\pi,k, n}(\hat{T})-6\sqrt{(M+h^{4k})n\epsilon}.
	\end{align}
\end{lemma}

The key of the proof is to consider a \textit{truncated} version of a prior $\pi$. 
Concretely, for any distribution $\pi$, define $\pi_h$ as the law of $\pi$ conditioned on $|\theta|\le h$. 
That is, for any event $E$, 
\begin{equation}\label{eq:pi_h}
	\pi_h(E) = \frac{\bbP_{\pi}(E\cap [-h, h])}{\bbP_{\pi}([-h, h])}
\end{equation}
We first begin with a helpful lemma regarding the mmse of a truncated prior.
\begin{lemma}\label{lmm:truncate}
	For any $\pi$ and $h$ such that $\mathbb{P}_{\pi}[|\theta|\le h] > 0$, the truncated prior $\pi_h$ satisfies $\mmse_k$($\pi_h$) $\le \frac{\mmse_k(\pi)}{\PP_\pi[|\theta|\le h]}$.
\end{lemma}

\begin{proof}
	Let $E$ be the event that $|\theta| \le h$ under $\pi$. Then \[\mmse_k(\pi) = \min_{\hat{T}}\E_{\pi}[(\hat{T}_n(X)-\theta_n^k)^2] \ge \min_{\hat{T}}\E_{\pi}[(\hat{T}_n(X)-\theta_n^k)^2|E]\PP[E] = \mmse_k(\pi_h)\PP(E).\]
\end{proof}

\begin{proof}[Proof of \prettyref{lmm:6_generalized}]
	Let $E$ be the event that $|\theta_i|\le h$ for all $i = 1, 2, \cdots, n$. For any estimator $\hat{T}$ taking values in $[-h^k,h^k]$ and any prior $\pi \in \calG$, \begin{align}\label{eq:lemma_6_help}
		\E_\pi[(\hat{T}_n(X_1^n)-\theta_n^k)^2] 
		&= \E_\pi[(\hat{T}_n(X_1^n)-\theta_n^k)^2\mathbf{1}_E] + \E_\pi[(\hat{T}_n(X_1^n)-\theta_n^k)^2\mathbf{1}_{E^C}]\nonumber\\
		&\le \E_\pi[(\hat{T}_n(X_1^n)-\theta_n^k)^2|E] 
		+ \sqrt{\E_\pi[(\hat{T}_n(X_1^n)-\theta_n^k)^4]\PP_\pi(E^C)}\nonumber\\
		&\le \E_\pi[(\hat{T}_n(X_1^n)-\theta_n^k)^2|E] + \sqrt{8(M+h^{4k})n\epsilon}
	\end{align}
	
	For all distributions in $\calP([-h,h])$, the optimal estimator will take values in $[-h^k, h^k]$. Then we get the following inequalities: \begin{align*}
		\inf_{\hat{T}}\sup_{\pi\in\calP([-h,h])}\Regret_{\pi,k, n}(\hat{T}) 
		&= \inf_{\hat{T}\in[-h^k,h^k]}\sup_{\pi\in\calP([-h,h])}\E_\pi[(\hat{T}_n(X_1^n)-\theta_n^k)^2]-\mmse_k(\pi)\nonumber\\
		&\ge\inf_{\hat{T}\in[-h^k,h^k]}\sup_{\pi\in\calG}\E_{\pi_h}[(\hat{T}_n(X_1^n)-\theta_n^k)^2]-\mmse_k(\pi_h)\nonumber\\
		&\ge\inf_{\hat{T}\in[-h^k,h^k]}\sup_{\pi\in\calG}\E_{\pi}[(\hat{T}_n(X_1^n)-\theta_n^k)^2|E]-\mmse_k(\pi_h)\nonumber\\
		&\stepa{\ge}\inf_{\hat{T}\in[-h^k,h^k]}\sup_{\pi\in\calG}\E_{\pi}[(\hat{T}_n(X_1^n)-\theta_n^k)^2]-\sqrt{8(M+h^{4k})n\epsilon}-\frac1{1-\epsilon}\mmse_k(\pi)\nonumber\\
		&\ge\inf_{\hat{T}}\sup_{\pi\in\calG}\Regret_{\pi,k, n}(\hat{T})-\sqrt{8(M+h^{4k})n\epsilon}-\frac\epsilon{1-\epsilon}\mmse_k(\pi)\nonumber\\
		&\stepb{\ge}\inf_{\hat{T}}\sup_{\pi\in\calG}\Regret_{\pi,k, n}(\hat{T})-\sqrt{8(M+h^{4k})n\epsilon}-2\epsilon\sqrt{M}
	\end{align*}
	where (a) follows from \prettyref{eq:lemma_6_help} and \prettyref{lmm:truncate} and (b) follows from $\epsilon\le\frac12$ and $\mmse_k(\pi)\le\E_\pi[\theta^{2k}]\le\sqrt{M}.$ We can combine the last two terms to obtain \prettyref{eq:lemma_6_generalized}.
\end{proof}

\subsection{Properties of Poisson model}
Here, we consider a tail bound that essentially pins down the constants as discussed in \cite[(122, 124)]{polyanskiy2021sharp}: 

\begin{lemma}\label{lmm:poi_tail_bound}
	Consider the Poisson model with mixture density 
	$f_{\pi}\equiv \Poi\circ \pi$. 
	Suppose $\pi\in\mathcal{P}([0, h])$. 
	Then there exist constants $c_1, c_2 > 0$ such that 
	for all $n\ge 3$ and $x_0 = \max \{h, (c_1 + c_2h)\frac{\log n}{\log \log n}\}$, we have 
	\[
	\bbP_{X\sim f_\pi}[X \ge x_0]
	= \sum_{x\ge x_0} f_{\pi}(x)
	\le \frac 1n
	\]
\end{lemma}

\begin{proof}[Proof of \prettyref{lmm:poi_tail_bound}]
	For the case $\mathcal{P}([0, h])$, we use the following Chernoff's bound, 
	e.g. \cite[(15.19)]{polyanskiy2025information} on Poisson-Binomial distribution
	to estimate the following tail bound for $X\sim\Poi(\lambda)$ and for $x\ge h\ge \lambda$: 
	\[
	\bbP[X\ge x] \le \frac{(e\lambda)^x e^{-\lambda}}{x^x}
	\le \frac{(eh)^x e^{-h}}{x^x}
	\]
	This would mean that for $x_0\ge h$ we have 
	$\bbP_{X\sim f_\pi}[X \ge x_0]
	\le \frac{(eh)^{x_0} e^{-h}}{{x_0}^{x_0}}$. 
	
	Now setting $x_0 = c\frac{\log n}{\log \log n}$, we want to find $c$ such that the following holds: 
	\[
	(c\log n)\left(1 - \frac{\log \log \log n}{\log \log n}\right)
	-c\frac{\log n}{\log \log n}(1 - \log c + \log h)
	+h
	=
	x_0\log x_0 + h - x_0(1 + \log h)\ge \log n
	\]
	Let $L = \sup_{n\ge 3}\frac{\log \log \log n}{\log \log n} < 1$, 
	so by taking $c\ge \max\{eh, \frac{1}{1 - L}\}$ we have 
	\begin{flalign*}
		&~(c\log n)\left(1 - \frac{\log \log \log n}{\log \log n}\right)
		-c\frac{\log n}{\log \log n}(1 - \log c + \log h) + h
		\nonumber\\
		&~\ge (c\log n)(1 - L) - c \frac{\log n}{\log \log n}(1 - \log (eh) + \log h) + h
		\ge \log n
	\end{flalign*}
	as desired. 
\end{proof}

\section{Proofs for lower bound: $\ell(\theta) := \theta^k$}\label{sec:lower-bound}
In view of \prettyref{lmm:totregretind}, we will instead provide a lower bound on the individual regret over a class $\mathcal{G}$ of priors. 
\subsection{General lower bound proof setup}

We first set up a generalization to \cite[Proposition 7]{polyanskiy2021sharp}. 
Let $\ell$ be any function that is continuously differentiable everywhere (on $\mathbb{R}$ for the normal means model, and on $\mathbb{R}_{\ge 0}$ for the Poisson model), 
$G_0$ a base prior, 
and $f_0(x)\equiv f_{G_0}(x) = \int f_{\theta}(x)G_0(d\theta)$ the induced mixture density 
(here $f_{\theta}$ is the density under the channel $\gamma$).  
For a function $r$, we define the operator $K$ acting on $r$ as follows: \cite[(21)]{polyanskiy2021sharp} 
\begin{equation}\label{eq:k-operate}
	Kr(x) \triangleq \mathbb{E}_{G_0}[r(\theta) \mid X = x]
	= \frac{\int r(\theta) f_{\theta}(x) G_0(d\theta)}{f_0(x)}
\end{equation}
Also, fix an arbitrary bounded function $r$,
consider the distribution $G_{\delta}$ given by the small perturbation 
\[
dG_{\delta} \triangleq \frac{(1 + \delta r)dG_0}{1 + \delta \int rdG_0}
\]

We now consider what happens as we consider $\mathbb{E}_{G_\delta}[\ell(\theta) | X = x]$. 
To start, by \prettyref{eq:k-operate}, we have $K\ell(x) = \mathbb{E}_{G_0}[\ell(\theta) | X = x]$ (i.e. the base distribution). Then similar to 
\cite[(24)]{polyanskiy2021sharp} we may obtain 
\begin{flalign}\label{eq:KFdelta}
	\mathbb{E}_{G_\delta}[\ell(\theta) |X = x] 
	&= \mathbb{E}_{G_0}
	\left[\frac{\ell(\theta)(1 + \delta r(\theta))}{1 + \delta  Kr(x)} | X = x\right]
	\nonumber
	\\&=\frac{K\ell + \delta K(\ell\cdot r)}{1 + \delta  Kr}(x)
	\nonumber 
	\\&=K\ell(x) + \delta \left(\frac{K(\ell\cdot r) - K\ell\cdot Kr}{1 + \delta Kr}\right)(x)
	\nonumber 
	\\&=K\ell(x) + \delta K_{\ell} r(x)
	- \delta^2\cdot \frac{(Kr)(K_{\ell}r)}{1 + \delta Kr}(x)
\end{flalign}

where the operation $K_{\ell}$ is defined as (modified from \cite[(25)]{polyanskiy2021sharp}). 
\begin{equation}\label{eq:Kell_def}
	K_{\ell}r\overset{\Delta}{=} K(\ell r) - (K\ell)(Kr)
\end{equation}
When $\ell(\theta) = \theta^k$ for $k\ge 1$, we will use the notation $K_{\ell} = K_k$. 
When $k = 0$, $\ell(\theta)\equiv 1$ and $K\ell = 1$ so $K_{\ell}r = Kr - (Kr) = 0$. 
Note also the following identity, which can again be generalized from 
\cite[(25)]{polyanskiy2021sharp}
\[
K_{\ell} r(x) = \frac{d}{d\delta}\mid_{\delta = 0}\mathbb{E}_{G_{\delta}} [\ell(\theta) \mid X = x]
\]
With this, we are ready to establish the following general recipe for proving lower bounds on the empirical Bayes functional estimation problems. 
\begin{lemma}\label{lmm:7_generalized}
	Fix a prior distribution $G_0$, constants $a,\tau,\tau_1,\tau_2,\gamma\ge0$ and $m$ real-valued functions $r_1,\dots, r_m$ on $\Theta$ with the following properties,
	\begin{align*}
		\|r_q\|_\infty&\le a\quad\forall q\nonumber\\
		\|Kr_q\|_{L_2(f_0)}&\le\sqrt{\gamma} \quad\forall q\nonumber\\
		\left\|\sum_{i=1}^m v_iK_{\ell}r_i\right\|_{L_2(f_0)}^2&\ge\tau\|v\|_2^2-\tau_2\quad\forall v\in\{0,\pm1\}^m\nonumber\\
		\left\|\sum_{i=1}^m v_iK_{\ell}r_i\right\|_{L_2(f_0)}^2&\le\tau_1^2m\quad\forall v\in\{0,\pm1\}^m.
	\end{align*}
	Then the optimal regret in $\ell(\theta)$ estimation over the class of priors $\calG=\{G:||\frac{dG}{dG_0}-1||_{\infty}\le\frac12\}$ satisfies \begin{align*}
		\inf_{\hat{T}}\sup_{\pi\in\calG}\Regret_{\pi,k,n}(\hat{T})\ge C\delta^2(m(4\tau-\tau_1^2)-\tau_2),\quad\delta\overset{\Delta}{=}\frac1{\max(\sqrt{n\gamma},ma)}
	\end{align*}
	for some constant $C>0$.
\end{lemma}
\begin{proof}
	The proof follows the proof of \cite[Proposition 7]{polyanskiy2021sharp} with appropriate modifications for the case of $\ell(\theta)\neq \theta$. 
	We start by defining the following hypercubes. 
	For $i = 1, 2, \cdots, m$, define $\mu_i\triangleq \int r_idG_0$, 
	$\delta > 0$ chosen with $\delta \le \frac{1}{16 ma}$, 
	and for each $u\in \{0, 1\}^m$, 
	\[
	r_u\triangleq \sum_{i=1}^m u_ir_i
	\qquad 
	h_u\triangleq Kr_u
	\qquad 
	\mu_u\triangleq \sum_{i=1}^m u_i\mu_i
	\qquad 
	dG_u \triangleq \frac{1+\delta r_u}{1+\delta\mu_u} dG_0
	\qquad 
	f_u \triangleq \frac{(1 + \delta h_u) f_0}{1+\delta \mu_u}
	\]
	where $f_u$ is the mixture density induced by the prior $G_u$. 
	Now consider the class of priors $\tilde{\calG}=\{G_u: u\in \{0, 1\}^m\}$. 
	The construction guarantees $\tilde{\calG}\subseteq \calG$, 
	and $\frac 12 \le \frac{dG_u}{dG_0}\le \frac 32$ given that $\delta ma\le \frac{1}{16}$. 
	Define, also, $T_u(x)\triangleq \EE_{G_u}[\ell(\theta) \mid X = x]$. 
	Then the minimax regret over $\mathcal{G}$ satisfies: 
	\begin{flalign*}
		\inf_{\hat{T}}\sup_{\pi\in \calG} \Regret_{\pi, k, n}(\hat{T})
		&\ge \inf_{\hat{T}}\max_{\pi\in \tilde{\calG}} \Regret_{\pi, k, n}(\hat{T})
		\\&=\inf_{\hat{T}}\max_{u\in \{0, 1\}^m} \mathbb{E}_{G_u}||\hat{T} - T_u||^2_{L^2(f_u)}
		\\&\stepa{\ge}\inf_{\hat{T}}\max_{u\in \{0, 1\}^m} \frac 12\mathbb{E}_{G_u}||\hat{T} - T_u||^2_{L^2(f_0)}
		\\&\stepb{\ge}\min_{\hat{u}\in \{0, 1\}^m}\max_{u\in \{0, 1\}^m} \frac 18\mathbb{E}_{G_u}||T_{\hat{u}} - T_u||^2_{L^2(f_0)}
	\end{flalign*}
	where (a) uses $f_u\ge \frac 12 f_0$ and (b) is due to the following: 
	given $\hat{T}$, define $\hat{u}=\argmin_{v} \norm{T_v - \hat{T}}_{L_2(f_0)}$. 
	Then for any $u$, 
	\[\norm{T_{\hat u} - T_u}_{L_2(f_0)} \le 
	\norm{T_{\hat u} - \hat{T}}_{L_2(f_0)} + \norm{\hat{T} - T_u}_{L_2(f_0)}
	\le 2\norm{\hat{T} - T_u}_{L_2(f_0)}\] 
	from the definition of $\hat{u}$ and by the triangle inequality.

	Now we consider the following property of $T_u$, due to \prettyref{eq:KFdelta} and that $Kr = h_u$
	\[
	T_u 
	=K\ell + \delta K_{\ell}r - \delta^2 \frac{h_u}{1 + \delta h_u}\cdot K_{\ell}r
	\]
	and by our assumptions: 
	\[
	||\delta^2 \frac{h_u}{1 + \delta h_u}\cdot K_{\ell}r||_2
	\le 2\delta^2 ma ||K_{\ell}r_u||_2
	\le 2\delta^2 m^{3/2} a\tau_1
	\le \frac 18 \delta a\sqrt{m}\tau_1
	\]
	Again, by the triangle inequality: 
	\[
	||T_u - T_v||_2\ge \delta ||K_\ell(r_u - r_v)||_2 - \frac{\delta \sqrt{m}\tau_1}{4}
	\]
	Thus by using $(a-b)^2\ge\frac 12 a^2 - b^2$ this translates into 
	\[
	||T_u - T_v||_2^2\ge \frac 12\delta^2 ||K_\ell(r_u - r_v)||_2^2 - \frac{\delta^2 m\tau_1^2}{16}
	\ge \frac 12\delta^2(\tau d_H(u, v) - \tau_2) - \frac {1}{16}\delta^2m\tau_1^2
	\]
	where $d_H$ denotes the Hamming distance. 
	
	The next step is to bound $\chi^2(f_u || f_v)$. 
	Note that: 
	\[
	\frac{f_u}{f_0} - \frac{f_v}{f_0}
	= \frac{1+\delta h_u}{1 + \delta\mu_u} 
	- \frac{1+\delta h_v}{1 + \delta\mu_v}
	= \delta \frac{h_u - h_v}{1 + \delta\mu_u}
	+ (1 + \delta h_v)\frac{\delta(\mu_v - \mu_u)}{(1 + \delta\mu_u)(1 + \delta\mu_v)}
	\]
	and note that all the quantities 
	$1 + \delta h_v, 1 + \delta\mu_u, 1 + \delta\mu_v\in [\frac 12, \frac 32]$. 
	Therefore, 
	\begin{flalign*}
		\chi^2(f_u || f_v) 
		&= \int d\mu f_0 \frac{(f_u/f_0 - f_v/f_0)^2}{f_v/f_0}
		\\&=
		\delta^2\int d\mu f_0
		\frac{1 + \delta\mu_v}{1+\delta h_v}
		\left(\frac{h_u - h_v}{1 + \delta\mu_u}
		+ (1 + \delta h_v)\frac{\mu_v - \mu_u}{(1 + \delta\mu_u)(1 + \delta\mu_v)}\right)^2
		\\&\le C_0\delta^2\norm{h_u-h_v}_2^2
		+\delta^2(\mu_v-\mu_u)^2
		\\&\le C_1\delta^2\norm{h_u-h_v}_2^2
	\end{flalign*}
	with $C_1\triangleq C_0 + 1$ and $C_0$ a constant.
	The last step follows from Cauchy-Schwarz inequality: 
	\[
	\mu_u-\mu_v 
	= \int dG_0(r_u - r_v) = \int d\mu f_0(h_u - h_v) \le \norm{h_u - h_v}_2.
	\]
	By our assumption, 
	for all $u, v$ with $d_H(u, v) = 1$, we have $\norm{h_u - h_v}^2\le \gamma$. 
	
	Therefore, provided $n\delta^2\gamma\le 1$, 
	\[
	\chi^2(f_u^{\otimes n} || f_v^{\otimes n})
	=(1+\chi^2(f_u || f_v))^n - 1
	=(1 + \delta^2\gamma)^n - 1
	\le e^{n\delta^2\gamma} - 1
	\le e - 1
	\]
	for all $u, v$ where $d_H(u, v) = 1$. 
	Thus by Assouad's lemma \cite[Theorem 31.2]{polyanskiy2025information} we have for any estimator $\hat u$
	\[
	\max_{u\in \{0, 1\}^m} \mathbb{E}_{G_u}[d_H(\hat{u}, u)]\ge C_3m
	\]
	for some constant $C_3$. 
	Therefore, by choosing $\delta > 0$ such that 
	$\delta^2=\frac{1}{\max(n\gamma, m^2a^2)}$, we have: 
	\begin{flalign*}
		\inf_{\hat{T}}\sup_{\pi\in\calG}\Regret_{\pi,k,n}(\hat{T})
		&\ge \inf{\hat{u}}\max_{u\in \{0, 1\}^m} \frac 18\mathbb{E}_{G_u}||T_{\hat{u}} - T_u||^2_{L^2(f_0)}
		\\&\ge 
		\inf_{\hat{u}}\max_{u\in \{0, 1\}^m}
		\frac {1}{16} \delta^2 (\tau \bbE_{G_u}[d_H(\hat{u}, u)] - \tau_2)
		- \frac{1}{128}\delta^2m\tau_1^2
		\\&=
		\frac {1}{16} \delta^2 (\tau C_3m - \tau_2)
		- \frac{1}{128}\delta^2m\tau_1^2
		\\&=\frac {1}{16}\delta^2(m(C_3\tau - \frac{1}{8}\tau_1^2) - \tau_2)
	\end{flalign*}
	as desired. 
\end{proof}

Before we proceed to each of the two individual models, we consider the following two linear operations on a function $r$ with argument $\theta$: $\theta\cdot$ as multiplying by $\theta$, $\partial\cdot$ as differentiating w.r.t. $\theta$. We now introduce $op_j(r)$ a linear differential operator acting on $r$ such that the following is satisfied: 
\begin{equation}\label{eq:opj_def}
	K(\theta^j)K(r) = K(op_j(r)). 
\end{equation}
We record two useful observations: 
\begin{itemize}
	\item $\frac{d}{d\theta}(\theta r) = r + \theta r'$, i.e. $\partial(\theta\cdot) = 1 + \theta(\partial \cdot)$. 
	
	\item For any nonnegative integers $j_1, j_2$, $op_{j_1}(op_{j_2}(\cdot)) = op_{j_2}(op_{j_1}(\cdot))$, 
	i.e. the two actions $op_{j_1}$ and $op_{j_2}$ commute. 
\end{itemize}
The most crucial identity is the following: 
\begin{equation}\label{eq:kk_vs_opk}
	K_k(r) = K(\theta^k r) - K(\theta^k)K(r) = K(\theta^k r - op_k(r)) 
\end{equation}
which we will heavily utilize in our proofs that follow. 
\subsection{Normal means model}\label{sec:normal_means_lower}
We use the normal means prior as per \cite[Section 3.1]{polyanskiy2021sharp}, i.e. $G_0 = \mathcal{N}(0, s)$ where the variance $s > 0$ is to be specified later. 
Let $\varphi(x) = \frac{1}{\sqrt{2\pi}}e^{-x^2/2}$ be the standard normal density. 
Then the mixture distribution is 
$P_{G_0} = \calN(0, 1 + s)$, 
and the mixture density $f_0(y) = \frac{1}{\sqrt{1 + s}}\varphi(\frac{y}{\sqrt{1 + s}}).$ 
The posterior density of $\theta$ given $y$ is 
\[
\frac{f_{\theta}(y)G_0(d\theta)}{f_0(y)}
=\frac{\varphi(y - \theta)
	\frac{1}{\sqrt{s}}\varphi(\frac{\theta}{\sqrt{s}})}
{\frac{1}{\sqrt{1 + s}}\varphi(\frac{y}{\sqrt{1 + s}})}
=\frac{1}{\eta}\varphi\left(\frac{\theta}{\eta} - \eta y\right)
\]
where $\eta = \sqrt{\frac{s}{s + 1}}$. 
That is, $\theta | y\sim N(\eta^2 y, \eta^2)$. 

We will next derive a recursive formula for the action of $K_j$ and $op_j$ operators acting on a subset of $k$ times continuously differentiable functions.
We defer the proof to \prettyref{app:gaussian_lower}.
\begin{lemma}\label{lmm:8_generalized}
	We consider the following subset $\mathcal{O}^{(k)}(\mathbb{R})$ of $\mathcal{C}^k(\mathbb{R})$: 
	\begin{equation}
		\label{eq:polybnd}
		\mathcal{O}^{(k)}(\mathbb{R})
		= \{r\in \mathcal{C}^k: \forall j=0,\ldots,k, \exists C_j, m_j < \infty, |r^{(j)}(\theta)| \le  C_j(1+|\theta|^{m_j})\}
	\end{equation} 
	Then for each $j = 0, 1, \cdots, k$ and for each $r\in \mathcal{O}^{(k)}(\mathbb{R})$, the $op_j$ operation acts on $r$ in the following (recursive) fashion: 
	\begin{equation}\label{eq:opj_topr}
		op_{j + 1}(r) = \theta \, op_j (r) - \eta^2 op_j( \partial r)
	\end{equation}
	and that $op_j(r)\in \mathcal{O}^{(k - j)}(\mathbb{R})$. 
	Consequently, we have the following two identities in closed form:
	\begin{equation}\label{eq:opj_gsn}
		op_j(\cdot) = \sum_{i=0}^j \binom{j}{i}(\theta\cdot)^i (-\eta^2\partial\cdot)^{j-i}
	\end{equation}
	\begin{equation}\label{eq:gsn_kop}
		K_j r = -\sum_{m = 1}^j (-\eta^2)^m \binom{j}{m} K(\theta^{j - m}r^{(m)})
	\end{equation}
	where $r^{(m)}$ denotes the $m$-th derivative of $r$ w.r.t. the variable $\theta$. 
	
\end{lemma}

Next, we define $S \triangleq K^*K$ the self-adjoint operator of $K$, i.e. for any $r_1, r_2\in L_2(\text{Leb})$ we have 
\[
(Kr_1, Kr_2)_{L_2(f_0)} = (Sr_1, r_2)_{L_2(\text{Leb})}
\]

We now state the following, adapted and generalized from \cite[Lemma 9]{polyanskiy2021sharp}. 
We defer the proof to \prettyref{app:gaussian_lower}.

\begin{lemma}\label{lmm:9_generalized}
	Let $G_0 = \mathcal{N}(0, s)$ for some $s > 0$ fixed. 
	Then there is an orthonormal basis $\{\psi_q: q = 0, 1, \cdots\}$ in $L_2(\text{Leb})$ consisting
	of eigenfunctions of the operator $S$, satisfying the following:
	\begin{itemize}
		\item $(K\psi_q, K\psi_r)_{L_2(f_0)} = \lambda_0\mu^q \indc{q = r}$; 
		
		\item $(K_k\psi_q, K_k\psi_r)_{L_2(f_0)} = 0$ for $|q - r|\ge 2k + 1$, 
		and 
		\[ (K_k\psi_q, K_k\psi_q)_{L_2(f_0)}\ge \lambda_0  (q - k + 1)_k\alpha_1^{-2k}c(s, k)^2 \mu^{q - k}
		\]
		for all $q\ge 2k$ (recall that $(q - k + 1)_k = q(q - 1)\cdots (q - k+1)$ is the Pochhammer symbol defined in \prettyref{eq:pochhammer}). 
		
		\item $\norm{\psi_q}_{\infty}\le c_0\sqrt{\alpha_1}$ for some absolute constant $c_0$. 
	\end{itemize}
	In addition, for $0 < s < \frac 12$, this satisfies: 
	\[
	\mu\asymp s \qquad \lambda_0 \asymp \frac{1}{\sqrt{s}}\qquad \alpha_1 \asymp \frac{1}{\sqrt{s}}
	\]
	and $c(s, k) > \frac 12$. 
\end{lemma}

\begin{proof}[Proof of \prettyref{thm:gsn_bounds} (lower bound)]
	Consider, now, $G_0 = \mathcal{N}(0, s)$ for some $s\in (0, \frac 12)$. 
	Choose $\mathcal{G}'_s = \{G: \frac 12\le \frac{dG}{dG_0}\le \frac 32\}$, 
	then all such $G$ are in $\SubG(2s)$ and so 
	\[
	\inf_{\hat{T}}\sup_{\pi\in\calG'_s}\Regret_{\pi,k,n}(\hat{T})
	\le \inf_{\hat{T}}\sup_{\pi\in\SubG(2s)}\Regret_{\pi,k,n}(\hat{T}). 
	\]
	Consider the following `basis functions': for $i = 1, 2, \cdots , m$ (with $m$ to be determined later),
	\[r_j = \xi_{m + (2k + 1)j}\psi_{m + (2k + 1)j}\]
	where $\psi_1, \psi_2, \cdots, \psi_m$ satisfy \prettyref{lmm:9_generalized}, and for each $q = m + (2k + 1) j$, 
	$j = 1, 2, \cdots, m$, we have: 
	\[\xi_q\triangleq \frac{1}{\sqrt{(K_k\psi_q, K_k\psi_q)_{L_2(f_0)}}}
	\in O_k\left(\frac{1}{\sqrt{s^{-1/2} q^k s^k s^{q - k}}}\right)
	=O_k\left((s^{-(q - 1/2)}q^{-k})^{\frac 12}\right)
	\]
	Note that our construction guarantees the orthogonality of $K_kr_j$, 
	i.e. $(K_k \psi_{q_1}, K_k \psi_{q_2})_{L_2(f_0)} = 0$ whenever $q_1\neq q_2$. 
	
	Next, we note that $\norm{K\psi_q}_{L_2(f_0)}^2 = \lambda_0\mu^{q}=\Theta_k(s^{-1/2} s^q) = \Theta_k(s^{q-1/2})$, we have 
	\[
	\gamma = \max_{m\le i\le (2k + 2)m} \xi_i^2 \norm{K\psi_i}^2 = O_k(q^{-k})= O_k(m^{-k})
	\]
	for $q = m + (2k + 1)j$, $j = 1, 2, \cdots, m$. 
	In addition, we also have 
	\[
	a^2 = \max_{m\le i\le (2k + 2)m} (\xi_i)^2\norm{\psi_i}_{\infty}^2
	= O_k\left(\frac{s^{-0.5}}{s^{-0.5}q^k s^q}\right)
	= O_k\left(\frac{1}{m^k s^{(2k + 2)m}}\right)
	\]
	Thus with this setup, we may apply \prettyref{lmm:7_generalized} on $a, \tau, \tau_1, \tau_2, \gamma$ and $r_1, \cdots, r_m$ such that $\tau = \tau_1 = 1, \tau_2 = 0$ to get the following: 
	\[
	\inf_{\hat{T}}\sup_{\pi\in\calG_s'}\Regret_{\pi,k, n}(\hat{T})\ge C\cdot \frac{1}{\max\{n\gamma, m^2a^2\}}(3m)
	= C'\min\{\frac{m}{n\gamma}, \frac{1}{ma^2}\}. 
	\]
	for some constants $C, C' > 0$. 
	
	Take $m = \lfloor\frac 12\frac{\log n}{(2k + 2)\log \frac{1}{s}}\rfloor$, 
	then $s^{(2k + 2)m} = \frac{1 - o(1)}{\sqrt{n}}$. 
	We now have: 
	\[
	\frac{m}{n\gamma} = \Omega_{k}\left(\frac{m^{k+1}}{n}\right) = \Omega_{k, s}\left(\frac 1n (\log n)^{k+1}\right); 
	\qquad \frac{1}{ma^2} 
	= \Omega_{k}\left(m^{k - 1}s^{(2k+2)m}\right)
	= \Omega_{k, s}\left(\frac{(\log n)^{k - 1}}{\sqrt{n}}\right)
	\]
	which gives the $\inf_{\hat{T}}\sup_{\pi\in\SubG(2s)}\Regret_{\pi,k, n}(\hat{T}) = \Omega_{k, s}(\frac 1n(\log n)^{k + 1})$ bound. 
	
	To obtain a lower bound for $\mathcal{G} = \mathcal{P}([-h, h])$, 
	let $s = \frac{c}{\log n}$ for some $c:=c(h)$ such that: 
	\[
	\sup_{G\in \mathcal{G}_s'} G[|\theta| > h]\le 2n^{-5}. 
	\]
	Note that all such $G$ are $2s$-subgaussian, and so 
	\[
	G[|\theta| > h]\le 2e^{-h^2/4s} = 2e^{-h^2\log n/4c}
	\]
	so we may take $c = \frac{h^2}{20}$. 
	Also choose 
	$m = \lfloor\frac 12\frac{\log n}{(2k + 2)\log (1/s)}\rfloor = \lfloor\frac 12\frac{\log n}{(2k + 2)(\log c + \log \log n)}\rfloor$. 
	Then, by the same logic as the subgaussian case, we have 
	\[
	\frac{m}{n\gamma} = \Omega_{k}(\frac{m^{k+1}}{n}) = \Omega_{k, h}\left(\frac 1n \left(\frac{\log n}{\log \log n}\right)^{k+1}\right); 
	\qquad \frac{1}{ma^2} 
	= \Omega_{k}\left(m^{k - 1}s^{(2k+2)m}\right)
	= \Omega_{k, h}\left(\frac{(\log n)^{k - 1}}{(\log \log n)^{k - 1}\sqrt{n}}\right)
	\]
	i.e. $\inf_{\hat{T}}\sup_{\pi\in\calG_s'}\Regret_{\pi,k,n}(\hat{T}) = \Omega_{k, h}\left(\frac 1n \left(\frac{\log n}{\log \log n}\right)^{k+1}\right)$. 
	
	To connect between $\calG_s'$ and $\mathcal{P}([-h, h])$, 
	note that the $4k$-th moment of $G$ in $\calG_s'$ is at most 
	$M = \frac 32 \bbE_{\theta\sim G_0}[\theta^{4k}]\le c_1s^{2k}\in O_{h, k}(1)$
	for some constant $c_1 := c_1(k)$. 
	We now invoke \prettyref{lmm:6_generalized} to deduce that 
	\[
	\inf_{\hat{T}}\sup_{\pi\in [-h, h]}\Regret_{\pi,k,n}(\hat{T}) 
	\le \inf_{\hat{T}}\sup_{\pi\in\calG_s'}\Regret_{\pi,k,n}(\hat{T}) 
	- 6\sqrt{(M + h^{4k})n^{-4}}
	=\Omega_{k, h}\left(\frac 1n \left(\frac{\log n}{\log \log n}\right)^{k+1}\right)
	\]
	given that $M\in O_{h, k}(1)$, and therefore $\sqrt{(M + h^{4k})n^{-4}}\in O_{h, k}(n^{-2})$. 
\end{proof}

\subsection{Poisson model}
Next, we apply \prettyref{lmm:7_generalized} to the Poisson model. 
We mimic the technique in \cite[Section 3.2]{polyanskiy2021sharp} by taking 
$G_0$ the Gamma prior $\text{Gamma}(\alpha, \beta)$ for some $\alpha > 0$, using the shape-rate parametrization. 
That is, the density of $G_0$ is $G_0(\theta) = \frac{\beta^{\alpha}}{\Gamma(\alpha)}\theta^{\alpha - 1}e^{-\beta\theta}$. 
Here, the mixture density $f_0$ is negative binomial, i.e. \cite[(54)]{polyanskiy2021sharp}
\begin{equation}\label{eq:f0_binom}
	f_0(y) = \binom{y + \alpha - 1}{y}\left(\frac{\beta}{1 + \beta}\right)^{\alpha}
	\left(\frac{1}{1 + \beta}\right)^{y}
\end{equation}
and the posterior distribution of $\theta | X = x$ is $\text{Gamma}(x + \alpha, \beta + 1)$ \cite[(55)]{polyanskiy2021sharp}

We start with the following generalization of \cite[Proposition 10]{polyanskiy2021sharp}: 
\begin{lemma}\label{lmm:Kkr}
	Let $G_0$ be the Gamma prior $\text{Gamma}(\alpha, \beta)$, 
	and suppose $r\in\mathcal{O}^{(k)}$ as defined in \prettyref{eq:polybnd}, i.e. 
	for each $j = 0, 1, \cdots, k$ there exist $C_j, m_j < \infty$ such that for all $\theta\in\mathbb{R}_+$, 
	\[
	|r^{(j)}(\theta)| \le  C_j(1+|\theta|^{m_j})
	\]
	Then for all $j\le k$ with $0\le j\le k$ we have the following identity for $op_j$: 
	\begin{equation}\label{eq:opj_poisson_recur}
		op_{j}(r) = \theta^j \left(1 - \frac{\partial\cdot}{1 + \beta}\right)^j(r)
	\end{equation}
	and that $op_j(r)\in \mathcal{O}^{(k - j)}(\mathbb{R}_+)$. 
	Consequently, we have the following two identities in closed form:
	\begin{equation}\label{eq:opj_poisson}
		op_j(\cdot) = \sum_{i=0}^j \frac{(-1)^i}{(1 + \beta)^i}\binom{j}{i} \theta^j\partial^{(i)}(\cdot)
	\end{equation}
	\begin{equation}\label{eq:Kkr_lmm}
		K_jr = \sum_{i=1}^j\frac{(-1)^{i+1}}{(1+\beta)^i}{j\choose i} K(\theta^jr^{(i)}). 
	\end{equation}
\end{lemma}
This lemma allows us to relate the operator $K_k$ to $K$. The proof is given in \prettyref{app:poisson_lower}. Using this result, we can construct a suitable set of functions that gives the desired result when used in \prettyref{lmm:7_generalized}. The analysis will be continued in \prettyref{app:poisson_lower}, but we state the results below, which are generalizations of \cite[Lemma 11, Lemma 12]{polyanskiy2021sharp}.

\begin{lemma}\label{lmm:11_generalized}
	Let $G_0 = \text{Gamma}(\alpha,\beta)$. Then there exist positive constants $C,m_0$ depending only on $k$ such that for all $m\ge m_0,\beta\ge2,\alpha\ge(2k+2)m,$ there exist functions $r_1, \dots, r_m$ such that \begin{align}
		\label{eq:lemma11_1}\|K_kr_j\|_{L_2(f_0)}^2&=1 &\forall j = 1, \cdots, m, \\
		\label{eq:lemma11_2}(Kr_j,Kr_i)_{L_2(f_0)} = (K_kr_j,K_kr_i)_{L_2(f_0)} &= 0&\forall i\ne j,\\
		\label{eq:lemma11_3}\|Kr_j\|_{L_2(f_0)}^2&\le\frac{C\beta^k}{\alpha^km^k}&\forall j = 1, \cdots, m, \\
		\label{eq:lemma11_4}\|r_j\|_{\infty}&\le\sqrt{\frac{\beta^k}{\alpha^k}}e^{C(m\log\beta+\alpha)}&\forall j = 1, \cdots, m.
	\end{align}
\end{lemma}
\begin{lemma}\label{lmm:12_generalized}
	Let $G_0 = \text{Gamma}(\alpha,\beta)$ where $\alpha=1$ and $\beta>0$ is fixed, there exists some constant $C(\beta)>0$ such that for all $m\ge 1$ there exist functions $r_1, \dots, r_m$ such that \prettyref{eq:lemma11_1} and \prettyref{eq:lemma11_2} hold, and for all $j = 1, \cdots, m$, \begin{align}
		\label{eq:lemma12_1}\|Kr_j\|_{L_2(f_0)}^2&\le\frac{C}{m^{2k}}, \\
		\label{eq:lemma12_2}\|r_j\|_{\infty}&\le m^{1-k}e^{Cm}.
	\end{align}
\end{lemma}

Now using the functions constructed in \prettyref{lmm:11_generalized} and \prettyref{lmm:12_generalized} in \prettyref{lmm:7_generalized}, we can finally prove a lower bound.
\begin{proof}[Proof of \prettyref{thm:poisson_bound} (lower bound)]
	For the set of bounded priors $\calP([0,h])$, we apply \prettyref{lmm:7_generalized} by using the functions generated by \prettyref{lmm:11_generalized} with \[m=c_1\frac{\log n}{\log\log n}, \quad\alpha=c_1\log n, \quad\beta=c_2\alpha\] 
	where $c_1, c_2 > 0$ to be specified later based on $h$. 
	Note that \prettyref{eq:lemma11_1} and \prettyref{eq:lemma11_2} ensure that $\tau=\tau_1=1$ and $\tau_2=0$. Furthermore, \prettyref{eq:lemma11_3} gives $\gamma = \frac{Cc_2^k}{m^k}$ for some absolute constant $C$ as defined in \prettyref{eq:lemma11_3}. \prettyref{eq:lemma11_4} gives \[a = c_2^{k/2}e^{C(\alpha+m\log\beta)} = c_2^{k/2}e^{(2Cc_1+o(1))\log n}=O_{h,k}(n^{2Cc_1+o(1)}).\]
	If we pick $c_1 = \frac1{8C}$, then $ma = O_{h,k}(n^{1/4+o(1)})$ while $\sqrt{n\gamma} =O_{h,k}(n^{1/2-o(1)})$ so $\delta = \frac1{\sqrt{n\gamma}}$. 
	Applying \prettyref{lmm:7_generalized}, we have \begin{align}\label{eq:regret_lb_pretruncate}
		\inf_{\hat{T}}\sup_{\pi\in\calG}\Regret_{\pi,k, n}(\hat{T})\ge \frac{3C}{n\gamma}m = \frac{3c_1^{k+1}}{c_2^kn}\left(\frac{\log n}{\log\log n}\right)^{k+1}
	\end{align}
	where $\calG = \left\{G:\left|\frac{dG}{dG_0}-1\right|\le\frac12\right\}$. Now we relate the regret over $\calG$ to the regret over $\calP([0,h])$ using the following lemma. 
	
	By the proof of \cite[Theorem 2]{polyanskiy2021sharp}, we can choose $c_2$ such that $\PP(G\ge h)\le2n^{-4}$ for $G\in\calG$. Furthermore, it is well known the moments of the Gamma distribution are \[\E_{G_0}[\theta^{4k}] = \frac{\Gamma(\alpha+4k)}{\beta^{4k}\Gamma(\alpha)} \asymp c_2^{-4k} = O_{h,k}(1)\] so $\sup_{\pi\in\calG}\E_\pi[\theta^{4k}] = O_{h,k}(1)$. Now using \prettyref{eq:regret_lb_pretruncate} and \prettyref{lmm:6_generalized} with $\epsilon=n^{-4}$ and constant $M$, we have \begin{align*}
		\inf_{\hat{T}}\sup_{\pi\in\calP([0,h])}\Regret_{\pi,k,n}(\hat{T})&\ge \frac{3c_1^{k+1}}{c_2^kn}\left(\frac{\log n}{\log\log n}\right)^{k+1}-O_{h,k}(n^{-3/2}) = \Omega_{h,k}\left(\frac{1}{n}\left(\frac{\log n}{\log\log n}\right)^{k+1}\right).
	\end{align*}
	Now we move on to the subexponential case. If we choose $\alpha=1$ and $\beta=\frac{1}{s}$, $G_0 = \text{Expo}(s)$ so $\PP_{G_0}(\theta\ge t)\le e^{-t/s}.$ Thus, for all $G\in \calG = \left\{G:\left|\frac{dG}{dG_0}-1\right|\le\frac12\right\}$, $\PP_G(\theta\ge t)\le 2e^{-t/s}$ so $\calG \subseteq \subexpo(s)$. Now we apply \prettyref{lmm:7_generalized} by using the functions generated by \prettyref{lmm:12_generalized} with $m=c\log n$. Again, \prettyref{eq:lemma11_1} and \prettyref{eq:lemma11_2} ensure that $\tau=\tau_1=1$ and $\tau_2=0$. Furthermore, \prettyref{eq:lemma12_1} gives $\gamma = \frac{C}{m^{2k}}$ and \prettyref{eq:lemma12_2} gives \[a = m^{1-k}e^{C(\alpha+m\log\beta)} = (c\log n)^{1-k}e^{(Cc\log (1/s)\log n)+C}=O_{s,k}(n^{Cc\log (1/s)+o(1)}).\] 
	If we pick $c = \frac1{4C\log (1+1/s)}$, then $ma = O_{s,k}(n^{1/4+o(1)})$ while $\sqrt{n\gamma} = O_{s,k}(n^{1/2-o(1)})$ so $\delta = \frac1{\sqrt{n\gamma}}$. Applying \prettyref{lmm:7_generalized}, and using $\calG\subseteq\subexpo(s),$ we have \begin{align*}
		\inf_{\hat{T}}\sup_{\pi\in\text{SubE}(s)}\Regret_{\pi,k,n}(\hat{T})\ge\inf_{\hat{T}}\sup_{\pi\in\calG}\Regret_{\pi,k,n}(\hat{T})\ge \frac{3C}{n\gamma}m = \Omega_{s,k}\left(\frac{1}{n}\left(\log n\right)^{2k+1}\right).
	\end{align*}
\end{proof}

\section{Proofs for upper bound (Poisson model): $\ell(\theta) := \theta^k$}\label{sec:upper_bound}
In this section, we establish the minimax optimality of several estimators for the Poisson model for the functional $\ell(\theta) \triangleq \theta^k$, 
up to constants that could depend on the exponent $k$ and either $h$ or $s$ for bounded or subexponential priors, respectively. 
Note that we will establish the total regret bound for the modified Robbins estimator, 
and the individual regret bound for both the minimum distance estimator and the ERM estimator. 
This is sufficient in view of \prettyref{lmm:totregretind}. 

In addition, for bounded priors, we will make the dependence on $h$ and $k$ explicit to facilitate the extension to general smooth functions. 

\subsection{Modified Robbins method}
We now give a formal statement on the optimality of the modified Robbins estimator. 
\begin{theorem}\label{thm:robbins}
	The modified Robbins estimator $\Trobk$ defined in \prettyref{eq:robbins} satisfies the following regret bounds:
	\begin{enumerate}
		\item For some universal constants $A, B > 0$ we have
		\[\sup_{\pi\in\calP([0,h])}\TotRegret_{\pi,k,n}(\Trobk) \le 
		(A + Bh)^{2k + 1}
		\left(\frac{\log n}{\log \log n} + k\right)^{k + 1}\le O_{h,k}\left(\left(\frac{\log n}{\log\log n}\right)^{k+1}\right)\]
		\item \[\sup_{\pi\in\subexpo(s)}\TotRegret_{\pi,k,n}(\Trobk) \le O_{s,k}\left((\log n)^{2k+1}\right).\]
	\end{enumerate}
\end{theorem}
We first start with the following lemma, which we use to bound the regret on a prior bounded by $[0,h]$. For the purpose of generalizing this argument to subexponential priors, we also allow $h$ to depend on $n$.
\begin{lemma}\label{lmm:rob_help}
	For every prior $\pi\in\calP([0,h])$, the total regret satisfies
	\begin{align*}
		\TotRegret_{\pi,k,n}(\Trobk) 
		\le & c\left(1 + h^{2k} + k!(1 + h^k)+\sum_{x\ge1}(x + 1)_k h^k\min\left\{n^2f_\pi(x)^2,1\right\}+h^{2k}\min\left\{nf_\pi(x),1\right\}\right).
	\end{align*}
	where $c > 0$ is a universal constant. 
\end{lemma}

We next state an auxiliary result that generalizes \cite[(120)]{polyanskiy2021sharp}, which we defer to proof to \prettyref{app:robbins}. 

\begin{lemma}\label{lmm:zero_term}
	Let $\pi\in\mathcal{P}([0, h])$. Then the following holds: 
	\[
	\frac{f_{\pi}(k)}{(1 - f_\pi(0))f_{\pi}(0)}
	\le 
	\frac{1}{k!}\max\{4h^k, e\}. 
	\]
\end{lemma}

\begin{proof}[Proof of \prettyref{lmm:rob_help}]
	We have  
	\begin{align}\label{eq:regret_robbins}
		\TotRegret_{\pi,k,n}(\Trobk) &= \E\left[\sum_{i=1}^n ((\Trobk(X_1^n))_i-\tpik(X_i))^2\right]\nonumber\\
		&=\E\left[\sum_{x\ge0}N(x)(x + 1)_k^2\left(\frac{N(x+k)}{N(x)}-\frac{f_\pi(x+k)}{f_\pi(x)}\right)^2 \mathbf{1}_{N(x)>0}\right]\nonumber\\
		&=\E\left[\sum_{x\ge 0}\frac{\mathbf{1}_{N(x)>0}(x + 1)_k^2}{N(x)}\left(N(x+k)-\frac{f_\pi(x+k)N(x)}{f_\pi(x)}\right)^2\right].
	\end{align}
	To proceed, we note a few observations. 
	First, we note that conditioned on $N(x)$, $N(x+k)\sim\text{Binom}\left(n-N(x),q\right)$ where $q\triangleq \frac{f_\pi(x+k)}{1-f_\pi(x)}$. It follows that 
	\begin{align}\label{eq:regret_robbins_condition}
		&~\E\left[\left(N(x+k)-
		\frac{f_\pi(x+k)N(x)}{f_\pi(x)}\right)^2|N(x)\right] \nonumber\\
		\stepa{=}& ~ (n-N(x))q(1-q) + \left((n-N(x))q-
		\frac{f_\pi(x+k)N(x)}{f_\pi(x)}\right)^2\nonumber\\
		\stepb{\le}& ~ n\frac{f_\pi(x+k)}{1-f_\pi(x)} + 
		\left(\frac{f_\pi(x+k)}{(1-f_\pi(x))f_\pi(x)}\right)^2\left(nf_\pi(x)-N(x)\right)^2.
	\end{align}
	where (a) follows from the bias-variance decomposition on $N(x+k) | N(x)\sim\text{Binom}\left(n-N(x),q\right)$, 
	and (b) simply uses $(n - N(x))q(1 - q) \le nq$. 
	
	Next, we note that for all $x\ge 1$, 
	\begin{equation}\label{eq:poi_sterling}
		f_{\pi}(x) = \int p(x|\theta) d\pi(\theta)
		\le \sup_{\theta} p(x |\theta)
		=\sup_{\theta} \frac{e^{-\theta}\theta^x}{x!}
		\stepa{=} \frac{e^{-x}x^x}{x!}
		\stepb{\le} \frac{1}{\sqrt{2\pi x}}
		\le \frac{1}{\sqrt{2\pi}}
	\end{equation}
	where (a) is obtained by noting that the supremum of $e^{-\theta}\theta^x$ is obtained when $\theta = x$, 
	and (b) is due to Stirling's method. 
	Thus this implies $\frac{1}{1 - f_{\pi}(x)} \le c_3\triangleq \frac{\sqrt{2\pi}}{\sqrt{2\pi} - 1}$ for all $x\ge 1$. 
	Finally, given $N(x)\sim\text{Binom}(n,f_\pi(x))$, by \cite[Lemma 16]{polyanskiy2021sharp}, 
	there exist absolute constants $c_1,c_2 > 0$ such that 
	\begin{equation}\label{eq:binom_inverse_bound}
		\E\left[\frac{\mathbf{1}_{N(x)>0}}{N(x)}\right]\le c_1\min\left\{nf_\pi(x),\frac1{nf_\pi(x)}\right\},\quad\E\left[\frac{\mathbf{1}_{N(x)>0}}{N(x)}(N(x)-nf_\pi(x))^2\right]\le c_2.
	\end{equation}
	
	Therefore, using \prettyref{eq:regret_robbins_condition}, we can continue \prettyref{eq:regret_robbins} to get \begin{align}\label{eq:regret_robbins_v2}
		&~n\cdot\Regret_{\pi,k,n}(\Trobk) \nonumber\\
		=&~ \E\left[\sum_{x\ge0}\frac{\mathbf{1}_{N(x)>0}(x+1)_k^2}{N(x)}\left(n\frac{f_\pi(x+k)}{1-f_\pi(x)}+ \left(\frac{f_\pi(x+k)}{(1-f_\pi(x))f_\pi(x)}\right)^2\left(nf_\pi(x)-N(x)\right)^2\right)\right]\nonumber\\
		\stepa{\le} & ~\E\Bigg[\frac{\mathbf{1}_{N(0)>0}(k!)^2}{N(0)}\left(n\frac{f_\pi(k)}{1-f_\pi(0)}+ \left(\frac{f_\pi(k)}{(1-f_\pi(0))f_\pi(0)}\right)^2\left(nf_\pi(0)-N(0)\right)^2\right)\nonumber\\
		& ~\quad +\sum_{x\ge1}\frac{\mathbf{1}_{N(x)>0}(x+1)_k^2}{N(x)}\left(c_3nf_\pi(x+k)+ c_3^2\left(\frac{f_\pi(x+k)}{f_\pi(x)}\right)^2\left(nf_\pi(x)-N(x)\right)^2\right)\Bigg]\nonumber\\
		\stepb{\le} & ~(k!)^2\left(\frac{c_1f_\pi(k)}{(1-f_\pi(0))f_\pi(0)}+ c_2\left(\frac{f_\pi(k)}{(1-f_\pi(0))f_\pi(0)}\right)^2\right)\nonumber\\
		&~\quad+\sum_{x\ge1}c_3c_1(x+1)_k\tpik(x)\min\left\{n^2f_\pi(x)^2,1\right\}+c_3^2c_2\tpik(x)^2\min\left\{nf_\pi(x),1\right\}\nonumber\\
		\stepc{\le} &~ c'(1 + h^{2k} + k!(1 + h^k))+\sum_{x\ge 1} c_3c_1(x+1)_kh^k\min\{n^2f_{\pi}(x)^2, 1\}
		+c_3^2c_2h^{2k}\min\{nf_{\pi}(x), 1\}. 
	\end{align}
	where $c', c_1, c_2, c_3$ are absolute constants. 
	Here, (a) combines \prettyref{eq:regret_robbins_condition} and \prettyref{eq:poi_sterling}, 
	(b) is due to \prettyref{eq:binom_inverse_bound} and that $\tpik(x) = (x + 1)_k\frac{f_\pi(x+k)}{f_\pi(x)}$, 
	and (c) is by \prettyref{lmm:zero_term} together with $\tpik(x)\le h^k$ for all $x\ge 0$. 
\end{proof}
To allow us to use \prettyref{lmm:rob_help} and establish a regret bound, 
we prove the following tail bounds on $f_\pi$ to bound the $\min$ using $f_\pi(x)$ for large $x$.
\begin{lemma}\label{lmm:17_generalized}
	The tail bounds of the Poisson mixture $f_{\pi}$ for $\pi\in\calP([0,h])$ and $\pi\in\subexpo(s)$ satisfy the following: 
	\begin{itemize}
		\item Let $\pi\in\calP([0,h])$. Then for $x_0=\max\left(2h,(c_1 + c_2h)\frac{\log n}{\log\log n}\right),$ \begin{align}\label{eq:lemma17_1}\sum_{x>x_0+k}f_\pi(x)^2(x+1)_k\le\frac{2^{k+1}h^k}{n^2}.\end{align}
		
		\item Let $\pi\in\subexpo(s)$. Then for $x_1=(s + 1)\log n,$ \begin{align}\label{eq:lemma17_2}
			\sum_{x>x_1+k}f_\pi(x)^2(x+1)_k\le \frac{2(s+1)^{2(k + 1)}}{(2s + 1)^{k + 1}n^2}(e(x_1 + k))^k.\end{align}
	\end{itemize}
\end{lemma}
A proof of this lemma can be found in \prettyref{app:robbins}.

To obtain regret bounds on a subexponential prior, other than results on regret of truncated prior established in \prettyref{lmm:truncate}, we also need to bound the moments of $\theta$ and $X_{\max}$ for $X_{\max} = \max \{X_1, \cdots, X_n\}$ for a subexponential prior.

\begin{lemma}\label{lmm:moment}
	Let $\pi\in\subexpo(s)$, $q\ge 1$, and $X\sim f_\pi$. 
	Then: 
	\begin{align}\label{eq:moment_bound}
		\E_\pi[\theta^{q}] \le 2q!s^{q}\quad\quad\quad \E[X_{\max}^{q}] \le \frac{4(\log n)^{q} + \frac 32 q!\max(1, (\log n)^{q - 1})}{(\log (1 + \frac{1}{2s}))^{q}}.
	\end{align}
\end{lemma}
A proof of this lemma relies mainly on rewriting the expectation using tail probabilities and bounding those values. The exact proof is given in \prettyref{app:robbins}.

Recall the notation of the truncated prior $\pi_h$ defined before \prettyref{lmm:truncate}. With these lemmas, we are now ready to show that the regret under $\pi$ exceeds the regret under the truncated prior $\pi_{c_1s\log n}$ by at most $o_{s,k}(1/n)$.
\begin{lemma}\label{lmm:truncate_regret} 
	Let $\pi\in \subexpo(s)$, 
	$M = 8k(4k - 1)!s^{4k}$ and $L = \bbE[\hat{T}_n(X)^4]$ for some given estimator $\hat{T}$. 
	Then for some absolute constant $c$, \[\Regret_{\pi,k,n}(\hat{T})\le\Regret_{\pi_{11s\log n},k, n}(\hat{T}) + c\sqrt{n^{-9}(L + M)}.\]
\end{lemma}
\begin{proof}[Proof of \prettyref{lmm:truncate_regret}]
	Let $\pi \in \subexpo(s)$, then for the constant $c(s)\overset{\Delta}{=} 11s$ such that \[\epsilon =\PP[\theta > c(s)\log n] \le \frac{1}{n^{10}},\quad\theta\sim\pi.\]
	Let $E$ be the event $\{\theta_i \le c(s)\log n, \forall i=1, \dots, n\}$. By the union bound, $\PP[E^c]\le n^{-9}$. 
	By the definition of regret in \prettyref{def:regret}, we obtain the following series of equations: \begin{align}\label{eq:regret_sube_truncate}
		\Regret_{\pi,k,n}(\hat{T}) 
		&= \E_\pi[(\hat{T}_n(X_1^n)-\theta_n^k)^2] - \mmse_k(\pi) \nonumber\\
		&\le \E_\pi[(\hat{T}_n(X_1^n)-\theta_n^k)^2|E] - \mmse_k(\pi_{11s\log n}) + \mmse_k(\pi_{11s\log n}) \nonumber\\
		&\quad- \mmse_k(\pi) + \E_\pi[(\hat{T}_n(X_1^n)-\theta_n^k)^2\mathbf{1}_{E^c}]\nonumber\\
		&= \Regret_{\pi_{11s\log n}, k, n}(\hat{T}) + \mmse_k(\pi_{11s\log n}) - \mmse_k(\pi) + \E_\pi[(\hat{T}_n(X_1^n)-\theta_n^k)^2\mathbf{1}_{E^c}].
	\end{align}
	For the last term of \prettyref{eq:regret_sube_truncate}, 
	we have 
	\begin{flalign*}
		\E_\pi[(\hat{T}_n(X_1^n)-\theta_n^k)^2\mathbf{1}_{E^c}]
		\stepa{\le} & ~\sqrt{\PP[E_c]\E_\pi[(\hat{T}_n(X_1^n) - \theta_n^k)^4]} 
		\le \sqrt{n^{-9}\E_\pi[(\hat{T}_n(X_1^n) - \theta_n^k)^4]}
		\nonumber\\
		\stepb{\le} & ~\sqrt{16n^{-9}(\bbE_{\pi}[\hat{T}_n(X)^4] + \bbE_{\pi}[\theta^{4k}])}
		\stepc{\le} \sqrt{16n^{-9}(L + M)}
	\end{flalign*}
	
	where (a) is due to Cauchy-Schwarz, (b) is due to the fact $(a - b)^4\le 16(a^4 + b^4)$, 
	and (c) is due to \prettyref{lmm:moment} that 
	$\bbE_{\pi}[\theta^{4k}]\le M$. 
	For the middle two terms of \prettyref{eq:regret_sube_truncate}, \prettyref{lmm:truncate} tells us that 
	\[\mmse_k(\pi) \ge \mmse_k(\pi_{11s\log n})(1-n^{-9})\]
	so we have 
	\[\mmse_k(\pi_{c_1s\log n}) - \mmse_k(\pi) \le \frac{n^{-9}}{1-n^{-9}}\mmse_k(\pi) \stepa{\le} 2n^{-9}\sqrt{M}\] 
	where (a) is because $n^{-9}\le\frac12$ and 
	$\mmse_k(\pi)\le \bbE_{\pi}[\theta^{2k}]\le \sqrt{M}$. 
	Combining these inequalities together, we obtain the desired result.
\end{proof}
\begin{proof}[Proof of \prettyref{thm:robbins}]
	We first deal with the case $\pi\in\calP ([0, h])$. 
	Using $x_0 = \max\{2h, (c_1+c_2h)\frac{\log n}{\log \log n}\}$ with 
	$c_1, c_2$ given in \prettyref{lmm:17_generalized}, 
	and splitting up the summation in \prettyref{lmm:rob_help} at $x_0+k$ and upper bounding the minima, 
	we obtain 
	\begin{align*}
		\TotRegret_{\pi,k,n}(\Trobk) &\lesssim 1 + k!(1 + h^k) + h^{2k} +\sum_{x=1}^{x_0+k}\left((x+1)_kh^k+h^{2k}\right)\\
		&\quad+\sum_{x>x_0+k}\left((x+1)_kh^kn^2f_\pi(x)^2+h^{2k}nf_\pi(x)\right)\nonumber\\
		&\stepa{\le} 1 + h^{2k} + k!(1 + h^k) + h^{2k}(x_0+k) + h^k(x_0+k)^k(x_0+1)_k + 2^{k+1}h^{2k} + h^{2k}\nonumber\\
		&\lesssim 
		1 + h^{2k} + k!(1 + h^k)
		+ h^{2k}[(c_1 + c_2h)(\frac{\log n}{\log \log n} + k) + 2^{k + 1} + 1]
		\nonumber\\
		&
		+ h^k(c_1+c_2h)^{k + 1}(\frac{\log n}{\log \log n} + k)^{k + 1}
		\nonumber\\
		&\stepb{\le}
		(c_1 + (c_2 + 1)h)^{2k + 1}
		\left(\frac{\log n}{\log \log n} + k\right)^{k + 1}
	\end{align*}
	where (a) is from \prettyref{eq:lemma17_1} and \cite[(122)]{polyanskiy2021sharp}, and (b) is from using $k!\le k^k$ for $k\ge 1$, and that $h^{2k}(2^{k + 1} + 1)\le (c_1 + (1 + c_2h))^{2k + 1}$. 
	
	Now we consider the case $\pi$ is subexponential. 
	Note that although \prettyref{lmm:truncate_regret} is a result on individual regret, we can adapt this result to bound the total regret of $\Trobk$ given that it is permutation invariant.
	Now, take $h = c's\log n$ (with $c' = 11$). 
	We split up the summation in \prettyref{lmm:rob_help} at $x_1+k$ from \prettyref{lmm:17_generalized} 
	(where $x_1\triangleq (s + 1)\log n$, as in \prettyref{lmm:17_generalized}) and upper bound the $\min$'s to obtain 
	\begin{flalign*}
		\TotRegret_{\pi_h, k,n}(\Trobk)
		&~\lesssim 1 + k!(1 + h^k) + h^{2k} +\sum_{x=1}^{x_1+k}\left((x+1)_kh^k+h^{2k}\right)
		\nonumber\\ &\quad+\sum_{x > x_1 + k} \left((x+1)_kh^kn^2f_\pi(x)^2+h^{2k}nf_\pi(x)\right)\nonumber\\
		&~\stepa{\lesssim} 1 + h^{2k} + k!(1 + h^k) 
		+ h^{2k}(x_1 + k) + h^k(x_1+k)(x_1+1)_k\nonumber\\
		&~ + h^k\frac{2(s+1)^{2(k + 1)}}{(2s + 1)^{k + 1}}(e(x_1 + k))^k + h^{2k}\frac{2(s+1)^{2}}{(2s + 1)n}\nonumber\\
		&~\stepb{\lesssim} 
		((c's + 1)\log n + k)^{2k + 1}
		+ e^k((c's + 1)\log n + k)^{2k}
		\cdot \frac{2(s + 1)^{2(k + 1)}}{(2s + 1)^{k + 1}}\nonumber\\
		&~=O_{s,k}\left(\left(\log n\right)^{2k+1}\right)
	\end{flalign*}
	where (a) follows from \prettyref{eq:lemma17_2} and 
	\cite[(124)]{polyanskiy2021sharp}, and 
	(b) is by plugging $h$ and $x_1$, $(x_1 + 1)_k\le (x_1+k)^k$, 
	$\frac{(s+1)^2}{(2s + 1)}\ge 1$ 
	(and therefore $\frac{(s+1)^2}{(2s + 1)}\le \frac{(s+1)^{2(k+1)}}{(2s + 1)^k}$), 
	and also 
	$h^kk!\lesssim ((c's + 1)\log n + k)^{2k + 1}$. 
	
	To bound $\Regret_{\pi, k, n}(\Trobk)$, 
	we remind ourselves that 
	$\Trobk(x) = \frac{(x + 1)_kN(x + k)}{N(x)}\le n(x + 1)_k$ 
	given that $N(x + k)\le n$ and $N(x)\ge 1$ 
	for $x\in \{X_1, \cdots, X_n\}$. 
	This means 
	\[
	\bbE_{\pi}[\Trobk(X)^4]
	\le n^4\bbE_{\pi}[(X + 1)_k^4]
	\stepa{\le} 2^{4k}n^4[k^{4k} + \bbE_{\pi}[X^{4k}]]
	\stepb{\le} 2^{4k}n^4[k^{4k} + 8k(4k - 1)!s^{4k}]
	\]
	where (a) is because $(x + 1)_k\le \max\{(2k)^k, 2^kx^k\}$ 
	depending on whether $x\le k$, 
	and (b) is due to \prettyref{lmm:moment}. 
	Now the $\Regret_{\pi, k, n}(\Trobk) - \Regret_{\pi_h, k, n}(\Trobk)$ is $\sqrt{n^{-9}(L + M)}$ as per \prettyref{lmm:truncate_regret}. 
	We have $M = 8k(4k - 1)!s^{4k}$, 
	and we may take $L = 2^{4k}n^4[k^{4k} + 8k(4k - 1)!s^{4k}]$. 
	It follows that the extra contribution to the regret is bounded by 
	$\sqrt{n^{-5}2^{4k}[k^{4k} + 8k(4k - 1)!s^{4k}]} \lesssim \frac{(2s + 1)^{2k}(4k)!}{n^2}$, 
	which is $o_{s, k}(1/n)$ 
	(note the use of the fact that $k^{4k}\le (4k)!$ due to Stirling's inequality). 
	
\end{proof}

\subsection{Minimum distance method}
In this section, we use the NPMLE prior \cite{KW56} to estimate $\theta^k$. This extension is very natural, as the estimation of the prior distribution remains the same. Using this estimated prior, we calculate its empirical Bayes estimator by applying \prettyref{eq:bayes_polyk}. 

Before proving the main results, we first discuss prior estimation via NPMLE. This method is a specific instance of the more general class of minimum distance estimators. These estimators are defined by a measure of distance between two distributions. There are many possible distance functions, so we will focus on a specific set of functions which we call \textit{generalized distance functions}. 
\begin{definition}[\textit{Generalized distance functions}]
	A function $d:\calP(\Z_+)\times\calP(\Z_+)\to [0, \infty]$ such that $d(p\parallel q) = 0$ if and only if $p=q$.
\end{definition}
Note that this includes all metrics and many divergences. Then a minimum distance estimator with respect to $d$ over a set of distributions $\calG$ is \[\hat\pi\in\argmin_{\pi\in\calG}d(p_n^{\sfem}\parallel f_\pi).\] 
In \cite{JPW24}, they describe some specific examples of these distance functions that correspond with well-known estimators, including the NPMLE estimator we focus on:
\begin{itemize}
	\item The NPMLE estimator corresponds to the KL-divergence $d(p\parallel q)=D(p\parallel q)=\E\left[\log\frac{p(x)}{q(x)}\right].$
	\item The minimum-Hellinger estimator corresponds to the squared Hellinger distance $d(p\parallel q)=H^2(p, q)=\sum\left(\sqrt{p(x)}-\sqrt{q(x)}\right)^2.$
	\item The minimum-$\chi^2$ estimator corresponds to the $\chi^2$-divergence $d(p\parallel q)=\chi^2(p\parallel q)=\sum\frac{(p(x)-q(x))^2}{q(x)}.$
\end{itemize}

These estimators have their corresponding generalized distance metric satisfying Assumptions 1 and 2 in \cite{JPW24}:
\begin{assump}
	There exists a map $\varphi:\calP(\Z_+)\to\R$ and $\ell:\R^2\to\R$ such that for any two distributions $p,q\in\calP(\Z_+)$, \[d(p\parallel q)=\varphi(p)+\sum_{x\ge0}\ell(p(x),q(x))\] where $\ell(a,b)$ is strictly decreasing and convex in $b$ for $a>0$ and $\ell(0,b)=0$ for $b\ge0$.
\end{assump}
\begin{assump}
	There exist positive constants $c_1, c_2$ such that for $p,q\in\calP(\Z_+)$, \[c_1H^2(p,q)\le d(p\parallel q)\le c_2\chi^2(p\parallel q).\]
\end{assump}

The KL-divergence satisfies these assumptions, so the following theorem applies to the NPMLE estimator as well. The main result of this section is that a minimum distance estimator satisfies the following regret bounds:

\begin{theorem}\label{thm:npmle}
	Suppose $d$ satisfies Assumptions 1 and 2. The following regret bounds hold:
	\begin{enumerate}
		\item 
		There exists an absolute constant $c_1 > 0$ such that if 
		$\hat\pi=\argmin_{\pi\in\calP([0,h])}d(p_n^{\text{emp}}\parallel f_{\pi})$ and $\hat{T} := \hat{T}_{\hat{\pi}} = (\hat{t}_{\hat{\pi}, k}(X_1), \cdots, \hat{t}_{\hat{\pi}, k}(X_n))$ is the Bayes estimator for the prior $\hat\pi$, then for any $n\ge 3$, 
		\[\sup_{\pi\in\calP([0,h])}\Regret_{\pi,k,n}(\hat{T})
		\le \frac{c_1}{n}(2(2+he))^{k + 1} h^{2k}\left(\frac{\log n}{\log \log n} + k\right)^{k + 1}
		= O_{h,k}\left(\frac1n\left(\frac{\log n}{\log\log n}\right)^{k+1}\right)\]
		
		\item If $\hat\pi=\argmin_{\pi\in \mathbb{P}(\mathbb{R}_+)}d(p_n^{\text{emp}}\parallel f_{\pi})$ and $\hat{T}$ is the Bayes estimator for the prior $\hat\pi$, then for a fixed $s$ and any $n\ge 2$, 
		\[\sup_{\pi\in\subexpo(s)}\Regret_{\pi,k,n}(\hat{T})=O_{s,k}\left(\frac1n\left(\log n\right)^{2k+1}\right).\]
	\end{enumerate}
\end{theorem}

To prove \prettyref{thm:npmle}, we will use the following more general lemma bounding the regret in terms of the Hellinger distance between $f_{\pi}$ and $f_{\hat{\pi}}$, before using the results on bounds of Hellinger distance achieved by $\hat{\pi}$ \cite[Theorem 2]{JPW24}.
\begin{lemma}\label{lmm:npmle}
	Let $\pi$ be a distribution such that $\E_\pi[\theta^{4k}] \le M$ for some constant $M$. Then for any distribution $\hat\pi$ supported on $[0,\hat h]$, any $h>0$ with $\PP_\pi(\theta\le h) > \frac12$ and any $K\ge1$,
	\begin{align*}
		\Regret_{\pi,k,n}(\hat{t}_{\hat{\pi}, k})\le&\left\{12(h^{2k}+\hat h^{2k})+48(h^k+\hat h^k)K^k\right\}(H^2(f_\pi,f_{\hat \pi})+4\PP_\pi(\theta>h))\\
		&+2(h^k+\hat h^k)^2\PP_{f_\pi}(X>K-k)+2(1+2\sqrt2)\sqrt{(M+\hat h^{4k})\PP_\pi(\theta>h)}
	\end{align*}
	where $\hat{t}_{\hat{\pi}, k}$ is the Bayes estimator for the prior $\hat\pi$.
\end{lemma}
A proof of this lemma is provided in \prettyref{app:mindist}. Now equipped with \prettyref{lmm:npmle}, we can find a suitable application to prove the original result.

\begin{proof}[Proof of \prettyref{thm:npmle}]
	For any $\pi\in\calP([0,h])$, we have 
	$\hat\pi=\argmin_{\pi\in\calP([0,h])}d(p_n^{\sfem}\parallel f_{\pi})$. We apply \prettyref{lmm:npmle} with \[\hat h=h,\quad M=h^{4k},\quad K=\left\lceil\frac{2(2+he)\log n}{\log\log n}+k-1\right\rceil.\] 
	
		We have $\mathbb{P}_{\pi}[\theta > h] = 0$.
		By \cite[Theorem 2(a)]{JPW24},
		the minimum distance estimator enjoys the following Hellinger bound:
		$\mathbb{E}[H^2(f_{\pi}, f_{\hat{\pi}})]
		\le \frac{C_1}n\left(\frac{\log n}{\log\log n}\right)$
		for some $C_1 := C_1(h)$. 
		By \cite[Lemma 11]{JPW24}, we have the tail bound $\PP_{f_\pi}\left(X\ge K-k+1\right)\le\frac2{n^2}$. 
	
	We note that from \cite[(19)]{JPW24} and also their choice of $K$ after (22) that 
	$C_1(h)\asymp 2 + he$. 
	Thus, putting together, we have \begin{align*}
		&~\Regret_{\pi,k,n}(\hat{T})\\
		\le&~\left\{24h^{2k}+96h^k\left\lceil\frac{2(2+he)\log n}{\log\log n}+k-1\right\rceil^k\right\}\E[H^2(f_\pi,f_{\hat \pi})]+2(2h^k)^2\PP_{f_\pi}(X>K-k)\nonumber\\
		\lesssim &~ \left\{2^k(2+he)^kh^{2k}\left(\frac{\log n}{\log \log n} + k\right)^{k}\right\}\left(\frac{(2 + he)\log n}{n\log \log n}\right) + \frac{h^{2k}}{n^2}\nonumber\\
		=& ~O_{h,k}\left(\frac1n\left(\frac{\log n}{\log\log n}\right)^{k+1}\right).
	\end{align*}
	
	For any $\pi\in\subexpo(s)$, we have $\hat\pi=\argmin_{\pi}d(p_n^{\sfem}\parallel f_{\pi})$. By \cite[Lemma 8]{JPW24}, $\hat\pi$ is supported on $[0,\hat h]$ where $\hat h=X_{\max}$. 
	By the bound on $\E[\theta^{4k}]$ in \prettyref{lmm:moment}, we can apply \prettyref{lmm:npmle} with \[h=4s\log n,\quad M=8k(4k-1)!s^{4k},\quad K=\frac{2\log n}{\log\left(1+\frac1{2s}\right)}+k-1.\] 
	By \cite[(44)]{JPW24} and the definition of subexponential prior, we have 
	\[\PP_{f_\pi}\left(X\ge\frac{2\log n}{\log\left(1+\frac1{2s}\right)}\right)\le\frac3{2n^2},\quad\PP_\pi(\theta>h)\le\frac2{n^4}.\] 
	Thus, we have \begin{align}\label{eq:SubE_npmle_regret}
		\Regret_{\pi,k,n}(\hat{T})&\le\E\left[\left\{12(h^{2k}+X_{\max}^{2k})+48(h^k+X_{\max}^k)K^k\right\}H^2(f_\pi,f_{\hat \pi})+\frac2{n^4}\right]\nonumber\\
		&\quad\quad+2(h^k+X_{\max}^k)^2\frac{3}{2n^2}+(1+2\sqrt2)\sqrt{(M+\hat h^{4k})\frac2{n^4}}\nonumber\\
		&= \E\left[\left\{12(h^{2k}+X_{\max}^{2k})+48(h^k+X_{\max}^k)K^k\right\}H^2(f_\pi,f_{\hat \pi})\right] + O_{s,k}\left(\frac1{n^2}\right).
	\end{align}
	It remains to bound the first term in \prettyref{eq:SubE_npmle_regret}. Splitting into the events $X_{\max}\le 2K$ and $X_{\max} > 2K$, and using $H^2 \le 2$, we have \begin{align}\label{eq:SubE_npmle_regret_t1}
		&\E\left[\left\{h^{2k}+X_{\max}^{2k}+4(h^k+X_{\max}^k)K^k\right\}H^2(f_\pi,f_{\hat \pi})\right] \nonumber\\\le& (h^{2k}+(2K)^{2k} + 4(h^k+(2K)^k)K^k)\E\left[H^2(f_\pi,f_{\hat \pi})\right]\nonumber\\
		&+2\E\left[\left\{h^{2k}+X_{\max}^{2k}+4(h^k+X_{\max}^k)K^k\right\}\mathbb{I}_{X_{\max}\ge 2K}\right].
	\end{align}
	By \cite[Theorem 2(b) and (23)]{JPW24}, we know $\E\left[H^2(f_\pi,f_{\hat \pi})\right]
	\le C(\frac{\log n}{n\log (1 + \frac{1}{2s})})$ for some constant $C$ depending on $c_1, c_2$ as stated in the assumptions. 
	Therefore, from our choice of $h$ and $K$, 
	$h^{2k}, (2K)^{2k}$ and $h^kK^k$ are all bounded by 
	$((c_3 + c_4s)\log n + k)^{2k}$ for some constant $c_3, c_4$ depending on 
	$c_1$ and $c_2$, so the first term in \prettyref{eq:SubE_npmle_regret_t1} is then bounded by $\frac{1}{n}((c_3 + c_4s)\log n + k)^{2k + 1}$. 
	By Cauchy-Schwarz, the second term in \prettyref{eq:SubE_npmle_regret_t1} is \begin{align*}
		&2\E\left[\left\{h^{2k}+X_{\max}^{2k}+4(h^k+X_{\max}^k)K^k\right\}\mathbb{I}_{X_{\max}\ge2K}\right] \nonumber\\
		\le&2\sqrt{\E\left[\left\{h^{2k}+X_{\max}^{2k}+4(h^k+X_{\max}^k)K^k\right\}^2\right]\PP_{f_\pi}\left(X_{\max}\ge2K\right)}\nonumber\\
		\le&2\sqrt{\E\left[\left\{h^{2k}+X_{\max}^{2k}+4(h^k+X_{\max}^k)K^k\right\}^2\right]n\PP_{f_\pi}\left(X\ge2K\right)}\nonumber\\
		\stepa{\le}&2\sqrt{\E\left[\left\{4h^{4k}+4X_{\max}^{4k}+64(h^{2k}+X_{\max}^{2k})K^{2k}\right\}\right]n\frac{3}{2n^4}}\nonumber\\
		\stepb{=}& O_{c_1, c_2}\left(n^{-3/2}[((c_3+c_4s)\log n + k)^{2k} + (2s + 1)^{2k}(2k)^{2k}(1 + (\log n)^{2k - \frac 12}]\right)
	\end{align*}
	where (a) follows from \cite[(44)]{JPW24} and (b) follows from the moment bounds on $X_{\max}$ in \prettyref{lmm:moment}. Plugging this back into \prettyref{eq:SubE_npmle_regret_t1} and then \prettyref{eq:SubE_npmle_regret}, we obtain \[\Regret_{\pi,k,n}(\hat{T})=O_{s,k}\left(\frac1n\left(\log n\right)^{2k+1}\right).\]
\end{proof}

\subsection{ERM method}
The first step is to show that the Bayes estimator, $\tpik$, is monotone, 
which allows us to take the ERM w.r.t. the monotone function class $\mathcal{F}_{\uparrow}$. 
\begin{lemma}\label{lmm:bayes-monotone}
	$\tpik$ is a monotone function. 
\end{lemma}
\begin{proof}
	By the Cauchy-Schwarz inequality, we have \[\frac{\int e^{-\theta}\theta^{x+2}d\pi(\theta)}{\int e^{-\theta}\theta^{x+1}d\pi(\theta) } \ge \frac{\int e^{-\theta}\theta^{x+1}d\pi(\theta)}{\int e^{-\theta}\theta^{x}d\pi(\theta) }, \qquad \forall x = 0, 1, 2, \cdots.\]
	Iterating this inequality $k$ times, we get 
	\[
	\frac{\int e^{-\theta}\theta^{x + k + 1}d\pi(\theta)}{\int e^{-\theta}\theta^{x+k}d\pi(\theta) } \ge \frac{\int e^{-\theta}\theta^{x+1}d\pi(\theta)}{\int e^{-\theta}\theta^{x}d\pi(\theta) }
	\qquad 
	\Rightarrow 
	\qquad 
	\frac{\int e^{-\theta}\theta^{x + k + 1}d\pi(\theta)}{\int e^{-\theta}\theta^{x + 1}d\pi(\theta) } \ge \frac{\int e^{-\theta}\theta^{x + k}d\pi(\theta)}{\int e^{-\theta}\theta^{x}d\pi(\theta) }
	\]
	Thus, we find that
	\[
	\tpik(x + 1) = \frac{(x + 2)_k f_{\pi}(x + k + 1)}{f_{\pi}(x + 1)} = 
	\frac{\int e^{-\theta}\theta^{x+k+1}d\pi(\theta)}{\int e^{-\theta}\theta^{x + 1}d\pi(\theta)}
	\ge \frac{\int e^{-\theta}\theta^{x+k}d\pi(\theta)}{\int e^{-\theta}\theta^{x}d\pi(\theta)}
	=\tpik(x)\,,
	\]
	as desired. 
\end{proof}
Now we recall our ERM-based estimator as defined in \prettyref{eq:erm_obj_mon}. 
Note that the minimizer is not unique, since $\hat{T}$ is only uniquely determined for values that appear in our empirical expectation;  we choose to take $\hat{t}$ which is a step function that can only change at values where it is determined, and also where $\hat{t}(x) = \hat{t}(X_{\max})$ for all $x > X_{\max}\triangleq \max\{X_1, \cdots, X_n\}$. 
An explicit solution to \prettyref{eq:erm_obj_mon} may be calculated via \cite[Lemma 1]{JPTW23}. We now show that our empirical estimator is always bounded by $X_{\max}^k$, which will help us bound the complexity of our function class.

\begin{lemma}\label{lmm:f_erm_max}
	Let $\Termk$ be the estimator defined in \prettyref{eq:erm_obj_mon}. Then $\max_{x\in\mathbb{Z}_+} \Termk(x) \le X_{\max}^k.$
\end{lemma}
\begin{proof}
	As $\Termk$ is monotonic, and $\Termk(x) = \Termk(X_{\max})$ for all $x > X_{\max}$, it is sufficient to bound $\Termk(X_{\max})$. By \cite[Lemma 1]{JPTW23}, there exists some $i^*$ for which 
	\[\Termk(X_{\max}) \stepa{=} \frac{\sum_{i=i^*}^{X_{\max}}(i+1)_kN(i+k)}{\sum_{i=i^*}^{X_{\max}}N(i)} = \frac{\sum_{i=i^*+k}^{X_{\max}}(i-k+1)_k N(i)}{\sum_{i=i^*}^{X_{\max}}N(i)} 
	\le (X_{\max} - k + 1)_k\le X_{\max}^k\] 
	where (a) follows from a change of indices in the numerator, 
	and that $N(x) = 0$ for all $x>X_{\max}$.
\end{proof}

Next, we consider an even more restrictive class in the case where our prior is bounded (i.e. $\pi\in \mathcal{P}([0, h])$). 
In this case we see that $\theta^k\in [0, h^k]$ at all times, 
so we may now consider the following class of functions instead: 
\begin{definition}\label{defn:monotone_clipped}
	The clipped monotone class $\mathcal{F}_{\uparrow, [a, b]}$ is defined as follows:
	\[
	\mathcal{F}_{\uparrow, [a, b]}= \{\hat{t}:\bbZ_+\to\bbR_+: \hat{t}\in \mathcal{F}_{\uparrow}, a\le \hat{t}(x)\le b, \forall x\in\bbZ_+\}
	\]
\end{definition}
Now we may define the following: 
\begin{align}\label{eq:erm_obj_clipped}
	\Termclipped{k, [a, b]}(X_1^n) = (\hat{t}(X_1), \cdots, \hat{t}(X_n)), 
	\qquad 
	\hat{t} \in \argmin_{\hat{t}\in \mathcal{F}_{\uparrow, [a, b]}}\sum_{i=1}^n \left[\hat{t}(X_i)^2-2(X_i - k + 1)_k \hat{t}(X_i - k)\right]
\end{align}
For brevity, we denote 
$\mathcal{F}_{\uparrow, b}\triangleq \mathcal{F}_{\uparrow, [0, b]}$ 
and 
$\Termclipped{k, b}\triangleq \Termclipped{k, [0, b]}$. 
Our ERM estimator in $\mathcal{P}([0, h])$ is then 
$\Termclipped{k, h^k}$. 

We now demonstrate that $\Termclipped{k, h^k}$ has an efficient computation just like $\Termk$.
In fact, the following lemma shows that the ERM of the clipped monotone class is simply the clipped version of the unclipped monotone class, 
which we defer the proof to \prettyref{app:erm}. 

\begin{lemma}\label{lmm:thetaerm_clipped}
	The clipped ERM for each $a < b$ has the form given by 
	\[
	\Termclipped{k, [a, b]} 
	= \mathsf{clip}(\Termk, a, b)
	\]
	where the clipping function is given by 
	\[
	\mathsf{clip}(x) = 
	\begin{cases}
		a & x \le a\\
		x & a\le x \le b \\
		b & x\ge b\\
	\end{cases}
	\]
\end{lemma}

We now present the main result of this section. 
\begin{theorem}\label{thm:erm}
	The ERM estimator of $\mathbb{E}[\theta^k|X=x]$ satisfies the following regret bounds:
	\begin{enumerate}
		\item \[\sup_{\pi\in\calP([0,h])}\Regret_{\pi,k,n}(\Termclipped{k, h^k}) \le (A + Bh)^{2k + 1}\frac 1n(\frac{\log n}{\log \log n})^{k + 1} 
		+ \frac{\left[2(Ah + Bh^2)(\frac{\log n}{\log \log n})(1 + 2k)\right]^k}
		{n^3}\]
		\item \[\sup_{\pi\in\subexpo(s)}\Regret_{\pi,k,n}(\Termk) \le O_{s, k}\left(\frac{(\log n)^{2k+1}}{n}\right).\]
	\end{enumerate}
	Here $A, B$ are absolute constants. 
\end{theorem}

In order to prove these bounds, we use the following lemma, which is a generalization of \cite[Theorem 3]{JPTW23} that allows us to bound the regret using Rademacher random variables.
\begin{lemma}\label{lmm:erm_bound_regret}
	Let $\calF$ be a convex function class that contains the Bayes estimator $\tpik$. Let $X_1,\dots,X_n$ be a sample drawn i.i.d. from $f_\pi$, $\epsilon_1,\dots,\epsilon_n$ an independent sequence of i.i.d. Rademacher random variables, and $\hat{t}_{\mathsf{ERM}}$ the corresponding ERM solution. 
	Denote the following operator $S$ on two univariate functions $t_1: \mathbb{R}\to \mathbb{R}$ and $t_2: \mathbb{R}\to\mathbb{R}$, and an input $x\in\mathbb{R}$ as follows: 
	\begin{equation}\label{eq:s_erm}
		S(t_1, t_2; x)
		= t_2(x)(t_1(x) - t_2(x)) - (x-k + 1)_k(t_1(x - k) - t_2(x - k))
	\end{equation}
	and similarly we may define $S_{\pi}$ based on a prior $\pi$, 
	as well as the empirical version $\hat{S} := \hat{S}_{X_1, \cdots, X_n}$. 
	\begin{equation}\label{eq:s_erm_exp}
		S_{\pi}(t_1, t_2) 
		= \mathbb{E}_{X\sim \text{Poi}\circ\pi}[S(t_1, t_2; X)]
		\qquad 
		\hat{S}(t_1, t_2)
		=\frac 1n \sum_{i=1}^n S(t_1, t_2, X_i)
	\end{equation}
	Then for any function class $\calF_{p_n}$ depending on the empirical distribution $p_n^{\sfem} = \frac1n \sum_{i=1}^n \delta_{X_i}$ that includes $\hat{t}_{\mathsf{ERM}}$ and $\tpik$, we have 
	\[\Regret_{\pi,k, n}(\hat{t})\le \frac3n R_1(n) + \frac2n R_2(n)\] 
	where 
	\begin{align}\label{eq:R1}
		R_1(n) = \E\Biggl[&\sup_{\hat{t}\in\calF_{p_n}\cup\calF_{p_n'}}\sum_{i=1}^n(\epsilon_i - \frac16)(\tpik(X_i)-\hat{t}_{\mathsf{ERM}}(X_i))^2\Bigg]
	\end{align}
	and
	\begin{align}\label{eq:R2}
		R_2(n) = \E\Biggl[&\sup_{\hat{t}\in\calF_{p_n}\cup\calF_{p_n'}}\sum_{i=1}^n\bigg\{2\epsilon_i
		S(\hat{t}, \tpik; X_i) - \frac14(\tpik(X_i)-\hat{t}_{\mathsf{ERM}}(X_i))^2\bigg\}\Bigg]
	\end{align}
	where $\calF_{p_n'}$ is defined with an independent copy of $X_1, \dots, X_n$.
\end{lemma}

\begin{proof}
	Recall the definition of $R_{\pi}$ and $\hat{R}$ as defined in 
	\prettyref{eq:erm_obj} and \prettyref{eq:erm_obj_emp}. 
	Now $(1 - \lambda)\hat{t}_{\mathsf{ERM}} + \lambda\hat{f}\in\mathcal{F}$ for any $\lambda\in (0, 1)$ due to the convexity of $F$. 
	Since $\hat{t}_{\mathsf{ERM}}$ is a minimizer of $\hat{R}$, we have 
	\[
	0\ge \frac{\partial}{\partial\lambda}\hat{R}((1-\lambda)\hat{t}_{\mathsf{ERM}} + \lambda \hat{t})|_{\lambda=0}
	=-2\hat\E[(\hat{t}(X)-\hat{t}_{\mathsf{ERM}}(X))\hat{t}_{\mathsf{ERM}}(X) - (X - k + 1)_k(\hat{t}(X-k)-\hat{t}_{\mathsf{ERM}}(X-k))]
	\]
	Rearranging, we get 
	\begin{align}\label{eq:convexity}
		\hat R(\hat{t}) - \hat R(\hat{t}_{\mathsf{ERM}}) - \hat\E[(\hat{t}(X)-\hat{t}_{\mathsf{ERM}}(X))^2] \ge 0.
	\end{align}
	Next, we use the following observation on $S$: for every estimator $\hat{t}_1, \hat{t}_2$ we have  
	\[
	S_{\pi}(\hat{t}_1, \hat{t}_2)\triangleq R_{\pi}(\hat{t}_1) - R_{\pi}(\hat{t}_2) - \mathbb{E}_{X\sim \text{Poi}\circ \pi}[(\hat{t}_1(X) - \hat{t}_2(X))^2]
	\]
	and similarly, $\hat{S}(\hat{t}_1, \hat{t}_2) = \hat{R}(\hat{t}_1) - \hat{R}(\hat{t}_2) - \hat{\mathbb{E}}[(\hat{t}_1(X) - \hat{t}_2(X))^2]$. 
	Then given that $\Regret_{\pi,k,n}(\hat{t}) = R(\hat{t})-R(\tpik)$, using \prettyref{eq:convexity} gives 
	\[
	\Regret_{\pi,k,n}(\hat{t}_{\mathsf{ERM}})
	\le R(\hat{t}_{\mathsf{ERM}})-R(\tpik) + \mathbb{E}[\hat R(\tpik) - \hat R(\hat{t}_{\mathsf{ERM}}) - \hat\E[(\tpik(X)-\hat{t}_{\mathsf{ERM}}(X))^2]]:= I_1 + I_2
	\]
	where: 
	\[
	I_1 = S_{\pi}(\hat{t}_{\mathsf{ERM}}, \tpik) - \mathbb{E}[\hat{S}(\hat{t}_{\mathsf{ERM}}, \tpik)]
	- \frac 14 (\E[(\tpik(X)-\hat{t}(X))^2] + \hat\E[(\tpik(X)-\hat{t}(X))^2])
	\]
	\[
	I_2 = \frac54\E[(\tpik(X)-\hat{t}(X))^2]-\frac74\hat\E[(\tpik(X)-\hat{t}(X))^2]
	\]
	Define, now, $\epsilon_1, \cdots, \epsilon_n$ independent Rademacher symbols. 
	Then using the symmetrization result in \cite[Lemma 3]{JPTW23}, 
	$I_1\le \frac 2n R_2(n)$ and $I_2\le \frac 3n R_1(n)$. Combining these two upper bounds proves the lemma.
\end{proof}

In order to effectively bound the Rademacher complexity, 
we leverage the fact that $\thetaerm$ is bounded by $X_{\max}^k$, 
and that $\Termclipped{k, h^k}$ is bounded by $h^k$. 
We now denote the following function class bounded by a given monotone,  data-dependent function $r$. 
\[
\mathcal{F}_{r} \triangleq 
\left\{\hat{t} : \hat{t}\text{ is monotone}, \hat{t}(x)\le r(x), \forall x\in\bbZ_+\right\}. 
\]
We will specify the choice of $r$ later on. 
Subsequently, we will apply 
\prettyref{lmm:erm_bound_regret} with $\calF_{p_n} = \mathcal{F}_r$ 
and $\calF_{p_n'} = \mathcal{F}_{r}'$. 
Note that we allow this function class to depend on $\tpik$ 
which is assumed to be unknown, but we can still use it for our theoretical analysis. 

For the rest of the proof, 
we consider a slightly generalized version of \prettyref{eq:R1} and \prettyref{eq:R2}:
\begin{equation}\label{eq:r1_alt}
	R_1(b,n) = \E\Biggl[\sup_{\hat{t}\in\calF_{r}\cup\calF_{r}'}\sum_{i=1}^n(\epsilon_i-\frac1b)(\tpik(X_i)-\hat{t}(X_i))^2\Bigg]
\end{equation}
and
\begin{align}\label{eq:r2_alt}
	R_2(b,n) = \E\Biggl[&\sup_{\hat{t}\in\calF_{r}\cup\calF_{r}'}\sum_{i=1}^n\bigg\{2\epsilon_i(S(\hat{t}, \tpik; X_i))-\frac1b(\tpik(X_i)-\hat{t}(X_i))^2\bigg\}\Bigg].
\end{align}
where $r'$ is the monotone function determined by an independent copy of $X_1, \cdots, X_n$.

Next, under suitable tail and moment bounds on the distribution $f_{\pi}$ and moment bounds on the maximum sample, we prove bounds on the expressions $R_1$ and $R_2$.
\begin{lemma}\label{lmm:erm_bound_T}
	Let $\pi\in\calP([0,h])$ where $h$ is a constant or $s\log n$ for some constant $s$. Let $M := M(n,h) > \max\{h, k\}$ be such that 
	\begin{itemize}
		\item $\sup_{\pi\in\calP[0,h]} \PP_{X\in f_{\pi}}[X>M] \le \frac1{n^7}$.
		\item For $X_i\overset{i.i.d.}{\sim}f_{\pi}, \E\left[X_{\max}^q\right] \le c_1(q)M^{q}$ for $q \le 1, \cdots, 4k$ and constant $c_1$.
		\item For $X_i\overset{i.i.d.}{\sim}f_{\pi}, \E\left[r(X_{\max})^{q}\right] \le c_2(k, q)L^{q}$ for $q \le 4$ and constant $c_2$.
	\end{itemize}
	Then $R_1, R_2$ as defined in \prettyref{eq:r1_alt} and \prettyref{eq:r2_alt} have the following bound: 
	\[R_1(b,n) \le c_0(b)\left(\max\{1, h^{2k}\}M + c_2(k, 2)L^2\right); \] 
	\[R_2(b,n) \le c_0(b)h^{2k}(2^k + c_1(1)M) + (2h)^kM^{k + 1} + Mh^kr(M) + k!(h^k + 1)\]\[  + \frac{h^k\sqrt{c_1(2k)}M^k + \sqrt{\bbE[X_{\max}^{2k} r(X_{\max})^2]} + \sqrt{c_1(2k)c_2(k, 2)}M^kL}{n^2}.\]
\end{lemma}
The proof of this lemma is in \prettyref{app:erm}. Just as in \prettyref{lmm:rob_help}, \prettyref{lmm:erm_bound_T} provides a regret bound on a bounded prior. We again require \prettyref{lmm:truncate_regret} to extend our results to a subexponential prior.

We now combine these lemmas to prove the regret bounds as claimed in \prettyref{thm:erm}. 

\begin{proof}[Proof of \prettyref{thm:erm}]
	We first consider $\pi\in\mathcal{P}([0, h])$, 
	where $\Termclipped{k, h^k}$ has individual components in $\mathcal{F}_{\uparrow, h^k}$, 
	and therefore we may make the following choices:
	\begin{itemize}
		\item $M = \max\{c_1, c_2h\}\frac{\log n}{\log\log n}$ for some constants $c_1, c_2$ due to \cite[Lemma 10 and 12]{JPTW23}; 
		
		\item $c_1(q) = 2^{q}(1 + q!)$; 
		
		\item $r$ the constant function $r(x)\equiv h^k$, $L = h^k$ and also $c_2(k, q) = 1$. 
	\end{itemize}
	Then by \prettyref{lmm:erm_bound_regret} and \prettyref{lmm:erm_bound_T}, we have 
	\begin{flalign*}
		n\Regret_{\pi, k, n}(\Termclipped{k, h^k})
		&\lesssim (A + Bh)^{2k + 1}\left(\frac{\log n}{\log \log n}\right)^{k + 1} 
		+ \frac{h^k(A + Bh)^k(\frac{\log n}{\log \log n})^k(2^k\sqrt{1 + (2k)!})}
		{n^2}
		\\&\le 
		(A + Bh)^{2k + 1}\left(\frac{\log n}{\log \log n}\right)^{k + 1} 
		+ \frac{\left[2(Ah + Bh^2)(\frac{\log n}{\log \log n})(1 + 2k)\right]^k}
		{n^2}
	\end{flalign*}
	for some absolute constants $A, B$. 
	Note that we used $\sqrt{1 + (2k)!}\le 1 + \sqrt{(2k)!}\le (2k)^k$.
	Therefore the regret is indeed $O_{h, k}\left(\frac{1}{n}\left(\frac{\log n}{\log\log n}\right)^{k + 1}\right)$ for $k\ge 2$. 
	
	For subexponential $\pi$, we combine \prettyref{eq:moment_bound} to obtain the following:
	\begin{itemize}
		\item By \cite[Lemma 11 and 12]{JPTW23}, we may choose $M = \max\{a, bs\}\log n$ for some absolute constants $a, b > 0$. 
		
		\item By \prettyref{eq:moment_bound}, we have 
		\[\mathbb{E}[X_{\max}^{q}] \le 
		\frac{4(\log n)^{q} + \frac 32 q!\max(1, (\log n)^{q - 1})}{(\log (1 + \frac{1}{2s}))^{q}}
		\le (2s + 1)^{q}\cdot 6q! [1 + (\log n)^{q}]
		\]
		which allows us to take $c_1(q) = 6q!$. 
		
		\item We recall that by \prettyref{lmm:f_erm_max} we may take $r(X_{\max})\le X_{\max}^{k}$, 
		i.e.$\mathbb{E}[r(X_{\max})^{q}]\le \mathbb{E}[X_{\max}^{kq}]
		\le (2s + 1)^{kq}\cdot 6(kq)! [1 + (\log n)^{kq}]$ and therefore we may take 
		$L = M^k = (a + bs)^k(\log n)^k$ and $c_2(k, q) = 6(kq)!$. 
	\end{itemize}
	
	By \prettyref{lmm:truncate_regret}, it suffices to bound $\Regret_{\pi_{c_1s\log n},k}$, 
	which then gives the existence of constants $A, B$ such that 
	$R_1(b, n)\le c_0(b)[(A+ Bs)^{2k+1}(\log n)^{2k + 1} + (2k!)[(A+ Bs)^{2k+1}(\log n)^{2k}]$ and 
	$R_2(b, n)\le 
	c_0(b)[(A+ Bs)^{2k+1}(\log n)^{2k + 1} + k!(A+ Bs)^{k}(\log n)^{k} + \frac{(4k)!(A+ Bs)^{2k}(\log n)^{2k}}{n^2}]
	$
	where the last inequality uses $(4k)! \ge (2k)!^2$. 
	Then by \prettyref{lmm:erm_bound_regret} and \prettyref{lmm:erm_bound_T}, the regret is $O_k\left(\frac{\max\{1,s^{2k+1}\}(\log n)^{2k+1}}{n}\right).$
\end{proof}

\section{Proofs for upper bound (normal means model): $\ell(\theta) := \theta^k$}\label{sec:upper_bound_gsn}

The main ingredient of the proof is establishing a regret bound given the Hellinger distance between the mixture densities $f_{\pi}$ and $f_{\hat{\pi}}$. For the mean estimation task, 
this was done in \cite[Theorem 3]{JZ09} and later refined in \cite[Theorems 1, 3]{chen2026sharp}. 

\begin{theorem}\label{thm:normal-hellinger-to-regret}
	\begin{enumerate}
		\item Let $h > 0$, and suppose $\pi, \hat{\pi}\in \mathcal{P}([-h, h])$. Then there exists a constant $C := C(h, k)$ such that 
		\[
		\Regret_{\pi, k, n}(\hat{t}_{\hat{\pi}, k})\le C\epsilon^2\left(\frac{\log \frac{1}{\epsilon}}{\log \log \frac{1}{\epsilon}} + k\right)^k
		\]where $\epsilon^2 = H^2(f_{\pi}, f_{\hat{\pi}})$, and $C\le (A+Bh)^{2k+2}e^{2h^2}$ for some absolute constants $A, B$. 
		
		\item Suppose $\pi, \hat{\pi}\in\SubG(s)$ for some $s > 0$. Then there exists a constant $C := C(s, k)$ such that 
		\[
		\Regret_{\pi, k, n}(\hat{t}_{\hat{\pi}, k})\le C\epsilon^2\left(\log \frac{1}{\epsilon}\right)^k\left(\log \log\frac{1}{\epsilon}\right)^{2k}
		\]
		where $\epsilon^2 = H^2(f_{\pi}, f_{\hat{\pi}})$. 
	\end{enumerate}
\end{theorem}

The proofs generally proceed in the spirit of \cite{chen2026sharp}. 
Indeed, the main observation is to consider the following quantity: 
\[
\Delta_k(\pi_1, \pi_2) = 2\int_x \frac{(\hat{t}_{\pi_1, k}(x) f_{\pi_1}(x) - \hat{t}_{\pi_2, k}(x) f_{\pi_2}(x))^2}{f_{\pi_1}(x) + f_{\pi_2}(x)}dx
\]
Define also the following like \cite[(18), (19)]{chen2026sharp}, 
where we recall that $\varphi(x) = \frac{1}{\sqrt{2\pi}}e^{\frac{-x^2}{2}}$ denotes the standard normal density.
\[
g_{\pi_1, \pi_2} = \frac{f_{\pi_1} - f_{\pi_2}}{\varphi}
\qquad 
w_{\pi_1, \pi_2} = \frac{2\varphi^2}{(f_{\pi_1}+f_{\pi_2})}
\]
Our main contribution is to adapt \cite{chen2026sharp} by noting that 
\cite[(21)]{chen2026sharp} can be generalized into the following nice form. 

\begin{lemma}\label{lmm:gsn-alternate-form}
	For any probability measures $\pi_1, \pi_2$, we have: 
	$\Delta = \int (g^{(k)}(x))^2 w(x)dx := ||g^{(k)}||^2_{L_2(w)}$ 
	where we use the shorthand notation $\Delta := \Delta_k(\pi_1, \pi_2)$, 
	$g := g_{\pi_1, \pi_2}$ and $w := w_{\pi_1, \pi_2}$. 
\end{lemma}

\begin{proof}
	We define a function $t$ based on some measure $G$ and a function $u$ in the following manner: 
	\[
	t(x) = \mathbb{E}_G[u(\theta) \varphi(x - \theta)]
	= \int_{\theta} u(\theta) \varphi(x - \theta)dG(\theta)
	\]
	Such $t$ is infinitely differentiable in $x$, and therefore we have 
	\[
	\frac{d}{dx}\frac{t(x)}{\varphi(x)} = \frac{t'(x)}{\varphi(x)} - \frac{t(x)\varphi'(x)}{\varphi^2(x)}
	= \frac{t'(x) + xt(x)}{\varphi(x)}
	\]
	using the identity $\frac{d}{dx}\varphi(x) = -x\varphi(x)$. 
	Now, $t'(x) = \mathbb{E}_G[u(\theta)(\theta -x) \varphi(x - \theta)]$, 
	i.e. $t'(x) + xt(x) = \mathbb{E}_G[\theta u(\theta)\varphi(x - \theta)]$. 
	Therefore, a $k$-times differential iteration yields 
	\[
	\varphi(x)\frac{d^k}{dx^k}\frac{t(x)}{\varphi(x)} = \mathbb{E}_G[\theta^k u(\theta) \varphi(x - \theta)]
	\]
	Taking $u(\theta)\equiv 1$, we now have 
	\[
	g^{(k)} = \frac{d^k}{dx^k}\left(\frac{f_{\pi_1} - f_{\pi_2}}{\varphi}(x)\right)
	= \frac{\mathbb{E}_{\pi_1}[\theta^k \varphi(x - \theta)] - \mathbb{E}_{\pi_2}[\theta^k \varphi(x - \theta)]}{\varphi(x)}
	= \frac{\hat{t}_{\pi_1, k}(x) f_{\pi_1}(x) - \hat{t}_{\pi_2, k}(x) f_{\pi_2}(x)}{\varphi(x)}
	\]
	which concludes the proof. 
\end{proof}
To proceed, we establish the following three auxiliary results. 
The first two concern polynomial approximations of $g$, which generalize \cite[Lemma 10, Lemma 12]{chen2026sharp}. The third establishes a bound on the posterior moments of a normal means model. All proofs are deferred to \prettyref{app:normal_means_ub_proof}. 
\begin{lemma}\label{lmm:gsn_poly_bound}
	Let $\pi_1, \pi_2\in\mathcal{P}([-h, h])$. 
	Assume $(\pi_1+\pi_2)/2$ has zero mean. Set $q\ge 2eh^2+k$. Then there exists a degree $q$ polynomial $p$ such that 
	\[
	||g_{\pi_1, \pi_2}^{(k)}-p^{(k)}||^2_{L_2(w_{\pi_1, \pi_2})}\le 4h^{2k}e^{\frac{h^2}{2}}\left(\frac{eh^2}{q-k+1}\right)^{q-k}
	\]
\end{lemma}

\begin{lemma}\label{lmm:subgaussian-poly-bound}
	Let $\pi_1, \pi_2\in\SubG(s)$. Assume $(\pi_1+\pi_2)/2$ has zero mean. 
	Then there exist constants $c_1, c_2 > 0$ depending only on $s$ such that for every $q\ge k$ there exists a degree-$q$ polynomial $p$ such that $||g^{(k)} - p^{(k)}||^2_{L_2(w)} \le c_1\exp(-c_2(q - k))$. 
\end{lemma}

\begin{lemma}\label{lmm:moment-bound}
	Let $s > 0$. For any prior $\pi\in\SubG(s)$, 
	there exists a constant $C:= C(s, k)$ such that $|\hat{t}_{\pi, k}(x)|\le C(1+|x|^k)$. 
\end{lemma}

\begin{proof}[Proof of \prettyref{thm:normal-hellinger-to-regret}]
	Recall that 
	$
	\Regret_{\pi, k, n}(\hat{t}_{\hat{\pi}, k})
	= \mathbb{E}_{\pi}[(\hat{t}_{\pi, k} - \hat{t}_{\hat{\pi}, k})^2]
	$. 
	Using the general identity $a_1 - a_2 = (a_1+a_2)\frac{b_2-b_1}{b_1+b_2}+2\cdot \frac{a_1b_1-a_2b_2}{b_1+b_2}$ and substituting $a_1 = \hat{t}_{\pi, k}, a_2 = \hat{t}_{\hat{\pi}, k}$, $b_1 = f_{\pi}$, $b_2 = f_{\hat{\pi}}$,
	we have 
	\begin{flalign}\label{eq:gaussian-regret-expansion}
		\mathbb{E}[(\hat{t}_{\pi, k} - \hat{t}_{\hat{\pi}, k})^2]
		&=\mathbb{E}\left[\left((\hat{t}_{\pi, k} + \hat{t}_{\hat{\pi}, k})\frac{f_{\hat{\pi}} - f_{\pi}}{f_{\pi} + f_{\hat{\pi}}} + 2\cdot \frac{\hat{t}_{\pi, k}f_{\pi} - \hat{t}_{\hat{\pi}, k}f_{\hat{\pi}}}{f_{\pi} + f_{\hat{\pi}}}\right)^2\right]
		\nonumber\\&\le 2\mathbb{E}\left[\left((\hat{t}_{\pi, k} + \hat{t}_{\hat{\pi}, k})\frac{f_{\hat{\pi}} - f_{\pi}}{f_{\pi} + f_{\hat{\pi}}}\right)^2\right] + 8\mathbb{E}\left[\left(\frac{\hat{t}_{\pi, k}f_{\pi} - \hat{t}_{\hat{\pi}, k}f_{\hat{\pi}}}{f_{\pi} + f_{\hat{\pi}}}\right)^2\right]
		\nonumber\\& = 2\mathbb{E}\left[\left((\hat{t}_{\pi, k} + \hat{t}_{\hat{\pi}, k})\frac{f_{\hat{\pi}} - f_{\pi}}{f_{\pi} + f_{\hat{\pi}}}\right)^2\right] + 4\Delta_k(\pi, \hat{\pi})
	\end{flalign}
	We now consider each of the two cases. 
	
	\paragraph{Bounded prior.} In the case where $\pi, \hat{\pi}\in \mathcal{P}([-h, h])$, we have 
	\[
	\mathbb{E}\left[\left((\hat{t}_{\pi, k} + \hat{t}_{\hat{\pi}, k})\frac{f_{\hat{\pi}} - f_{\pi}}{f_{\pi} + f_{\hat{\pi}}}\right)^2\right]
	\stepa{\le} 4h^{2k} \mathbb{E}\left[\left(\frac{f_{\hat{\pi}} - f_{\pi}}{f_{\pi} + f_{\hat{\pi}}}\right)^2\right]
	\stepb{\le} 8h^{2k}\epsilon^2
	\]
	where (a) uses $|\hat{t}_{\pi, k}|, |\hat{t}_{\hat{\pi}, k}|\le h^k$, and (b) is via the expansion 
	\[
	\left(\frac{f_{\hat{\pi}} - f_{\pi}}{f_{\pi} + f_{\hat{\pi}}}\right)^2 
	= \frac{(\sqrt{f_{\pi}} - \sqrt{f_{\hat{\pi}}})^2(\sqrt{f_{\pi}} + \sqrt{f_{\hat{\pi}}})^2}{(f_{\pi} + f_{\hat{\pi}})^2}
	\le 2\frac{(\sqrt{f_{\pi}} - \sqrt{f_{\hat{\pi}}})^2}{f_{\pi} + f_{\hat{\pi}}}
	\le 2 \left(\sqrt{\frac{f_{\pi}}{f_{\hat{\pi}}}} - 1\right)^2
	\]
	It then suffices to show that $\Delta_k(\pi, \hat{\pi}) = O((A+Bh)^{2k+2}e^{2h^2}\epsilon^2\left(\frac{\log \frac{1}{\epsilon}}{\log \log \frac{1}{\epsilon}}\right)^k)$ for some absolute constants $A, B$. 
	Choose polynomial $p$ of degree $q\ge 2eh^2+k$ such that \prettyref{lmm:gsn_poly_bound} is fulfilled on $g := g_{\pi, \hat{\pi}}$. 
	Note that \prettyref{lmm:gsn_poly_bound} applies only to cases where $\frac{\pi + \hat{\pi}}{2}$ has mean zero. 
	Denote, now, $\frac{\mathbb{E}_{\pi}[\theta] + \mathbb{E}_{\hat{\pi}}[\theta]}{2}=\mu$. 
	We first apply this result on $(\pi', \hat{\pi}')$, defined as the pushforward of $\pi, \hat{\pi}$ under $\theta \to \theta - \mu$.
	Then by \prettyref{lmm:gsn-alternate-form}, we expand the following:
	\begin{flalign*}
		\Delta_k(\pi', \hat{\pi}')
		&=||g^{(k)}||_{L_2(w)}^2
		\\&\stepa{\le} 2||g^{(k)} - p^{(k)}||^2_{L_2(w)} + 2||p^{(k)}||^2_{L_2(w)}
		\\&\stepb{\le} 2||g^{(k)} - p^{(k)}||^2_{L_2(w)} + 2(2h+1)^{2k}(q-k+2)_k||p||^2_{L_2(w)}
		\\&\stepc{\le} 2||g^{(k)} - p^{(k)}||^2_{L_2(w)} + 4(2h+1)^{2k}(q-k+2)_k(||g-p||^2_{L_2(w)} + ||g||^2_{L_2(w)})
		\\&\stepd{\le} 8(2h)^{2k}e^{2h^2}(\frac{4eh^2}{q-k+1})^{q-k} + 
		16(2h+1)^{2k}(q-k+2)_k\left(\left(\frac{4eh^2}{q}\right)^q + \epsilon^2\right)
	\end{flalign*}
	where (a), (c) are due to triangle inequality and using $(a+b)^2\le 2(a^2+b^2)$, 
	(b) uses \cite[Lemma 9]{chen2026sharp} $k$ times on the polynomial $p$, and 
	where (d) uses \prettyref{lmm:gsn_poly_bound} and \cite[Lemma 10]{chen2026sharp} (but using constant $2h$ instead of $h$), 
	and also $||g||^2_{L_2(w)}\le 4\epsilon^2$. 
	By choosing a suitable absolute constant $c_1$ such that $q = \lceil c_1h^2 \frac{\log \frac{1}{\epsilon}}{\log \log \frac{1}{\epsilon}}+k\rceil$, 
	we can make both $(\frac{4eh^2}{q-k+1})^{q-k}$ and $(\frac{4eh^2}{q})^q$ to be at most $\epsilon^4$, 
	and $(q-k+2)_k\le (c_1h^2 \frac{\log \frac{1}{\epsilon}}{\log \log \frac{1}{\epsilon}}+k+1)^k$. 
	Note that the other constants 
	$8(2h)^{2k}e^{2h^2}$ and $(2h+1)^{2k}$ have the desired form (i.e. bounded by $(A_0+B_0h)^{2k+2}e^{2h^2}$ for some absolute constants $A_0, B_0>0$). 
	
	To apply results back to $(\pi, \hat{\pi})$, we note that $f_{\pi}$ and $f_{\hat{\pi}}$ is shift-invariant. 
	We now use the expansion $(\theta+\mu)^k = \sum_{j=0}^k \binom{k}{j}\theta^k \mu^{k-j}$ and $|\mu|\le h$ to obtain
	\begin{flalign*}
		\Regret_{\pi, k, n}(\hat{t}_{\hat{\pi}, k})
		&\le k\sum_{j=1}^k \binom{k}{j}^2\Regret_{\pi', j, n}(\hat{t}_{\hat{\pi'}, j})|\mu|^{2(k-j)}
		\\&\stepa{\le} k\sum_{j=0}^k 2^{2k}h^{2(k-j)}e^{2h^2}(A_0+B_0h)^{2j+2}(c_1h^2\frac{\log \frac{1}{\epsilon}}{\log\log \frac{1}{\epsilon}}+j)^{j}\epsilon^2
		\\&\stepb{\le} k^2 e^{2h^2}(2A_0+2B_0h)^{2k+2}(c_1h^2\frac{\log \frac{1}{\epsilon}}{\log\log \frac{1}{\epsilon}}+k)^{k}\epsilon^2
		\\&\stepc{\le} e^{2h^2}(4A_0+4B_0h)^{2k+2}(c_1h^2\frac{\log \frac{1}{\epsilon}}{\log\log \frac{1}{\epsilon}}+k)^{k}\epsilon^2
	\end{flalign*}
	where in (a) we used $\binom{k}{j} < 2^k$ and in (b) the factor $h^{2(k-j)}\le (A_0+B_0h)^{2j}$ for $B_0\ge 1$ and (c) 
	the crude inequality $k^2\le 4^k$ for all $k\ge 1$. 
	
	\paragraph{Subgaussian prior.} In the case where $\pi, \hat{\pi}\in \SubG(s)$, 
	we obtain constants $c_1, c_2, c_3$ depending only on $s$ such that for any threshold term $T > 1$,
	the first term in \prettyref{eq:gaussian-regret-expansion} becomes 
	\begin{flalign*}
		\mathbb{E}\left[\left((\hat{t}_{\pi, k} + \hat{t}_{\hat{\pi}, k})\frac{f_{\hat{\pi}} - f_{\pi}}{f_{\pi} + f_{\hat{\pi}}}\right)^2\right]
		&\stepa{\le} \int \frac{(\hat{t}_{\pi, k} + \hat{t}_{\hat{\pi}, k})^2\left(f_{\pi} - f_{\hat{\pi}}\right)^2}{f_{\pi} + f_{\hat{\pi}}}dx 
		\\&\stepb{\le} c_1\int (1 + x^{2k})\frac{\left(f_{\pi} - f_{\hat{\pi}}\right)^2}{f_{\pi}(x) + f_{\hat{\pi}}(x)}dx
		\\&\le c_1(1+T^{2k})\int_{|x| \le T} \frac{\left(f_{\pi} - f_{\hat{\pi}}\right)^2}{f_{\pi}(x) + f_{\hat{\pi}}(x)}dx + 2c_1\int_{|x| > T} x^{2k}\frac{\left(f_{\pi} - f_{\hat{\pi}}\right)^2}{f_{\pi}(x) + f_{\hat{\pi}}(x)}dx
		\\&\stepc{\le} 2c_1(1+T^{2k})\epsilon^2 + 2c_1\int_{|x| > T} x^{2k}e^{-c_2x^2}dx
		\\&\stepd{\le} 2c_1(1+T^{2k})\epsilon^2 + 2c_1T^{2k}e^{-c_3T^2}
	\end{flalign*}
	where (a) is due to the crude inequality $\frac{f_{\pi}}{f_{\pi}+f_{\hat{\pi}}}\le 1$, 
	and (b) is by \prettyref{lmm:moment-bound}, 
	(c) is by the general identity $\frac{(a-b)^2}{a+b}\le 2(\sqrt{a}-\sqrt{b})^2$, and (d) is tail bound arising from the assumption that $\pi, \hat{\pi}\in \SubG(s)$. 
	We can then choose $T = c\sqrt{\log \frac{1}{\epsilon^2}}$ for some $c: c(s)$ so that both quantities $(1+T^{2k})\epsilon^2$ and $T^{2k}e^{-c_3T^2}$ are bounded by $\epsilon^2\left(\log \frac{1}{\epsilon}\right)^{k}$. 
	
	Next, we bound $\Delta_k(\pi, \hat{\pi})$ using \prettyref{lmm:gsn-alternate-form}. 
	We again note that \prettyref{lmm:subgaussian-poly-bound} requires $\frac{\pi+\hat{\pi}}{2}$ to be zero mean. In a similar manner as the bounded prior case, we consider $(\pi', \hat{\pi}')$ as the pushforward of $(\pi, \hat{\pi})$ under $\theta \to \theta - \mu$ where $\mu = -\mathbb{E}_{\frac{\pi+\hat{\pi}}{2}}[\theta]$. 
	Note that both $\mu, \mu'$ are still $s$'-subgaussian for some $s'$ depending only on $s$. Thus for some constants $c_1, c_2, c_3$ depending on $s$ we have 
	\begin{flalign*}
		\Delta_k(\pi', \hat{\pi}')
		&=||g^{(k)}||_{L_2(w)}^2
		\\&\stepa{\le} 2||g^{(k)} - p^{(k)}||^2_{L_2(w)} + 2||p^{(k)}||^2_{L_2(w)}
		\\&\stepb{\le} 2||g^{(k)} - p^{(k)}||^2_{L_2(w)} + 4c_3(q - k + 2)_k(\log q)^{2k}(||g||^2_{L_2(w)} + ||g - p||^2_{L_2(w)})
		\\&\stepc{\le} c_1(\exp(-c_2(q-k))) + 4(q - k + 2)_k(\log q)^{2k}(4\epsilon^2 + c_1\exp(-c_2q))
	\end{flalign*}
	where (a) is triangle inequality and the identity $(a+b)^2\le 2(a^2+b^2)$, (b) is by applying \cite[Proposition 13]{chen2026sharp} $k$ times to obtain 
	\[||p^{(k)}||_{L_2(w)}\le C\sqrt{(q - k + 2)_k}(\log q)^k ||p||_{L_2(w)}\] and (c) is due to \prettyref{lmm:subgaussian-poly-bound}, \cite[Lemma 12]{chen2026sharp}, 
	and that $||g||_{L_2(w)}\le 4\epsilon^2$. Taking $q = \lceil k + c\log(\frac{1}{\epsilon^2})\rceil$ for a suitable choice $c$ gives 
	\[
	||g^{(k)}||_{L_2(w)}^2 \le O_{s, k}\left(\epsilon^2 \left(\log{\frac{1}{\epsilon}}\right)^k\left(\log\log{\frac{1}{\epsilon}}\right)^{2k}\right)
	\]
	as claimed. Like the bounded prior case, bounding $\Regret_{\pi, k, n}(\hat{t}_{\hat{\pi}, k})$ involves the steps
	\begin{flalign*}
		\Regret_{\pi, k, n}(\hat{t}_{\hat{\pi}, k})
		&\le k\sum_{j=1}^k \binom{k}{j}^2 \Regret_{\pi', j, n}(\hat{t}_{\hat{\pi}', j})|\mu|^{2(k-j)}
		\stepa{\le} k\sum_{j=1}^k 2^{2k}(\sqrt{2\pi}s)^{2(k-j)}c_{s, j}\epsilon^2\left(\log{\frac{1}{\epsilon}}\right)^j\left(\log\log{\frac{1}{\epsilon}}\right)^{2j}
	\end{flalign*}
	where (a) uses $\binom{k}{j}\le 2^k$ and that any $s$-subgaussian $\pi$ has $\mathbb{E}[|\theta|]\le \sqrt{2\pi s}$ (hence $|\mu|\le \sqrt{2\pi s}$. 
	The conclusion follows as all terms can be absorbed into the $O_{s, k}\left(\epsilon^2 \left(\log{\frac{1}{\epsilon}}\right)^k\left(\log\log{\frac{1}{\epsilon}}\right)^{2k}\right)$ form. 
\end{proof}

\begin{proof}[Proof of \prettyref{thm:gsn_bounds}]
	We first consider the case where $\pi\in \mathcal{P}([-h, h])$. 
	Consider, now, the function 
	\[
	\psi(\delta)=\delta\left(\frac{\log(1/\sqrt{\delta})}{\log\log(1/\sqrt{\delta})}\right)^k\qquad \psi(0)=0
	\]
	We have $\frac{d\log(\psi(\delta))}{d\delta} \ge \frac{1}{\delta}\left(1 - \frac{k}{\log(1/\delta)}\right)$, 
	i.e. $\psi(\delta)$ is increasing in $(0, e^{-k})$ for $x_0 = e^{-k}$. 
	Given that any $\hat{\pi}\in \mathcal{P}([-h, h])$ satisfies $\Regret_{\pi, k, n}(\hat{t}_{\hat{\pi}, k})\le 4h^{2k}$, 
	for some absolute constants $A, B$ we have 
	\begin{flalign*}
		\Regret_{\pi, k, n}(\hat{t}_{\hat{\pi}, k})
		&\le (A+Bh)^{2k+2}e^{2h^2}\mathbb{E}_{\pi}[\psi(H^2(f_{\pi}, f_{\hat{\pi}}))\indc{H^2(f_{\pi}, f_{\hat{\pi}})\le e^{-k}}]
		+4h^{2k}\mathbb{P}_{\pi}[H^2(f_{\pi}, f_{\hat{\pi}})> e^{-k}]
		\\&\stepa{\le} (A+Bh)^{2k+2}e^{2h^2}\psi(\mathbb{E}_{\pi}[H^2(f_{\pi}, f_{\hat{\pi}})])
		+4h^{2k}e^k\mathbb{E}_{\pi}[H^2(f_{\pi}, f_{\hat{\pi}})]
	\end{flalign*}
	where (a) is due to Jensen's inequality (using concavity of $\psi$ in $(0, x_0)$) and Markov's inequality. 
	We now use the result where the Birgé-Le Cam estimator \cite{birge1983approximation} satisfies $\sup_{\pi\in \mathcal{P}([-h, h])}\mathbb{E}_\pi[H^2(f_{\hat{\pi}}, f_{\pi})]\le C_h \frac{\log n}{n\log\log n}$ for some constant $C_h$ \cite[Theorem 18]{chen2026sharp}, 
	giving us the final bound of 
	\begin{flalign*}
		\Regret_{\pi, k, n}(\hat{t}_{\hat{\pi}, e^{-k}})
		&\le (A+Bh)^{2k+2}e^{2h^2}C_h \frac{1}{n}\left(\frac{\log n}{\log \log n}\right)^{k+1} + 4h^{2k}e^kC_h\frac{1}{n}\frac{\log n}{\log \log n}\nonumber\\
		&\le \frac 1n (A\sqrt{e}+B\sqrt{e}h)^{2k+2}e^{2h^2}C_h \left(\frac{\log n}{\log \log n}\right)^{k+1}
	\end{flalign*}
	For subgaussian prior, we may also proceed similarly by using $\psi(\delta) = \delta(\log(1/\sqrt{\delta}))^k(\log\log(1/\sqrt{\delta}))^{2k}$, 
	and then apply the result that 
	$\mathbb{E}[H^2(f_{\pi}, f_{\hat{\pi}_{\mathsf{NPMLE}}})]\le O_s\left(\frac{(\log n)^2}{n}\right)$ (\cite[Theorem 1]{JZ09}; \cite[Theorem 6]{polyanskiy2020self}). 
\end{proof}

\section{Proofs for upper and lower bounds: $\ell(\theta)$ smooth}\label{sec:smooth}
Having established the results for polynomial approximation, we now reduce \prettyref{thm:eb_cts} to estimation of finitely many posterior moments. Throughout this section, we will use $I(h)$ to denote $[0, h]$ for the Poisson model and $[-h, h]$ for the normal means model. 
\begin{lemma}\label{lmm:eb-cts-sub}
	\prettyref{thm:eb_cts} holds true provided there exist constants $a_1:=a_1(h), a_2:=a_2(h)$, absolute constants $a_3, a_4, a_5$, and estimators $\{\hat{T}_k\}_{k\ge 1}$ which estimate $\theta^k$, satisfying the following worst case total regret: 
	\begin{equation}\label{eq:regret-k-condition}
		\sup_{\pi\in\mathcal{P}(I(h))}\TotRegret_{\pi, k, n}(\hat{T}_k)
		\le a_1\cdot a_2^{2(k+1)}
		\left[(\frac{\log n}{\log \log n} + k)^{k + a_3} + \frac{1}{n^2}(\frac{\log n}{\log\log n})^k )(a_4+a_5k)^k\right]
	\end{equation}
\end{lemma}

\begin{proof}[Proof of \prettyref{thm:eb_cts}, assuming \prettyref{lmm:eb-cts-sub}.]
	We just need to verify the conditions for both the Poisson and normal means models: 
	\begin{itemize}
		\item Poisson model: \prettyref{sec:upper_bound} shows that $a_1$ can be taken as a constant, $a_2=A+Bh$ for some $A, B > 0$, 
		$a_3=1$, $a_4=1$ and $a_5=2$. 
		These choices uniformly cover all three estimators (Robbins, NPMLE, ERM). 
		
		\item Normal means model: 
		we use \cite[Theorem 18]{chen2026sharp} that $\mathbb{E}[H^2(f_{\hat{\pi}}, f_{\pi})]\le c_0 \frac 1n(\frac{\log n}{\log \log n})$ for $\hat{\pi}$ defined as the Birgé-Le Cam estimator. 
		Then we may use 
		\prettyref{thm:normal-hellinger-to-regret} and get $a_1=e^{2h}\cdot c_0$, $a_2=A+Bh, a_3=1, a_4=a_5=0$. 
		
	\end{itemize}
\end{proof}

Before proceeding with proving \prettyref{lmm:eb-cts-sub}, we quote a few lemmas regarding polynomial approximation. 
\begin{lemma}[Chapter 2.6, (9) in \cite{timan2014theory}]\label{lmm:poly_approx_coefs}
	Given a polynomial $p_k(\theta) = \sum_{m = 0}^k c_m\theta^m$ with coefficients 
	$c_m\in\bbR$ for $m = 0, 1, \cdots, k$. 
	Then 
	\[
	|c_m|\le \frac{k^m}{m!}\max_{|\theta|\le 1} |p_k(\theta)|\le e^k \max_{|\theta|\le 1} |p_k(\theta)|. 
	\]
\end{lemma}

\begin{lemma}[Theorem 8 in \cite{jackson31}]\label{lmm:deriv_poly}
	Let $g$ be $p$ times differentiable in $[-h, h]$ with modulus of continuity \[\omega(\delta) = \sup_{x, y\in [-h, h]: |x-y|\le\delta} |g^{(p)}(x) - g^{(p)}(y)|.\]
	Then for every $k > p$, there exists a polynomial $p_k$ of degree $k$ such that for $x\in [-h, h]$ we have 
	\[
	|g(x) - p_k(x)|\le \frac{L_ph^p}{k^p} \omega\left(\frac{2h}{k - p}\right)
	\]
	where $L_p$ is an absolute constant depending only on $p$. 
\end{lemma}

\begin{proof}[Proof of \prettyref{lmm:eb-cts-sub}]
	For each polynomial degree $k$, decompose $\ell$ into $\ell(\theta) = p_k(\theta) + r_k(\theta)$, 
	where $p_k(\theta) = \sum_{m = 0}^{k} c_m\theta^m$ is a degree-$k$ approximation of $\ell$. 
	The Bayes estimator $\tpiell$ has the following form, 
	thanks to the linearity of expectation: 
	\begin{align*}
		\tpiell(x) 
		&= \bbE_{\pi}[\ell(\theta) \mid X = x]
		= \bbE_{\pi}\left[\sum_{m = 0}^{k} c_m\theta^m + r_k(\theta) \mid X = x\right]
		\\&= \sum_{m = 0}^{k} c_m\bbE_{\pi}\left[\theta^m | x\right] + r_{\pi, k}(x)
		= \sum_{m = 0}^{k} c_m \hat{t}_{\pi, m}(x) + r_{\pi, k}(x)
	\end{align*}
	where $r_{\pi, k}(x)\triangleq \bbE_{\pi}\left[r_k(\theta) | X = x\right]$. 
	
	Now, suppose that for each $m\ge 1$ we have an estimator $\hat{T}_m$ of $\theta^m$, each of which satisfies \prettyref{eq:regret-k-condition}. 
	Define also $\hat{T}_0\equiv 1$.
	We now consider the following estimate for $\ell$: 
	\[
	\hat{T}_{\ell, k} \triangleq \sum_{m = 0}^k c_m \hat{T}_m
	\]
	(given that $\hat{T}$ is $n$-letter to $n$-letter function, 
	we set the addition to be done component-wise). 
	Then we may bound the total regret of the estimator $\hat{T}_{\ell, k}$ in terms of 
	the total regret of $\hat{T}_{m}$ as follows: 
	\begin{flalign}\label{eq:regret_full_bound}
		\TotRegret_{\pi, \ell,n}(\hat{T}_{\ell, k})
		&=\bbE[||\hat{T}_{\ell, k} - \Tpiell||^2]
		\nonumber\\
		&=\sum_{i=1}^n \bbE[(\sum_{m = 0}^k c_m (\hat{T}_m(X_1^n)_i - \hat{t}_{\pi, m}(X_i))
		-r_{\pi, k}(X_i)
		)^2]
		\nonumber\\
		&\le \sum_{i=1}^n 2\bbE\left[\left(\sum_{m = 1}^k c_m (\hat{T}_m(X_1^n)_i - \hat{t}_{\pi, m}(X_i))\right)^2\right]
		+2\bbE[r_{\pi, k}(X_i)^2]
		\nonumber\\
		&\le 2(k + 1)\bbE\left[\sum_{m = 1}^k (c_m ||\hat{T}_m(X_1^n) - \hat{T}_{\pi, m}(X_1^n)||)^2\right]
		+2n\sup_{\theta\in I(h)} r_k(\theta)^2
		\nonumber\\
		&= 2(k + 1)\sum_{m = 1}^k c_m^2\TotRegret_{\pi, m,n}(\hat{T}_{m})
		+2n\sup_{\theta\in I(h)} r_k(\theta)^2
	\end{flalign}

	Next, note that $p_k(\theta) = \ell(\theta) - r_k(\theta)$, 
	so $|p_k(\theta)|\le \sup |\ell(\theta)| + \sup |r_k(\theta)| \le M + \sup |r_k(\theta)|$, 
	where $M := \sup_{\theta\in I(h)} |\ell(\theta)|$. 
	By \prettyref{lmm:poly_approx_coefs} applied to $p_k(\frac{\theta}{h})$, 
	for each $m$ we have 
	\[
	\frac{|c_m|}{h^m}\le \frac{k^m}{m!}\max_{|\theta|\le h} |p_k(\theta)|
	\]
	and therefore
	\begin{equation}\label{eq:poly_coef_bound}
		|c_m|\le \frac{k^mh^m}{m!}\max_{|\theta|\le h}|p_k(\theta)|
		\le e^k\max(1, h^k)(M + \sup |r_k(\theta)|). 
	\end{equation}
	We now take $k = c\frac{\log n}{\log \log n}$ where $c < 1$ is to be determined later. 
	We may also consider $n$ sufficiently large such that $\sup_{\theta\in I(h)} |r_k(\theta)|\le 1$. 
	This means we may continue from the last display of 
	\prettyref{eq:regret_full_bound} to establish 
	\begin{flalign*}
		\TotRegret_{\pi, \ell, n}(\hat{T}_{\ell, k})&\le  2(k + 1)\sum_{m = 1}^k c_m^2\TotRegret_{\pi, m, n}(\hat{T}_{m})
		+2n\sup_{\theta\in I(h)} r_k(\theta)^2
		\\&\le 2(k + 1) e^{2k}\max(1, h^{2k})(M + 1)^2 
		\sum_{m = 1}^k Reg(m) + 2n\sup_{\theta\in I(h)} r_k(\theta)^2
		\\&\le 2k(k + 1) e^{2k}\max(1, h^{2k})(M + 1)^2 \sup_{1\le m\le k} Reg(m) + 2n\sup_{\theta\in I(h)} r_k(\theta)^2
	\end{flalign*}
	where we define $Reg(m) := \sup_{\pi\in\mathcal{P}(I(h))}\TotRegret_{\pi, m, n}(\hat{T}_{m})$.
	Now under this choice of $k$ we have the following asymptotic bound: 
	\[
	(\frac{\log n}{\log \log n} + k)^{k + a_3} \le n^{c + O(\frac{1}{\log \log n})}
	\qquad 
	(c_4+c_5k)^k \le n^{c + O(\frac{1}{\log \log n})}
	\]
	while other factors are at most $n^{O_h(\frac{1}{\log \log n})} = n^{o(1)}$. 
	Therefore we have $Reg(m)\le n^{c + O(\frac{1}{\log \log n})}$ for each $m$ with $1\le m\le k$. 
	For the following class of functions, then, it suffices to analyze $\sup_{\theta\in I(h)} r_k(\theta)^2$ 
	as this term is the limiting factor in bounding regret: 
	\begin{itemize}
		\item 
		$\ell^{(p)}\in \text{Lip}_{\alpha}(L)$ on $[-h, h]$ where $0 < \alpha \le 1$. 
		We now bound the modulus of convergence of $\ell^{(p)}$, which is 
		\[
		\omega\left(\frac{2h}{k - p}\right)
		= \sup_{|x - y|\le \frac{2h}{k - p}} |\ell^{(p)}(x) - \ell^{(p)}(y)|
		\le L\left(\frac{2h}{k - p}\right)^{\alpha}
		\]
		and by \prettyref{lmm:deriv_poly}, 
		there exists a polynomial $p_k$ that satisfies 
		$|r_k(\theta)|\le L'\frac{(2h)^\beta}{k^p(k - p)^{\alpha}}\le 2L'\frac{(2h)^\beta}{k^{\beta}}$ for all $k\ge 2p$. 
		Taking $k = c\frac{\log n}{\log \log n}$ would give the total regret bound 
		$O_{h,\beta, L}\left(n\left(\frac{\log \log n}{\log n}\right)^{2\beta}\right)$. 
		
		\item $\ell$ real analytic around neighbourhood of $\theta = 0$ 
		with radius of convergence $R>h$: 
		we may now pick $R'$ with $h < R' < R$ such that there exists $C_1$ where $|a_m|\le C_1(R')^{-m}$. 
		Then we have 
		\[
		|r_k(\theta)|\le \sum_{m = k + 1}^{\infty} C_1(R')^{-m}h^m = \frac{C_1(R'')^{k + 1}}{1 - R''}
		\]
		where $R''\triangleq \frac{h}{R'} < 1$. 
		Picking $\hat{T}_{\ell}\triangleq \hat{T}_{\ell, k}$ with $k = c \frac{\log n}{\log \log n}$, 
		we have 
		$\TotRegret_{\pi, \ell, n}(\hat{T}_{\ell, k}) \le n^{1 - c - o(1)} + \frac{C_1(R'')^{c \frac{\log n}{\log \log n} + 1}}{1 - R''} \le O\left(n^{1-\frac{2c'}{\log \log n}}\right)$ with $c'\triangleq c\log(R'/h) > 0$. 
	\end{itemize}
	
	Finally, we consider the regime where $\ell$ has its coefficients decaying very rapidly, 
	where $\ell(\theta)\triangleq \sum_{m\ge 0} c_m\theta^m$ with $|c_m|\le (D_1m)^{-D_2m}$ for some $D_1, D_2 > 0$. 
	We note the following: there exists a constant $c_0\triangleq c_0(D_1, D_2, h)$ such that for all $k$, 
	\[
	\sup_{\theta\in I(h)}r_k(\theta) \le \sum_{m > k} (D_1m)^{-D_2m}h^m \le c_0(D_1m)^{-D_2m}h^m. 
	\]
	
	For each $m\in O(\frac{\log n}{\log \log n})$, we now bound $c_m^2\TotRegret_{\pi, \ell, n}(\hat{T}_m)\le (D_1m)^{-2D_2m}Reg(m) := I_1(m) + I_2(m)$, where: 
	\[
	I_1(m) = (D_1m)^{-2D_2m}a_1a_2^{2(m+1)}\left(\frac{\log n}{\log \log n}+m\right)^{m + a_3}
	\]\[	
		I_2(m) = \frac{1}{n^2}(D_1m)^{-2D_2m}a_1a_2^{2(m+1)}\cdot (a_4+a_5m)^m \cdot \left(\frac{\log n}{\log \log n}\right)^{m}
	\]
	Given the bound on $m$, 
	$a_1D_1^{-2D_2m}a_2^{2m+2}\left(\frac{\log n}{\log \log n}+m\right)^{a_3} \le e^{O(\frac{\log n}{\log \log n})}\le n^{o(1)}$, 
	so $I_1(m)\le n^{o(1)} \cdot m^{- 2D_2m}(\frac{\log n}{\log \log n})^m$. 
	Similarly, $D_1^{-2D_2m}a_2^{2(m+1)}(a_4+a_5m)^m \le e^{O(\frac{\log n}{\log \log n})}\le n^{o(1)}$, 
	so $I_2(m)\le n^{-2+o(1)}m^{(1 - 2D_2)m}(\frac{\log n}{\log \log n})^m$. 
	
	Next, we continue from \prettyref{eq:regret_full_bound} to obtain that for $k = c\frac{\log n}{\log \log n}$, 
	\begin{flalign*}
		\Regret_{\pi, \ell, n}(\hat{T}_{\ell, k})&\le  2(k + 1)\sum_{m = 1}^k c_m^2\TotRegret_{\pi, m}(\hat{T}_{m})
		+2n\sup_{\theta\in I(h)} r_k(\theta)^2
		\\&\le 2k(k + 1)\sup_{1\le m\le k} (D_1m)^{-2D_2m}Reg(m) + 2n c_1^2(D_1k)^{-2D_2k}h^{2k}
		\\&\le 2k(k + 1)\sup_{1\le m\le k} n^{o(1)}\left(m^{- 2D_2}\frac{\log n}{\log \log n}\right)^m + n^{-2+o(1)}\left(m^{1 - 2D_2}\frac{\log n}{\log \log n}\right)^m
		\\&+2n c_1^2(D_1k)^{-2D_2k}h^{2k}
		\\&\stepa{\le} 2k(k + 1)\sup_{0\le c'\le c} n^{o(1)}\left((c')^{-2D_2}(\frac{\log n}{\log \log n})^{1 - 2D_2}\right)^{c'\frac{\log n}{\log \log n}} 
		\\&+ n^{-2+o(1)}\left((c')^{1 - 2D_2}(\frac{\log n}{\log \log n})^{2 - 2D_2}\right)^{c'\frac{\log n}{\log \log n}} 
		+2nc_1^2(D_1k)^{-2D_2k}h^{2k}
	\end{flalign*}
	where in (a) we parametrize $m = c'\frac{\log n}{\log \log n}$. 
	Note again that factors in terms of $(c')^m$ has $(c')^m\le c^m = n^{O(\frac{1}{\log \log n})}=n^{o(1)}$ and can be ignored. 
	
	We now split into two cases: 
	if $D_2\ge \frac 12$, then choosing $c = 1$ yields $I_1(m) \in n^{o(1)}$ for all $m\le k$ and 
	$I_2(m)\le n^{-2+o(1)+c} = n^{-1+o(1)}$, 
	and the residual term $r_k(\theta)$ are also bounded by $n^{-2D_2 + o(1)}\le n^{-1+o(1)}$, 
	giving the total regret bound as $n^{O(\frac{\log \log n}{\log n})}$. 
	
	If $0 < D_2 < \frac 12$, 
	then for each $m\le k$, $I_1(m)$ is bounded by $n^{c(1 - 2D_2) + o(1)}$, 
	$I_2(m)$ by $n^{-2+c(2 - 2D_2) + o(1)}$, 
	and $|r_k(\theta)|$ by $n^{-2cD_2+o(1)}$. Thus choosing $c = 1$ yields each of these bounded by $n^{1-2D_2+o(1)}$.

\end{proof}

\begin{proof}[Proof of \prettyref{thm:eb_cts} (lower bound)]
	
	By \prettyref{prop:reduction}, our problem is now reduced to showing that 
	\[
	\sup_{\ell\in \mathcal{F}_{\beta}(h, L)}
	\inf_{\hat{U}}\sup_{\pi\in \mathcal{P}(I(h))} 
	\mathbb{E}\left[\left(\hat{U}(X_1^n) - \mathbb{E}_{\pi}[\ell(\theta)]\right)^2\right] = \Omega_{h, L}\left(\left(\frac{\log\log n}{\log n}\right)^{2\beta}\right)
	\]
	
	Now consider a fixed $\ell$. 
	By Le Cam's two-point method \cite[Theorem 31.1]{polyanskiy2025information}, 
	it suffices to identify $\pi_1, \pi_2$ such that $\mathsf{TV}(f_{\pi_1}, f_{\pi_2})\le \frac{1}{n^2}$
	(which then implies $\mathsf{TV}(f_{\pi_1}^{\otimes n}, f_{\pi_2}^{\otimes n})\le 1 - (1-\frac{1}{n^2})^n = O(\frac 1n)$),
	but where 
	$|\mathbb{E}_{\pi_1}[\ell(\theta)] - \mathbb{E}_{\pi_2}[\ell(\theta)]| = \Omega_h\left(\left(\frac{\log\log n}{\log n}\right)^{\beta}\right)$. 
	We now recall the following two results: 
	\begin{itemize}
		\item \cite[Lemma 9]{wu2020optimal}: For the normal means model, if $\pi_1, \pi_2\in\mathcal{P}([-h, h])$ and have equal first $d$ moments for $d > 2h^2$, then 
		\[
		\mathsf{TV}(f_{\pi_1}, f_{\pi_2})\le 
		\frac 12\left[\sum_{m=d+1}^{\infty}\frac{4h^{2m}}{m!}\right]^{1/2} 
		\le \frac 12 \left[2\cdot \frac{4h^{2(d+1)}}{(d+1)!}\right]^{1/2} 
		= O_{h}\left(\frac{1}{\sqrt{(d+1)!}}\right)
		\]
		
		\item \cite[Lemma 3]{wu2016minimax}: For the Poisson model, if $\pi_1, \pi_2\in\mathcal{P}([0, h])$, and equal first $d$ moments with $d > 2eh$, then 
		\[
		\mathsf{TV}(f_{\pi_1}, f_{\pi_2})\le \left(\frac{2eh}{d}\right)^d
		\]
	\end{itemize}
	In particular, choose $d = C\frac{\log n}{\log \log n}$ for some $C:= C(h)$ satisfies the required bounds in both the Poisson and normal means models. 
	We are now left to identify 
	\begin{equation}\label{eq:max_functional_diff}
		\sup_{\pi_1, \pi_2\in \mathcal{P}(I(h))} \{|\mathbb{E}_{\pi_1}[\ell(\theta)] - \mathbb{E}_{\pi_2}[\ell(\theta)]|
		: 
		\mathbb{E}_{\pi_1}[\theta^j]=\mathbb{E}_{\pi_2}[\theta^j], \forall j=0, \cdots, d\}
	\end{equation}
	To proceed, denote $\mathcal{P}_d$ as the set of polynomials with degree at most $d$. 
	By Riesz-Markov representation theorem, 
	\prettyref{eq:max_functional_diff} is equivalent to 
	\[
	\sup_{\Lambda: C(I(h))\to \mathbb{R}
	} 
	\{\Lambda(\ell): 
	||\Lambda||\le 1, 
	\Lambda(\theta^j) = 0, 
	\forall j=0, 1, \cdots, d\} 
	= 2\inf_{p\in \mathcal{P}_d} (\sup_{x\in I(h)} |\ell(x) - p(x)|)
	\]
	where the supremum is taken over all linear functionals $\Lambda: C(I(h))\to \mathbb{R}$, 
	and the equality is due to Hahn-Banach separation theorem 
	(e.g. \cite[Lemma 1]{carleson1972best}). 
	
	To finish, note that the Jackson-style approximation bound in \prettyref{lmm:deriv_poly} is tight. 
	For example, when $\beta$ is not an even integer, 
	Bernstein \cite{bernstein1938meilleure} shows that 
	\[\inf\{|| |x - \frac{h}{2}|^{\beta} - p||_{L_{\infty}[-h, h]}: \deg(p)\le k\}=
	\Theta_h(k^{-\beta})
	\]
	Note that $\ell(x) := |x - \frac{h}{2}|^{\beta}\in \mathcal{F}_{\beta}(h, L)$ for some $L:= L(\beta)$, and $\sup_{x\in [-h, h]} |\ell(x)|\le (\frac{3}{2}h)^{\beta}$.
	When $\beta$ is even integer, we may just replace the target function with $(x - \frac{h}{2})^{\beta-1}|x - \frac{h}{2}| + (4h)^{\beta}$. 
	
\end{proof}

 \section*{Acknowledgement}
 This work was supported in part by the MIT-IBM Watson AI Lab and the National Science
    Foundation under Grant No CCF-2131115. A.T. was supported by a fellowship from the Eric and Wendy Schmidt Center at the Broad Institute.

\bibliographystyle{alpha}
\bibliography{references}

\appendix
\section{Additional details from introduction section}\label{app:intro-details}
\subsection{Experimental setup}\label{app:exp-details}
In the experiment that produces \prettyref{fig:intro_unif10}, 
each experiment is repeated over 1000 seeds. 
For both $\hat{\theta}_{\mathsf{NPMLE}}$ and $\Tnpmlek$, 
we compute $\hat{\pi}_{\mathsf{NPMLE}}$ using \cite[Algorithm 1]{JPW24}, with the following details and modifications. 

\begin{itemize}
	\item The atom locations are initialized as $2m$ evenly spaced locations in $[0, X_{\max}]$ where $m$ is the number of distinct elements among $X_1, \cdots, X_n$, and $X_{\max}=\max\{X_1, \cdots, X_n\}$. The atom weights $\mu_i$ are initialized by sampling $\mu_i'\simiid \mathsf{Unif}([0, 1])$ and then normalized so that the sum is 1. 
	
	\item The atom-merging and atom-pruning (lines 5 and 7 in \cite[Algorithm 1]{JPW24}) are performed every 7 steps and 2 steps, respectively. 
	
	\item The thresholds $\eta_1$ and $\eta_2$ for atoms merging and pruning are chosen as $10^{-2}$ and $10^{-10}$, respectively. 
	$\eta_1$ denotes the distance of two adjacent atoms; 
	$\eta_2$ denotes the weight of the atoms. 
	
	\item We do not fix $N$, the number of iterations. Instead, we consider the following criterion of convergence: 
	\[
	\sup_{\theta\in [0, X_{\max}]} \frac{1}{n}\sum_{i=1}^n \frac{f_{\hat{\pi}}(X_i)}{f_{\theta}(X_i)} \le 1 + 10^{-3}
	\]
	with $f_{\hat{\pi}}$ and $f_{\theta}$ denote $X_i$'s PMF at the Poisson mixture $\Poi\circ\hat{\pi}$ and Poisson distribution $\Poi(\theta)$, respectively. 
	This then implies that the loglikelihood $\ell$ of $\hat{\pi}$ satisfies $\ell(\hat{\pi})\ge \ell(\hat{\pi}_{\mathsf{NPMLE}}) - 10^{-3}$ (see, for example, Footnote 3 of \cite{polyanskiy2020self}). 
\end{itemize}

\subsection{Transport map}\label{app:transport-map}
We state things more rigorously in this section, as follows. 
\begin{prop}[Reduction of posterior functional to prior functional]\label{prop:reduction}
	For any prior $\pi$ and an estimator $\hat{T}: \mathbb{R}^n\to\mathbb{R}^n$, there exists an estimator $\hat{U}: \mathbb{R}^n\to\mathbb{R}$ such that 
	\[
	\mathbb{E}_{\pi}[(\hat{U}(X_1^n) - \mathbb{E}_{\pi}[\ell(\theta)])^2]
	\le 2\frac{\TotRegret_{\pi, \ell, n}(\hat{T})}{n} + 2\frac{\mathsf{Var}(\ell_{\sharp}(\pi))}{n}
	\]
\end{prop}
\begin{proof}
	Take $\hat{U} := \frac 1n \sum_{i=1}^n \hat{T}_i(X_1, \cdots, X_n)$, 
	we have 
	\begin{flalign*}
		\mathbb{E}\left[\left(\hat{U} - \mathbb{E}_{\pi}[\ell(\theta)]\right)^2\right]
		&\le 2\mathbb{E}\left[\left(\frac 1n \sum_{i=1}^n \left(\hat{T}(X)_i - \hat{t}_{\pi, \ell}(X_i)\right)\right)^2\right]
		+ 2\mathbb{E}\left[\left(\frac 1n \sum_{i=1}^n \hat{t}_{\pi, \ell}(X_i) -\mathbb{E}_{\pi}[\ell(\theta)]\right)^2\right]
		\\&\le \frac{2}{n}\TotRegret_{\pi, \ell, n}(\hat{T})
		+ \frac{2}{n}\mathsf{Var}(\ell_{\sharp}(\pi))
	\end{flalign*}
	where we remind ourselves that $\hat{t}_{\pi, \ell}(X) = \mathbb{E}_{\pi}[\ell(\theta) | X]$. 
\end{proof}

\begin{prop}\label{prop:transport-map}
	Let $\pi, \hat{\pi}$ be both supported on $[-h, h]$. 
	Suppose that there exists $a, b\in [-h, h]$ with $a < b$, and $\epsilon \in (0, \frac 12)$ such that $\hat{\pi}([-h, a]) > \pi([-h, b]) + \epsilon$. 
	Then there exists a polynomial $p$ of degree $O_h(\frac{1}{b - a}\log \frac{1}{\epsilon})$ such that $0\le p(x)\le 1$ in $[-h, h]$, and $|\mathbb{E}_{\hat{\pi}}[p(\theta)] - \mathbb{E}_{\pi}[p(\theta)]|\ge \frac{\epsilon}{2}$. 
\end{prop}

\begin{proof}
	We invoke a polynomial approximation result on step functions (e.g. \cite[Lemma 25]{gilyen2019quantum}): there exists a polynomial $p$ of degree $O_h(\frac{1}{b - a}\log \frac{1}{\epsilon})$ such that $p(\theta)\in [0, 1]$ for $\theta\in [-h, h]$, and 
	\[
	p(\theta)
	\begin{cases}
		\ge 1 - \frac{\epsilon}{4} & \theta\ge b\\
		\le \frac{\epsilon}{4} & \theta \le a\\
	\end{cases}
	\]
	Then we have 
	\[
	\mathbb{E}_{\pi}[p(\theta)]
	\ge \pi([b, h])\cdot(1 - \frac{\epsilon}{4})
	\ge \hat{\pi}([a, h]) + \epsilon -\pi([b, h])\cdot\frac{\epsilon}{4}
	\stepa{\ge} \hat{\pi}([a, h]) + \frac{\epsilon}{4}\hat{\pi}([-h, a]) + \frac{\epsilon}{2}
	=\mathbb{E}_{\hat{\pi}}[p(\theta)] + \frac{\epsilon}{2}
	\]
	where (a) uses $\pi([b, h]), \hat{\pi}([-h, a])\le 1$.  
\end{proof}

\section{Technical proofs for normal means model (lower bound)}\label{app:gaussian_lower}

We remind that in this section we choose prior to be $N(0, s)$, and $\eta=\sqrt{\frac{s}{s + 1}}$. As in~\cite[Proposition 8]{polyanskiy2021sharp}, we then have that $K$ acts as follows
$$ Kr(y) = (r(\eta\cdot) * \varphi)(\eta y) = \mathbb{E}_{\theta\sim \calN(\eta^2y, \eta^2)}[r(\theta)]\,. $$
In particular, we have the following first- and second- order identities: 
\[
(K\theta)(y) = \eta^2y\qquad (K\theta^2)(y) = \eta^4y^2+\eta^2. 
\]

\begin{proof}[Proof of \prettyref{lmm:8_generalized}]
	We proceed by induction on $j$. 
	Fix $r$ fulfilling \prettyref{eq:polybnd}. 
	For $j = 0$, both sides of \prettyref{eq:gsn_kop} are 0. 
	For $j = 1$, the claim reduces to 
	$K_1r(y) = \eta^2K(r')(y)$, which is established in \cite[Proposition 8]{polyanskiy2021sharp} under a more restrictive condition on $r$. 
	To apply this to a broader class of $r$ like in \prettyref{eq:polybnd}, 
	note that 
	\begin{flalign}\label{eq:gsn_opj1}
		K_1r(y)
		&= (K\theta r)(y) - (K\theta)(Kr)(y)
		\nonumber\\&= \frac{1}{\eta}\int\theta r(\theta)\varphi(\frac{\theta}{\eta} - \eta y)d\theta
		-(K\theta)\cdot 
		(\frac{1}{\eta}\int r(\theta)\varphi(\frac{\theta}{\eta} - \eta y))d\theta
		\nonumber\\&\stepa{=} \frac{1}{\eta}
		\int r(\theta) (\theta - \eta^2 y)\varphi(\frac{\theta}{\eta} - \eta y)d\theta
		\nonumber\\&\stepb{=}-\eta 
		\int r(\theta) \frac{d}{d\theta}(\varphi(\frac{\theta}{\eta} - \eta y))d\theta
		\nonumber\\&\stepc{=}\eta \int r'(\theta)\varphi(\frac{\theta}{\eta} - \eta y)d\theta
		\nonumber\\&=\eta^2 Kr'(y)
	\end{flalign}
	where (a) is because 
	$(K\theta)(y)=\eta^2 y$ is the posterior mean of $\theta$ given $y$, 
	(b) uses the identity $\frac{d}{d\theta}\varphi\left(\frac{\theta}{\eta} - \eta y\right) = \frac{-\theta + \eta^2 y}{\eta^2}\varphi(\frac{\theta}{\eta} - \eta y)$, 
	and (c) follows from integration by parts, together with  
	\[
	\lim_{\theta\to\infty}r(\theta)\varphi\left(\frac{\theta}{\eta} - \eta y\right)
	=\lim_{\theta\to -\infty}r(\theta)\varphi\left(\frac{\theta}{\eta} - \eta y\right) = 0
	\]
	due to $r\in\mathcal{O}^{(k)}\subseteq \mathcal{O}^{(0)}$ as defined in 
	\prettyref{eq:polybnd}. 
	Moreover, given that $r'\in \mathcal{O}^{(k - 1)}$, we have $op_1(r)\in \mathcal{O}^{(k - 1)}$. 
	
	Now suppose that our claim at \prettyref{eq:gsn_kop} holds for some $j$ where $1\le j\le k - 1$. 
	To extend this claim to $j + 1$, we note the following: 
	\[
	K(\theta^{j + 1}) - (K\theta)(K\theta^j) = K_1(\theta^j) = j\eta^2K(\theta^{j - 1})
	\]
	so 
	$$ op_{j + 1}(\cdot) = j\eta^2 op_{j-1}(\cdot) + op_1( op_j(\cdot))
	=j\eta^2 op_{j-1}(\cdot) + (\theta\cdot - \eta^2 \partial\cdot)op_j(\cdot)\,$$
	where in the last equation we used the fact that $op_j(r)\in \mathcal{O}^{(k - j)}\subseteq \mathcal{O}^{(0)}$ and therefore 
	\prettyref{eq:gsn_opj1} can be applied to $op_j(r)$. 
	In addition, 
	we have both $op_{j-1}(r)$ and $(\theta\cdot - \eta^2 \partial\cdot)op_j(r)\in \mathcal{O}^{(k - j - 1)}$, 
	hence $op_{j + 1}(r)\in \mathcal{O}^{(k - j - 1)}$. 
	
	To derive~\eqref{eq:opj_topr} from here we could plug back in the formula \prettyref{eq:opj_gsn} for $op_j$ and $op_{j - 1}$.
	Instead, we note that in \prettyref{eq:opj_gsn}, the coefficient of $(\theta\cdot)^i (-\eta^2\partial\cdot)^{j-i}$ is equal to that of $X^iY^{j-i}$ in the expansion $(X - \eta^2Y)^{j}$ for real (scalars) $X$ and $Y$ 
	(assuming $X$ and $Y$ are independent variables). 
	We observe that 
	\[
	\frac{\partial}{\partial X}[(X - \eta^2Y)^j] + (X - \eta^2Y)^jY
	=j(X - \eta^2Y)^{j-1} + (X - \eta^2Y)^jY
	\]
	i.e. $\partial \cdot op_j = jop_{j-1} + op_j\partial\cdot$. 
	It then follows that we have 
	\[
	op_{j + 1} = j\eta^2 op_{j-1} + (\theta\cdot - \eta^2 \partial\cdot)op_j
	=j\eta^2 op_{j-1} +  \theta\cdot op_j - \eta^2(jop_{j-1} + op_j\partial\cdot)
	=\theta\cdot op_j - \eta^2op_j\partial\cdot
	\]
	as desired. 
	Finally, going from \prettyref{eq:opj_gsn} to \prettyref{eq:gsn_kop} is immediate by 
	\prettyref{eq:kk_vs_opk}.

\end{proof}

To proceed, we consider the following system of Hermite polynomials, orthogonal in $L_2(\mathcal{N}(0,1))$: 
\[
H_k(y) = (-1)^ke^{y^2/2}\frac{d^k}{dy^k}e^{-y^2/2}
\]
Next, denote the following quantities: 
\[
\rho = \frac{s}{s + 1} = \eta^2\qquad \lambda_1 = \frac{1}{2\pi s} \frac{1 + s}{\sqrt{1 + 2s}}
\qquad \lambda_2 = \frac{(1 + s)^2}{s(1 + 2s)}
\qquad \rho_1 = \sqrt{1 - \rho^2} 
\]
\begin{equation}\label{eq:gsn_notation_def}
	\alpha_1 = \sqrt{2\lambda_2\rho_1}
	\qquad 
	\mu = \frac{\rho}{1 + \rho_1}\qquad \lambda_0 = \frac{1}{\sqrt{2\pi s(1 + \rho_1)}}
\end{equation}
Denote, also, the following function: 
\[
\psi_q(\theta)\triangleq \sqrt{\frac{\alpha_1}{q!}} H_q(\alpha_1 \theta)\sqrt{\varphi(\alpha_1 \theta)}
\]
where we recall that $\varphi(x) = \frac{1}{\sqrt{2\pi}}e^{-\frac{x^2}{2}}$ denotes the standard normal density of $x$. 
Then the self-adjoint operator $S=K^*K$ satisfies $S\psi_q = \lambda_0 \mu^q\psi_q$ by \cite[(85)]{polyanskiy2021sharp}, 
and $\{\psi_q\}_{q\ge 0}$ forms an orthonormal basis of $L^2
(\text{Leb})$ consisting of eigenfunctions of $S$ \cite[(82)]{polyanskiy2021sharp}. 
This indeed follows from the orthogonality of $H_i$ \cite[7.374]{gradshteyn2014table}. 

We now consider the following identity, 
which says that both $\theta^j \psi_q$ and $op_j(\psi_q)$ can be written as linear combinations of $\psi_{q - j}, \cdots, \psi_{q + j}$, and in addition the coefficients at $\psi_{q - j}$ and $\psi_{q + j}$ can be easily identified. 
\begin{lemma}\label{lmm:gsn_xmpsiqsimple}
	Let $q\ge j\ge 0$. 
	Then there exist coefficients $c_1(\cdot, \cdot, \cdot)$ and $c_2(\cdot, \cdot, \cdot)$ such that 
	\begin{equation}\label{eq:thetaj_psiq}
		\theta^j \psi_q = \sum_{i = -j}^j c_1(q, j, i)\psi_{q + i}
		\qquad 
		op_j(\psi_q) = \sum_{i = -j}^j c_2(q, j, i)\psi_{q + i}
	\end{equation}
	satisfying the following: 
	\begin{equation}\label{eq:psiq_c1}
		c_1(q, j, j) = \frac{\sqrt{(q + 1)_{j}}}{\alpha_1^j}
		\qquad 
		c_1(q, j, -j) = \frac{\sqrt{(q - j + 1)_{j}}}{\alpha_1^j}
	\end{equation}
	\begin{equation}\label{eq:psiq_c2}
		c_2(q, j, j) = \sqrt{(q + 1)_{j}}\left(\frac{1}{\alpha_1} + \frac{\eta^2\alpha_1}{2}\right)^j
		\qquad c_2(q, j, -j) = \sqrt{(q - j + 1)_{j}}\left(\frac{1}{\alpha_1} - \frac{\eta^2\alpha_1}{2}\right)^j
	\end{equation}
\end{lemma}

\begin{proof}
	We first quickly check that for any $q\ge 0$ and any $j\ge 0$, $\theta^j\psi_q\in \mathcal{O}^{(k)}$ for every $k\ge 0$: indeed, $H_q(y)$ is a polynomial in $y$ with degree $q$. 
	It follows that the recursion from \prettyref{eq:opj_topr} can be used on $\theta^j\psi_q$. 
	
	We invoke the following two identities: 
	$\theta H_q(\theta) = H_{q+1}(\theta) + qH_{q - 1}(\theta)$ \cite[8.952.2]{gradshteyn2014table}, 
	and $H_q'(\theta) = qH_{q - 1}(\theta)$ \cite[8.952.1]{gradshteyn2014table}. 
	This gives the following as a starting point, as established in \cite{polyanskiy2021sharp} (between (85) and (86)): 
	\begin{equation}\label{eq:psiq_prime}
		\psi_q' = \frac{\alpha_1}{2}[\sqrt{q}\psi_{q - 1} - \sqrt{q + 1}\psi_{q + 1}]
	\end{equation}
	and 
	\begin{align}\label{eq:theta_psiq}
		\theta\psi_q(\theta) &= \sqrt{\frac{\alpha_1}{q!}}\theta H_q(\alpha_1\theta)\sqrt{\varphi (\alpha_1\theta)}
		\nonumber\\&=\frac{1}{\alpha_1}\left(\sqrt{\frac{\alpha_1}{q!}}H_{q + 1}(\alpha_1\theta)\sqrt{\varphi(\alpha_1\theta)}
		+\sqrt{\frac{\alpha_1}{q!}}qH_{q - 1}(\alpha_1\theta)\sqrt{\varphi(\alpha_1\theta)}\right)
		\nonumber\\&=\frac{1}{\alpha_1}(\sqrt{q + 1}\psi_{q + 1}(\theta) + \sqrt{q}\psi_{q - 1}(\theta))
	\end{align}
	
	We first tackle $\theta^j\psi_q$, 
	the case $j = 1$ follows from \prettyref{eq:theta_psiq}, 
	and if \prettyref{eq:thetaj_psiq} and \prettyref{eq:psiq_c1} hold for some $j\ge 1$, then 
	\[
	\theta^{j+1}\psi_q(\theta)
	=\theta\left(\sum_{i = -j}^j c_1(q, j, i)\psi_{q + i}\right)
	=\frac{1}{\alpha_1}\left(\sum_{i = -j}^j c_1(q, j, i)(\sqrt{q + i + 1}\psi_{q + i + 1}+ \sqrt{q + i}\psi_{q + i - 1})\right)
	\]
	which verifies that $\theta^{j+1}\psi_q$ is indeed linear combinations of $\psi_{q-j-1}, \cdots, \psi_{q+j+1}$. 
	In addition $c_1(q, j+1, j+1)$ and $c_1(q, j + 1, -(j+1))$ can be obtained via 
	\[
	c_1(q, j+1, j+1) = \frac{1}{\alpha_1} c_1(q, j, j) \sqrt{q+j+1} = \frac{\sqrt{q+j+1}}{\alpha_1}\cdot \frac{\sqrt{(q + 1)_j}}{\alpha_1^j} = \frac{\sqrt{(q + 1)_{j+1}}}{\alpha_1^{j+1}}
	\]
	\[
	c_1(q, j+1, -(j+1)) = \frac{1}{\alpha_1} c_1(q, j, -j) \sqrt{q-j} = \frac{\sqrt{q-j}}{\alpha_1}\cdot \frac{\sqrt{(q - j + 1)_j}}{\alpha_1^j} = \frac{\sqrt{(q - j)_{j+1}}}{\alpha_1^{j+1}}
	\]
	as desired. 
	
	Now to tackle $op_j(\psi_q)$, we again rely on recursion and note the starting point $c_2(q, 0, 0) = 1$. 
	Assume that \prettyref{eq:thetaj_psiq} and \prettyref{eq:psiq_c2} hold for some $j\ge 0$, and for all $q\ge j$. 
	Then we use $op_{j+1}(\psi_q) = \theta op_j(\psi_q) - \eta^2 op_j(\psi_q')$ to get 
	\begin{flalign*}
		op_{j+1}(\psi_q) &\stepa{=} \theta op_j(\psi_q) - \eta^2 op_j(\psi_q')
		\nonumber\\&\stepb{=} \theta op_j(\psi_q) - \frac{\alpha_1}{2}\sqrt{q}\cdot \eta^2 op_j(\psi_{q-1})
		+ \frac{\alpha_1}{2}\sqrt{q + 1}\cdot \eta^2 op_j(\psi_{q+1})
		\nonumber\\&=\sum_{i=-j}^j c_2(q, j, i)\theta\psi_{q+i}
		- \frac{\alpha_1\eta^2}{2}\sqrt{q} \sum_{i=-j}^j c_2(q - 1, j, i)\psi_{q + i - 1}
		+ \frac{\alpha_1}{2}\sqrt{q + 1}\cdot \eta^2 \sum_{i=-j}^j c_2(q + 1, j, i)\psi_{q + i + 1}
		\nonumber\\&=\sum_{i=-j}^j 
		\left[c_2(q, j, i)\frac{1}{\alpha_1}\left(\sqrt{q+i+1}\psi_{q+i + 1} + \sqrt{q+i}\psi_{q+i-1}\right)\right.
		\nonumber\\&+\left.\frac{\alpha_1\eta^2}{2}\left(\sqrt{q+1}c_2(q+1, j, i)\psi_{q+i+1} - \sqrt{q}c_2(q - 1, j, i)\psi_{q-i-1}\right)
		\right]
	\end{flalign*}
	where (a) uses the recursion in \prettyref{eq:opj_topr}, 
	and (b) uses \prettyref{eq:psiq_prime}. 
	It follows that $op_{j+1}(\psi_q)$ is linear combination of $\psi_{q-j-1}, \cdots, \psi_{q+j+1}$. 
	In addition, the coefficients $c_2(q, j + 1, j + 1)$ and $c_2(q, j + 1, -(j + 1))$ can be determined via the following recursion formula (and using induction hypothesis for $j$)
	\begin{flalign*}
		c_2(q, j + 1, j + 1) 
		&= c_2(q, j, j)\cdot \frac{1}{\alpha_1}\sqrt{q+j + 1}
		+ \frac{\alpha\eta^2}{2}\sqrt{q+1}c_2(q+1, j, j)
		\\&=\sqrt{(q + 1)_{j}}\left(\frac{1}{\alpha_1} + \frac{\eta^2\alpha_1}{2}\right)^j\frac{1}{\alpha_1}\sqrt{q+j + 1}
		+ \frac{\alpha_1\eta^2}{2}\sqrt{(q + 2)_{j}}\left(\frac{1}{\alpha_1} + \frac{\eta^2\alpha_1}{2}\right)^j\frac{1}{\alpha_1}\sqrt{q + 1}
		\\&=\sqrt{(q+1)_{j+1}}\left(\frac{1}{\alpha_1} + \frac{\eta^2\alpha_1}{2}\right)^{j+1}
	\end{flalign*}
	\begin{flalign*}
		c_2(q, j + 1, -(j + 1)) 
		&= c_2(q, j, -j)\cdot \frac{1}{\alpha_1}\sqrt{q-j}
		- \frac{\alpha_1\eta^2}{2}\sqrt{q}c_2(q-1, j, -j)
		\nonumber\\&=\sqrt{(q - j + 1)_j}\left(\frac{1}{\alpha_1} - \frac{\eta^2\alpha_1}{2}\right)^{j}\cdot \frac{1}{\alpha_1}\sqrt{q-j}
		- \frac{\alpha_1\eta^2}{2}\sqrt{(q - j)_j}\sqrt{q}\left(\frac{1}{\alpha_1} - \frac{\eta^2\alpha_1}{2}\right)^{j}
		\nonumber\\&=\sqrt{(q - j)_{j + 1}}\left(\frac{1}{\alpha_1} - \frac{\eta^2\alpha_1}{2}\right)^{j + 1}
	\end{flalign*}
	
\end{proof}

\begin{proof}[Proof of \prettyref{lmm:9_generalized}]
	We bound $||\psi_q||_{\infty}$ to start with. 
	This is shown in the end of proof of \cite[Lemma 9]{polyanskiy2021sharp}, 
	using Cramér's inequality \cite[8.954.2]{gradshteyn2014table} that for each $y$, 
	$|H_k(y)|\le c\sqrt{k!}e^{y^2/4}$ for some absolute constant $c$, and therefore 
	\[
	|\psi_q(y)| \le c\sqrt{\frac{\alpha_1}{q!}}\sqrt{q!}e^{(\alpha_1 y)^2/4}\varphi(\alpha_1 y)
	\le \frac{c\sqrt{\alpha_1}}{\sqrt{2\pi}}e^{(\alpha_1 y)^2/4 - (\alpha_1 y)^2/2}
	\le \frac{c\sqrt{\alpha_1}}{\sqrt{2\pi}}
	\]
	
	Next, we consider the following for all $q\ge k$: 
	\begin{flalign*}
		(K_k\psi_q, K_k\psi_q)
		&= \| K(\theta^k \psi_q - op_k(\psi_q))\|^2
		\nonumber\\&= \|(\theta^k \psi_q - op_k(\psi_q))\|^2_S
		\nonumber\\&\stepa{=} \|\sum_{i=-k}^k (c_1(q, k, i) - c_2(q, k, i)) \psi_{q + i}\|^2_S
		\nonumber\\&\stepb{=} \sum_{i=-k}^k (c_1(q, k, i) - c_2(q, k, i))^2 \|\psi_{q + i}\|^2_S
		\nonumber\\&\ge (c_1(q, k, -k) - c_2(q, k, -k))^2\|\psi_{q - k}\|^2_S + (c_1(q, k, k) - c_2(q, k, k))^2\|\psi_{q + k}\|^2_S
		\nonumber\\&= \frac{(q-k+1)_k}{\alpha_1^{2k}}\left(1 - (1 - \frac{\eta^2\alpha_1^2}{2})^k\right)^2\lambda_0\mu^{q - k} + \frac{(q+1)_k}{\alpha_1^{2k}}\left(1 - (1 + \frac{\eta^2\alpha_1^2}{2})^k\right)^2\lambda_0\mu^{q + k}
	\end{flalign*}
	where (a) is due to \prettyref{lmm:gsn_xmpsiqsimple} which shows that both $\theta^k\psi_q$ and $op_k(\psi_q)$ are linear combinations of $\psi_{q-k}, \cdots, \psi_{q+k}$, 
	and (b) is due to the orthogonality of $\{\psi_q\}_{q\ge 1}$ in $S$-norm. 
	Note that the expansion of $K_k(\psi_q)$ into linear combinations of $\psi_{q-k}, \cdots, \psi_{q+k}$ also implies that $K_k(\psi_q), K_k(\psi_r)$ are orthogonal in $L^2(f_0)$ whenever $|q - r|\ge 2k + 1$. 
	
	Finally, 
	by the definitions in \prettyref{eq:gsn_notation_def}, 
	for all $s$ with $0 < s < \frac 12$, we have $0\le \rho\le\frac 13$ and $\sqrt{\frac 89}\le \rho_1\le 1$ 
	(hence $\rho_1\asymp 1$). Therefore, 
	\[
	\lambda_0\asymp\frac{1}{\sqrt{s}}, 
	\mu = \frac{s}{(s + 1)(1+\rho_1)}\asymp s, 
	\alpha_1 = (1+s)\sqrt{2\rho_1\frac{1}{s(1+2s)}}
	\asymp s^{-1/2}
	\]
	as required. 
	For the term $c(s, k) = 1 - (1 - \frac{\eta^2\alpha_1^2}{2})^k$ we note the following: 
	\[
	\eta^2\alpha_1^2 = \frac{s}{1 + s}\cdot 2 \frac{(1 + s)^2}{s (1 + 2s)}\sqrt{1 - (\frac{s}{1 + s})^2}
	=\frac{2(1 + s)}{1 + 2s}\sqrt{1 - (\frac{s}{1 + s})^2}
	\]
	We note that $\eta^2\alpha_1^2 < 2$ for all $s > 0$.  
	On the other hand, for $0 < s < \frac 12$, the expression above gives $\eta^2\alpha_1^2\ge \frac{4}{3}\sqrt{\frac{3}{4}}> 1$. 
	Therefore, $ 1 - (1 - \frac{\eta^2\alpha_1^2}{2})^k\ge 1 - (1 - \frac{\eta^2\alpha_1^2}{2}) = \frac{\eta^2\alpha_1^2}{2} > \frac 12$, which gives the parameter scalings claimed in \prettyref{lmm:9_generalized}. 
\end{proof}

\section{Technical proofs for Poisson model (lower bound)}\label{app:poisson_lower}
We recall that given $G_0 = \text{Gamma}(\alpha, \beta)$, the posterior $\theta | y$ is given by $\text{Gamma}(\alpha + y, \beta + 1)$. Consequently, $K$ acts on $r$ as 
\[
Kr(y) = \int_{\theta \ge 0} \frac{(1 + \beta)^{y + \alpha - 1}}{\Gamma(y + \alpha)}\theta^{y + \alpha - 1}e^{-(1+\beta)\theta}  r(\theta)d\theta.
\]
In other words, $Kr(y) = \mathbb{E}[r(\theta)]$ where $\theta \sim \text{Gamma}(\alpha + y, \beta + 1)$, 
so $K(\theta)(y) = \frac{y+ \alpha}{1 + \beta}$. 
\begin{proof}[Proof of \prettyref{lmm:Kkr}]
	We first prove \prettyref{eq:opj_poisson_recur}, where \prettyref{eq:opj_poisson} follows immediately. 
	Note that for $j = 1$ the identity $op_1(r) = \theta r - \frac{r'}{1 + \beta}$ is shown in \cite[(56)]{polyanskiy2021sharp} for a more restrictive class of $r$. 
	To apply this to a broader class of $r$ (as stated in \prettyref{eq:polybnd}), we note the following boundary conditions for all $r\in \mathcal{O}^{(k)}\subseteq \mathcal{O}^{(0)}$: 
	\begin{equation}\label{eq:theta_lim_gamma}
		\lim_{\theta\to\infty}\theta^{y + \alpha}e^{-(1+\beta)\theta}r(\theta) = 0
	\end{equation}
	and $\theta^{y + \alpha} = 0$ when $\theta = 0$ since $y+\alpha\ge \alpha > 0$.
	Therefore: 
	\begin{flalign*}
		K(\theta r')(y) &= \int_{\theta \ge 0} \frac{(1 + \beta)^{y + \alpha - 1}}{\Gamma(y + \alpha)}\theta^{y + \alpha}e^{-(1+\beta)\theta}  r'(\theta)d\theta
		\\&\stepa{=}-\int_{\theta \ge 0} \frac{(1 + \beta)^{y + \alpha - 1}}{\Gamma(y + \alpha)}(\theta^{y + \alpha}e^{-(1+\beta)\theta})'  r(\theta)d\theta
		\\&=-\int_{\theta \ge 0} \frac{(1 + \beta)^{y + \alpha - 1}}{\Gamma(y + \alpha)}e^{-(1+\beta)\theta}((y+\alpha)\theta^{y+\alpha-1} - (1+\beta)\theta^{y + \alpha})  r(\theta)d\theta
		\\&=  (1+\beta)K(\theta r)(y) - (y+\alpha) Kr(y)
		\\&\stepb{=} (1+\beta)K(\theta r)(y) - (1+\beta) K(\theta)Kr(y)
	\end{flalign*}
	where (a) follows from integration by parts and \prettyref{eq:theta_lim_gamma}, 
	and (b) uses the posterior mean identity $K(\theta)(y) = \frac{y + \alpha}{1 + \beta}$. 
	This effectively establishes that 
	$op_1(\cdot) = \theta(\cdot) - \frac{1}{1 + \beta}\theta(\partial \cdot) = \theta(1 - \frac{1}{1 + \beta}\partial\cdot)$. 
	
	Next, suppose our claim holds for some $j \ge 1$.
	We now consider the following: 
	\[
	K(\theta^{j+1})
	=K_1(\theta^j) + K(\theta)K(\theta^j)
	=\frac{j}{1+\beta}K(\theta^j) + K(\theta)K(\theta^j)
	\]
	where the second equality is due to \cite[(56)]{polyanskiy2021sharp}, $(\theta^j)' = j\theta^{j-1}$, 
	and the function $\theta\to \theta^j$ also satisfies \prettyref{eq:polybnd}. 
	This gives the following: 
	\[
	op_{j+1}(r) = \frac{j}{1+\beta}op_j(r) + op_1(op_j(r)), 
	\]
	
	Next, we claim the following identity: 
	for any real number $c$ and $j\ge 1$, we have 
	\begin{equation}\label{eq:diff-theta-op}
		(1 + c\partial\cdot )^j (\theta\cdot)
		=\theta\cdot (1 + c\partial\cdot )^j
		+jc(1 + c\partial\cdot )^{j - 1}. 
	\end{equation}
	Indeed, this can also be established via induction. 
	Note that $(\theta r)' = r + \theta r'$, 
	so $\partial\cdot\theta\cdot = 1 + \theta\cdot\partial\cdot$. 
	Therefore, for $j = 1$ we have 
	\[
	(1 + c\partial\cdot)(\theta\cdot) = \theta\cdot + c\partial\cdot\theta\cdot
	=\theta\cdot + c + c\theta\cdot\partial\cdot
	=\theta\cdot(1 + c\partial\cdot) + c
	\]
	Now suppose \prettyref{eq:diff-theta-op} holds for some $j\ge 1$, then 
	\begin{flalign*}
		(1 + c\partial\cdot )^{j+1} (\theta\cdot)(r)
		&=(1 + c\partial\cdot)
		[(\theta\cdot (1 + c\partial\cdot )^j
		+ jc(1 + c\partial\cdot )^{j - 1}]
		\\&=(1 + c\partial\cdot)(\theta\cdot)(1 + c\partial\cdot )^j
		+ jc(1 + c\partial\cdot )^{j}
		\\&=(\theta\cdot(1 + c\partial\cdot) + c)(1 + c\partial\cdot)^j
		+ jc(1 + c\partial\cdot )^{j}
		\\&=\theta\cdot(1 + c\partial\cdot)^{j+1} + (j + 1)c(1 + c\partial\cdot )^{j}
	\end{flalign*}
	which finishes the induction hypothesis. 
	Therefore, we now have: 
	\begin{flalign*}
		op_{j+1}(\cdot) 
		&= \frac{j}{1+\beta}op_j + op_1(op_j)
		\nonumber\\
		&\stepa{=}\frac{j}{1+\beta}
		\left(\theta^j\cdot \left(1 - \frac{1}{1 + \beta}\partial\cdot\right)^j \right)
		+\theta^j\cdot \left(1 - \frac{1}{1 + \beta}\partial\cdot\right)^j \theta\cdot \left(1 - \frac{1}{1 + \beta}\partial\cdot\right)
		\nonumber\\
		&=\frac{j}{1+\beta}
		\left(\theta^j\cdot \left(1 - \frac{1}{1 + \beta}\partial\cdot\right)^j \right)
		+\theta^j\cdot \theta\cdot\left(1 - \frac{1}{1 + \beta}\partial\cdot\right)^j  \left(1 - \frac{1}{1 + \beta}\partial\cdot\right)
		\nonumber\\
		&-\frac{j}{1+\beta}\theta^j\cdot \left(1 - \frac{1}{1 + \beta}\partial\cdot\right)^{j-1}\cdot \left(1 - \frac{1}{1 + \beta}\partial\cdot\right)
		\nonumber\\
		&=\theta^{j+1}\cdot \left(1 - \frac{1}{1+\beta}\partial\cdot\right)^{j+1}
	\end{flalign*}
	where in (a) we used the commutativity of $op_1$ and $op_j$, 
	and $op_1(r)$ fulfills \prettyref{eq:polybnd}.
\end{proof}

We now construct the functions to be used in \prettyref{lmm:7_generalized}. 
Like \cite[Appendix B]{polyanskiy2021sharp} we define $S \overset{\Delta}{=}K^*K$ and $S_k \overset{\Delta}{=}K_k^*K_k$, 
satisfying the following as per \cite[(89)]{polyanskiy2021sharp}: 
\[
(Kf, Kg)_{L_2(\mathbb{Z}_+, f_0)} = (Sf, g)_{L_2(\mathbb{R}_+, \text{Leb})}. 
\]
Next, denote $L_q^{\nu}$ using the generalized Laguerre polynomials for all $q\ge 1$, c.f. \cite[(8.970)]{gradshteyn2014table}. 
Recall that $\alpha, \beta$ such that the prior $G_0$ we have chosen is the Gamma prior $\text{Gamma}(\alpha, \beta)$. Based on these two quantities, we define the following symbols: 
\begin{equation}\label{eq:poisson-beta-z}
	z=(\sqrt{1+\beta}-\sqrt\beta)^2, 
	\quad 
	\gamma_2=2\sqrt{\beta(1+\beta)},\quad \gamma := \gamma_1 = \frac{\gamma_2}{2},\quad \nu = \alpha - 1. 
\end{equation}
and in addition, let $\Gamma_q^{\nu}(\theta) = e^{-\gamma_1\theta}L_q^\nu(\gamma_2 \theta) = e^{-\gamma\theta}L_q^{\nu}(2\gamma\theta)$. 
By \cite[(101)]{polyanskiy2021sharp}, 
the system of functions $\{\Gamma_q\}_{q\ge 1}$ satisfies 
\begin{align}\label{eq:bn_def}
	(S\Gamma_q,\Gamma_r)=b_q\mathbf{1}_{q=r},\quad b_q=C_2(\alpha,\beta)z^q\frac{\Gamma(q+\alpha)}{q!}.
\end{align}
where $C_2(\alpha, \beta) = \frac{(1 + \beta)^3\beta^{\alpha}}{\Gamma(\alpha)}(1 - z)((1 + \beta)z)^{\frac{\alpha-1}{2}}\gamma_2^{-\alpha-1}$.

We now study some properties of these generalized Laguerre polynomials. 
Specifically, they follow the following bound \cite[22.14.13]{AS64}:
\begin{align}\label{eq:LG_expo}
	|L_q^\nu(\theta)|\le e^{\theta/2}{q+\nu\choose q}, \forall \theta \ge 0.
\end{align}
Next, we establish some recurrence relations governing $\Gamma_q$. 

\begin{lemma}\label{lmm:recur_gamma_poisson}
	Let $j$ be such that $0\le j\le q$. Then there exist coefficients $c_1(\cdot, \cdot, \cdot, \cdot)$ and $c_2(\cdot, \cdot, \cdot)$ such that 
	\begin{equation}\label{eq:Gamma_poisson}
		\theta^j\Gamma_q^{\nu} = \sum_{i = - j}^j c_1(q, j, i, \nu) \Gamma_{q + i}^{\nu}
		\qquad op_j(\Gamma_q^{\nu}) = \sum_{i = - j}^j c_2(q, j, i, \nu) \Gamma_{q + i}^{\nu}
	\end{equation}
	satisfying the following: 
	\begin{equation}\label{eq:Gamma_c1}
		c_1(q, j, j) = (-2\gamma)^{-j}(q + 1)_j\qquad c_1(q, j, -j) = (-2\gamma)^{-j}(q - j + \nu + 1)_{j}
	\end{equation}
	\begin{equation}\label{eq:Gamma_c2}
		c_2(q, j, j) = (-\frac 12)^j(\frac{1}{\gamma} + \frac{1}{1 + \beta})^j(q + 1)_j\qquad c_2(q, j, -j) = (\frac 12)^j (-\frac{1}{\gamma} + \frac{1}{1 + \beta})^j (q - j + \nu + 1)_j
	\end{equation}
\end{lemma}
\begin{proof}
	We first verify the required regularity of $\Gamma_q^{\nu}$. 
	Indeed, the generalized Laguerre polynomials are polynomial in $\theta$, so for all $k\ge 0$, 
	$\Gamma_q^{\nu}\in\mathcal{O}^{(k)}$ as defined in
	\prettyref{eq:polybnd}, and therefore \prettyref{lmm:Kkr} applies. 
	
	We now consider $\theta^j\Gamma_q$. 
	Here, $c_1(q, 0, 0, \nu) = 1$, and the case $j = 1$ follows immediately from \cite[22.7.12]{AS64}: 
	\[
	(q + 1)L_{q+1}^{\nu}(\theta) = (2q+\nu+1 -\theta)L_q^{\nu}(\theta) -(q+\nu)L_{q - 1}^{\nu}(\theta)
	\]
	which can be rearranged into 
	$\theta L_q^{\nu}(\theta) = (2q+\nu+1)L_q^{\nu}(\theta) - (q + 1)L_{q+1}^{\nu}(\theta) - (q+\nu)L_{q - 1}^{\nu}(\theta)$. 
	Consequently, replacing $\theta$ with $2\gamma\theta$ in the argument, and multiplying by $e^{-\gamma\theta}$ gives 
	\begin{flalign*}
		\theta \Gamma_q^{\nu}(\theta) 
		&= (2\gamma)^{-1}(2\gamma\theta e^{-\gamma\theta}L_q^{\nu}(2\gamma\theta))
		\nonumber\\&=(2\gamma)^{-1}e^{-\gamma\theta}[(2q+\nu+1)L_q^{\nu}(2\gamma\theta) - (q + 1)L_{q+1}^{\nu}(2\gamma\theta) - (q+\nu)L_{q - 1}^{\nu}(2\gamma\theta)]
		\nonumber\\&=(2\gamma)^{-1}[(2q+\nu+1)\Gamma_q^{\nu}- (q + 1)\Gamma_{q+1}^{\nu} - (q+\nu)\Gamma_{q - 1}^{\nu}]
	\end{flalign*}
	Next, suppose that \prettyref{eq:Gamma_poisson} holds for some $j\ge 1$ and $q > j$ with coefficients satisfying \prettyref{eq:Gamma_c1}. 
	Then 
	\begin{flalign*}
		\theta^{j+1}\Gamma_q^{\nu}
		&=\theta(\theta^{j}\Gamma_q^{\nu})
		=\sum_{i = - j}^j c_1(q, j, i, \nu) \theta\Gamma_{q + i}^{\nu}
		\nonumber\\
		&=\sum_{i = - j}^j c_1(q, j, i, \nu) (2\gamma)^{-1}[(2(q+i)+\nu+1)\Gamma_{q+i}^{\nu}- (q + i + 1)\Gamma_{q + i +1}^{\nu} - (q + i +\nu)\Gamma_{q + i - 1}^{\nu}]
	\end{flalign*}
	which shows that $\theta^{j+1}\Gamma_q^{\nu}$ is indeed a linear combination of $\Gamma_{q - j - 1}, \cdots, \Gamma_{q + j + 1}$. 
	In addition, we may continue from the display above to obtain the following recursion for the tail coefficients: 
	\begin{flalign*}
		c_1(q, j + 1, j + 1, \nu) &= -(2\gamma)^{-1}(q + j + 1)c_1(q, j, j, \nu)  
		\nonumber\\&
		= (-2\gamma)^{-1}(q + j + 1)(-2\gamma)^{-j}(q + 1)_j
		\nonumber\\&
		= (-2\gamma)^{-j - 1}(q + 1)_{j + 1}
	\end{flalign*}
	\begin{flalign*}
		c_1(q, j + 1, -(j + 1), \nu)&= -(2\gamma)^{-1}(q - j +\nu)c_1(q, j, -j, \nu)
		\nonumber\\&= (-2\gamma)^{-1}(q - j +\nu)(-2\gamma)^{-j}(q - j + \nu + 1)_{j}
		\nonumber\\&= (-2\gamma)^{-j - 1}(q - j +\nu)_{j + 1}
	\end{flalign*}
	proving the induction claim. 
	
	Next, for $op_j(\Gamma_q^{\nu})$, we first note the starting point where $c_2(q, 0, 0, \nu) = 1$, 
	and from the recurrence relation in \cite[22.8.6]{AS64}: 
	\begin{equation}\label{eq:recur_dx1}
		\theta\frac{d}{d\theta}L_q^{\nu}(\theta) = q L_q^{\nu}(\theta) - (q + \nu)L_{q - 1}^{\nu}(\theta)
	\end{equation}
	we get the following recurrence \cite[(113)]{polyanskiy2021sharp}:
	\begin{equation}\label{eq:recur_dxGamma}
		\theta\frac{d}{d\theta}\Gamma_q^{\nu}(\theta) = 
		-\frac {q + \nu}{2}\Gamma_{q-1}^{\nu} - \frac {\nu + 1}{2}\Gamma_q^{\nu} + \frac {q + 1}{2}\Gamma_{q + 1}^{\nu}
	\end{equation}
	Therefore, 
	\begin{flalign*}
		op_1(\Gamma_q^{\nu}) &= \theta \Gamma_q^{\nu} - \theta\frac{(\Gamma_q^{\nu})'}{1 + \beta}
		\\&=(2\gamma)^{-1}[(2q+\nu+1)\Gamma_q^{\nu}- (q + 1)\Gamma_{q+1}^{\nu} - (q+\nu)\Gamma_{q - 1}^{\nu}]
		-\frac{1}{1 + \beta}(-\frac {q + \nu}{2}\Gamma_{q-1}^{\nu} - \frac {\nu + 1}{2}\Gamma_q^{\nu} + \frac {q + 1}{2}\Gamma_{q + 1}^{\nu})
	\end{flalign*}
	which gives $c_2(q, 1, 1, \nu) = -\frac{q + 1}{2}(\frac{1}{\gamma} + \frac{1}{1 + \beta})$ and 
	$c_2(q, 1, -1, \nu) = \frac{q + \nu}{2}(-\frac{1}{\gamma} + \frac{1}{1 + \beta})$. 
	
	Next, suppose that \prettyref{eq:Gamma_poisson} holds for some $j\ge 1$ and $q > j$ with coefficients satisfying \prettyref{eq:Gamma_c2}. We use the identity $op_j(r) = \theta^j(1 - \frac{\partial\cdot}{1 + \beta})^jr$ and also the recurrence $op_{j+1}(r) = \frac{j}{1 + \beta}op_j + op_1(op_j(r))$. 
	This gives the following: 
	\begin{flalign*}
		op_{j + 1}(\Gamma_q^{\nu})
		&=\frac{j}{1 + \beta}op_j(\Gamma_q^{\nu}) + op_1(op_j(\Gamma_q^{\nu}))
		\nonumber\\&=\sum_{i = - j}^j \frac{j c_2(q, j, i, \nu)}{1 + \beta} \Gamma_{q + i}^{\nu}
		+ c_2(q, j, i, \nu)op_1(\Gamma_{q + i}^{\nu})
		\nonumber\\&=\sum_{i = - j}^j \frac{j c_2(q, j, i, \nu)}{1 + \beta} \Gamma_{q + i}^{\nu}
		+ c_2(q, j, i, \nu)\theta\Gamma_{q + i}^{\nu}
		- \frac{c_2(q, j, i, \nu)}{1 + \beta}\theta(\Gamma_{q + i}^{\nu})'
		\nonumber\\&=\left(\sum_{i = - j}^j \frac{j c_2(q, j, i, \nu)}{1 + \beta} \Gamma_{q + i}^{\nu}\right.
		\nonumber\\&+ c_2(q, j, i, \nu)(2\gamma)^{-1}[(2(q+i)+\nu+1)\Gamma_{q+i}^{\nu}- (q + i + 1)\Gamma_{q + i +1}^{\nu} - (q + i +\nu)\Gamma_{q + i - 1}^{\nu}]
		\nonumber\\&\left.- \frac{c_2(q, j, i, \nu)}{1 + \beta}
		(-\frac {q + i + \nu}{2}\Gamma_{q + i - 1}^{\nu} - \frac {\nu + 1}{2}\Gamma_{q + i}^{\nu} + \frac {q + i + 1}{2}\Gamma_{q + i + 1}^{\nu})\right)
	\end{flalign*}
	which demonstrates that $op_{j+1}(\Gamma_q^{\nu})$ is a linear combination of $\Gamma_{q - j - 1}, \cdots, \Gamma_{q + j + 1}$. 
	In addition, the tail coefficients follow the following recursion: 
	\[
	c_2(q, j + 1, j + 1, \nu) 
	= -\frac{c_2(q, j, j, \nu) (q + j + 1)}{2} \left(\frac{1}{\gamma} + \frac{1}{1 + \beta}\right)
	\]
	\[=-\frac{(q + j + 1)}{2} (\frac{1}{\gamma} + \frac{1}{1 + \beta})\cdot (-\frac 12)^j(\frac{1}{\gamma} + \frac{1}{1 + \beta})^j(q + 1)_j
	=(-\frac 12)^{j+1}(\frac{1}{\gamma} + \frac{1}{1 + \beta})^{j+1}(q + 1)_{j + 1}
	\]
	\[
	c_2(q, j + 1, -(j + 1), \nu) 
	=\frac{c_2(q, j, -j, \nu) (q - j + \nu)}{2}(-\frac{1}{\gamma} + \frac{1}{1 + \beta})
	\]
	\[
	=(\frac 12)^j (-\frac{1}{\gamma} + \frac{1}{1 + \beta})^j (q - j + \nu + 1)_j\cdot \frac{(q - j + \nu)}{2}(-\frac{1}{\gamma} + \frac{1}{1 + \beta})
	=(\frac 12)^{j + 1} (-\frac{1}{\gamma} + \frac{1}{1 + \beta})^{j + 1} (q - j + \nu)_{j + 1}
	\]
\end{proof}

Now we prove a generalization of \cite[Lemma 15]{polyanskiy2021sharp}. This will allow us to construct a set of functions that are orthogonal and have bounded magnitude. 
\begin{lemma}\label{lmm:15_generalized}
	The functions $\Gamma_q(x)$ satisfy \begin{align}
		\label{eq:lemma15_1}\|\Gamma_q(\theta)\|_\infty&\le{q+\alpha\choose q} \\
		\label{eq:lemma15_2}(S_k\Gamma_q, \Gamma_q) &\ge \frac{b_q}{2^{2k}\beta^{k-1}(1+\beta)^{k+1}}(q+\nu - k + 1)_k(q - k + 1)_kz^{-k}\\
		\label{eq:lemma15_3}(S_k\Gamma_{q_1}, \Gamma_{q_2}) &= 0\quad \forall |q_1-q_2|\ge 2k+1
	\end{align}
\end{lemma}
\begin{proof}
	First, \prettyref{eq:lemma15_1} follows directly from \prettyref{eq:poisson-beta-z} and \prettyref{eq:LG_expo}, so it remains to show the identities with $S_k$. Using \prettyref{lmm:recur_gamma_poisson} we have for all $q\ge k$, 
	
	\begin{flalign*}
		(K_k\Gamma_q^{\nu}, K_k\Gamma_q^{\nu})
		&= \| K(\theta^k \Gamma_q^{\nu} - op_k(\Gamma_q^{\nu}))\|^2
		\nonumber\\&= \|(\theta^k \Gamma_q^{\nu} - op_k(\Gamma_q^{\nu}))\|^2_S
		\nonumber\\&\stepa{=} \|\sum_{i=-k}^k (c_1(q, k, i, \nu) - c_2(q, k, i, \nu)) \Gamma_{q + i}^{\nu}\|^2_S
		\nonumber\\&\stepb{=} \sum_{i=-k}^k (c_1(q, k, i, \nu) - c_2(q, k, i, \nu))^2 \|\Gamma_{q + i}^{\nu}\|^2_S
		\nonumber\\&\ge (c_1(q, k, -k) - c_2(q, k, -k))^2\|\Gamma_{q - k}^{\nu}\|^2_S + (c_1(q, k, k) - c_2(q, k, k))^2\|\Gamma_{q + k}^{\nu}\|^2_S
		\nonumber\\& = (q - k + \nu + 1)_k^2(-2\gamma)^{-2k}(1 + (1 - \frac{\gamma}{1 + \beta})^k)^2\|\Gamma_{q - k}^{\nu}\|^2_S
		\nonumber\\&+ (q+1)_k^2(-2\gamma)^{-2k}(1 - (1 + \frac{\gamma}{1 + \beta})^k)^2\|\Gamma_{q + k}^{\nu}\|^2_S
		\nonumber\\& \ge (q - k + \nu + 1)_k^2(-2\gamma)^{-2k}(1 - (1 - \frac{\gamma}{1 + \beta})^k)^2b_{q - k}
		\nonumber\\& \stepc{=} (q - k + \nu + 1)_k^2(-2\gamma)^{-2k}(1 - (1 - \frac{\gamma}{1 + \beta})^k)^2b_q z^{-k}\frac{(q - k +1)_k}{(q - k + \alpha)_k}
		\nonumber\\& \stepd{\ge} (q - k + \nu + 1)_k(q - k +1)_k\frac{b_q}{2^{2k}\beta^{k-1}(1+\beta)^{k+1}} z^{-k}
	\end{flalign*}
	where (a) is due to \prettyref{lmm:recur_gamma_poisson} which shows that both $\theta^k\Gamma_q^{\nu}$ and $op_k(\Gamma_q^{\nu})$ are linear combinations of $\Gamma_{q-k}, \cdots, \Gamma_{q+k}$, 
	(b) is due to the orthogonality of $\{\Gamma_q^{\nu}\}_{q\ge 1}$ in $S$-norm \cite[(101)]{polyanskiy2021sharp}, 
	(c) uses the closed form of $b_q$ from \prettyref{eq:bn_def} and the (d) follows from plugging in $\gamma=\sqrt{\beta(1+\beta)}$ and using the fact that $1-\left(1-\frac\gamma{1+\beta}\right)^k \ge \frac\gamma{1+\beta}$ (note also $\gamma < 1 + \beta$ given our choice of $\gamma$) when $k\ge 1$. Thus, we established \prettyref{eq:lemma15_2}.
	
	Note that the expansion of $K_k(\Gamma_q^{\nu})$ into linear combinations of $\Gamma_{q-k}, \cdots, \Gamma_{q+k}$ also implies that $K_k(\Gamma_q), K_k(\Gamma_r)$ are orthogonal in $S$-norm whenever $|q - r|\ge 2k + 1$, establishing \prettyref{eq:lemma15_3}. 
	
\end{proof}

With \prettyref{lmm:15_generalized}, we are able to prove \prettyref{lmm:11_generalized} and \prettyref{lmm:12_generalized}.

\begin{proof}[Proof of \prettyref{lmm:11_generalized}]
	Fix $m$ and let 
	\[r_q = \frac{\Gamma_q}{\sqrt{(S_k\Gamma_q,\Gamma_q)}},\quad q\in\calQ=\{m + (2k+1)j: j = 1, 2, \cdots, m\}.\] 
	Note that this definition guarantees \prettyref{eq:lemma11_1} and \prettyref{eq:lemma11_2}. Since $z = \frac1{(\sqrt{1+\beta}+\sqrt\beta)^2}$ and $\beta\ge2$, we have $\frac1{6\beta}\le z\le\frac1{4\beta}\le\frac18$. Then \begin{align*}
		\|Kr_q\|_{L_2(f_0)}^2 &= \frac{(S\Gamma_q,\Gamma_q)}{(S_k\Gamma_q,\Gamma_q)} \le \frac{2^{2k}\beta^{k-1}(1+\beta)^{k+1}z^k}{(q+\nu - k + 1)_k(q - k + 1)_k} \in O_k(\frac{\beta^k}{\alpha^km^k})
	\end{align*}
	where the final bound follows since $z^k = \Theta_k(\beta^{-k})$, $q=\Theta_k(m)$, and $q+\nu=\Omega(\alpha)$. This proves \prettyref{eq:lemma11_3}.\\
	From the proof of \cite[Lemma 11]{polyanskiy2021sharp}, we know ${q+\alpha\choose q}^2b_q^{-1} \le \exp\left\{C'(\alpha+m\log\beta)\right\}$ for some absolute constant $C'$. 
	Using \prettyref{eq:lemma15_1}, \[\max_{q\in\calQ}\|r_q\|_\infty\le\sqrt{\frac{\beta^k}{\alpha^kq^kb_q}}{q+\alpha\choose q} \le \sqrt{\frac{\beta^k}{\alpha^k}}e^{C(\alpha+m\log\beta)}\]
	thus proving \prettyref{eq:lemma11_4}.
\end{proof}
\begin{proof}[Proof of \prettyref{lmm:12_generalized}]
	We choose the same $r_q$ as in the previous proof. We know $\nu=\alpha-1=0$ and $\beta$ is a constant, so \begin{align*}
		\|Kr_q\|_{L_2(f_0)}^2 &= \frac{(S\Gamma_q,\Gamma_q)}{(S_k\Gamma_q,\Gamma_q)} \le \frac{2^{2k}\beta^{k-1}(1+\beta)^{k+1}z^k}{(q - k + 1)_k^2} = O_{\beta,k}\left(\frac{1}{m^{2k}}\right).
	\end{align*}
	From the proof of \cite[Lemma 12]{polyanskiy2021sharp}, we know $b_q\asymp z^q$, so using \prettyref{eq:lemma15_1} with $\alpha=1$, we also have for some constant $c=c(k)$ and $C=C(k)$, \begin{align*}
		\|r_q\|_\infty = \frac{\|\Gamma_q\|_\infty}{\sqrt{(S_k\Gamma_q,\Gamma_q)}} \le \frac{cm}{\sqrt{m^{2k}b_q}} \le cm^{1-k}z^{-(2k+2)m} = m^{1-k}e^{Cm}
	\end{align*}
	since $z < 1$.
\end{proof}

\section{Technical proofs for Poisson model (upper bound)}\label{app:poisson_upper}
\subsection{Proofs related to modified Robbins}\label{app:robbins}

\begin{proof}[Proof of \prettyref{lmm:zero_term}]
	The case $k = 1$ has been shown in \cite[(119)]{polyanskiy2021sharp}, 
	with bound $2\max \{1, h\}$. 
	For $k \ge 2$, 
	we first observe the following: 
	\begin{equation}\label{eq:ratiobound}
		\frac{f_\pi(x + k)}{f_\pi(x)} = \frac{\tpik(x)}{(x + 1)_{k}}
		= \frac{\EE[\theta^k|X = x]}{(x + 1)_{k}} \le \frac{h^k}{(x + 1)_{k}}, \forall x\ge 0
	\end{equation}
	Therefore, 
	\[
	\frac{f_\pi(k)}{(1-f_\pi(0))f_\pi(0)} = \frac{\max\left\{\frac{f_\pi(k)}{f_\pi(0)}, \frac{f_\pi(k)}{1-f_\pi(0)}\right\}}{\max\{f_\pi(0),1-f_\pi(0)\}} \stepa{\le} 2\max\left\{\frac{h^k}{k!},1\right\}\]
	where (a) is due to \prettyref{eq:ratiobound}, 
	$f_{\pi}(k) = 1 - \sum_{x\neq k} f_{\pi}(x)\le 1 - f_{\pi}(0)$, and $\max\{f_{\pi}(0), 1 - f_{\pi}(0)\}\ge \frac 12$. 
	Therefore, for the case $2\le k < h + 1$, 
	we note that $\frac{h^{k - 1}}{(k - 1)!} > 1$, 
	and therefore $\frac{h^k}{k!}\ge \frac{h}{k}\ge \frac{k - 1}{k} \ge \frac 12$, 
	and therefore $\frac{f_\pi(k)}{(1-f_\pi(0))f_\pi(0)}\le \frac{4h^k}{k!}$ whenever $2\le k < h + 1$. 
	
	We now establish another bound for the case $k \ge h + 1$. 
	Denote $r(\theta, k) = \frac{e^{-\theta} \theta^k}{1 - e^{-\theta}}$, 
	and $L = \sup_{0 < \theta\le h} r(\theta, k)$. 
	Then 
	\[
	\frac{f_{\pi}(k)}{1 - f_{\pi}(0)}
	=\frac{\int e^{-\theta}\theta^k/k! d\pi(\theta)}{\int (1 - e^{-\theta})d\pi(\theta)}
	\le \frac{L}{k!}
	\]
	We show that $r(\theta, k)$ is increasing in $\theta$ for all $\theta < k - 1$, 
	which means that $L = r(h, k)$ since $k\ge h + 1$. 
	Note that 
	\[
	\frac{\partial \log r(\theta, k)}{\partial \theta}
	= -1 + \frac{k}{\theta} - \frac{e^{-\theta}}{1 - e^{-\theta}}
	=\frac{k - \theta - ke^{-\theta}}{\theta(1-e^{-\theta})}
	\stepa{\ge} \frac{k - \theta - \frac{k}{1+\theta}}{\theta(1-e^{-\theta})}
	= \frac{k - \theta - 1}{(1+\theta)(1-e^{-\theta})}
	\ge 0
	\]
	for $\theta\le k - 1$, 
	where (a) uses $e^\theta\ge 1 + \theta$ for all $\theta\ge 0$.
	Therefore, the supremum of $\frac{e^{-\theta} \theta^k}{1 - e^{-\theta}}$ is indeed attained at $\theta = h$. 
	Given also that $f_\pi(0) = \int e^{-\theta}d\pi(\theta)
	\ge e^{-h}$, we have 
	\[
	\frac{f_\pi(k)}{(1 - f_\pi(0))f_\pi(0)}
	\le \frac{L}{k!e^{-h}}
	=\frac{e^{-h}h^k}{k!(1 - e^{-h})e^{-h}}
	=\frac{h^k}{k!(1 - e^{-h})}
	\]
	Finally, when $h\ge 1$, $1 - e^{-h}\ge 1 - e^{-1}$ so 
	$\frac{h^k}{1 - e^{-h}}\le \frac{h^k}{1 - e^{-1}}$;
	when $h < 1$, $1 - e^{-h} = \int_0^h e^{-x}dx\ge he^{-h}$, 
	and therefore $\frac{h^k}{1-e^{-h}}\le \frac{h^k}{he^{-h}} 
	= h^{k - 1}e^h \le e$. 
	Combining the two cases gives 
	$\frac{h^k}{1 - e^{-h}}\le \max \{e, \frac{h^k}{1-e^{-1}}\}$. 
	Thus combining this with the case $2\le k < h + 1$ and that 
	$\frac{1}{1 - e^{-1}} < 4$, the claim follows. 
\end{proof}
\begin{proof}[Proof of \prettyref{lmm:17_generalized}]
	For the case $\pi\in \calP([0, h])$, define $\bar f(x)=\frac{h^xe^{-h}}{x!} \ge f_\pi(x)$ for all $x\ge h$, 
	just as in \cite[Lemma 17]{polyanskiy2021sharp}. Then \begin{align*}
		\sum_{x>x_0+k}\bar f(x)^2(x+1)_k \le &\sum_{x>x_0+k} 2^k\bar f(x)^2(x - k + 1)_k 
		\stepa{\le} (2h)^k\sum_{x>x_0+k}\bar f(x-k)\bar f(x)\nonumber\\
		\stepb{\le} &(2h)^k\bar f(x_0)\sum_{x>x_0+k}\bar f(x)
		\stepc{\le} 2(2h)^k\bar f(x_0)^2
	\end{align*}
	where (a) is by the identity $x\bar f(x) = h\bar f(x-1)$, (b) is by the monotonicity of $\bar f$ on the domain $[2h, \infty)$, 
	and (c) is by \cite[(133)]{polyanskiy2021sharp}, 
	i.e. $\sum_{x \ge x_0} \sum \bar{f}(x) \le 2\bar{f}(x_0)$ for all $x_0\ge 2h$. 
	By taking $x_0 = \max\{2h, (c_1 + c_2h)\frac{\log n}{\log \log n}\}$, 
	$x_0$ satisfies $\bar f(x_0)\le\frac1n$ by \prettyref{lmm:poi_tail_bound}, so we obtain \prettyref{eq:lemma17_1}.
	
	We now consider the case $\pi\in\subexpo(s)$. 
	By the proof of \cite[Lemma 17]{polyanskiy2021sharp} (towards the end of the proof), 
	we know $\tilde f(x) \triangleq 2(1+\frac1s)^{-x}\ge f_{\pi}(x)$. 
	
	We now claim the following: for any real $r$ with $0 < r < 1$ and $x_1\ge 0$, 
	\begin{equation}\label{eq:geom_sum}
		\sum_{x\ge x_1} (x + 1)_kr^x \le \frac{r^{x_1}}{(1 - r)^{(k + 1)}}\cdot (e(x_1+k))^k
	\end{equation}
	Indeed, \prettyref{eq:geom_sum} has the following form: 
	\[
	\sum_{x\ge x_1} (x + 1)_kr^x 
	= \frac{\partial ^k}{\partial r^k}\sum_{x\ge x_1 + k} r^x
	= \frac{\partial ^k}{\partial r^k}\left(\frac{r^{x_1 + k}}{1 - r}\right)
	\]
	
	Next, we use the fact that the $k$-th derivative $(ab)^{(k)}$ of $ab$ has the form 
	$\sum_{j=0}^k \sum \binom{k}{j}a^{(k - j)}b^{(j)}$. 
	Therefore, 
	\begin{flalign*}
		\frac{\partial ^k}{\partial r^k}\left(\frac{r^{x_1 + k}}{1 - r}\right)
		& ~= \sum_{j = 0}^k \binom{k}{j}
		\frac{\partial ^{k - j}}{\partial r^{k - j}}\left(r^{x_1 + k}\right) \frac{\partial ^{j}}{\partial r^{j}}\left((1 - r)^{-1}\right)\nonumber\\
		& ~= \sum_{j = 0}^k \binom{k}{j}
		P(x_1 + k, k - j)r^{x_1 + j}\cdot j!\cdot (1 - r)^{-(j + 1)}
		\nonumber\\
		& ~= k!\sum_{j = 0}^k \binom{x_1 + k}{k - j}\frac{r^{x_1 + j}}{(1 - r)^{(j + 1)}}
		\nonumber\\
		&~\stepa{\le} k!\sum_{j = 0}^k \binom{x_1 + k}{k - j}\frac{r^{x_1}}{(1 - r)^{(k + 1)}}
		\nonumber\\
		&~\stepb{=} \frac{r^{x_1}}{(1 - r)^{(k + 1)}}\cdot (e(x_1+k))^k
	\end{flalign*}
	where (a) follows from that 
	$r^{x_1+j}\le r^{x_1}$ and $(1 - r)^{j + 1}\ge (1-r)^{k + 1}$ for $0 < r < 1$, 
	and (b) follows from the following fact about partial sum of binomial coefficients \cite[Lemma A.5]{shalev2014understanding}:
	\[
	\sum_{j = 0}^k \binom{m}{j}
	\le \left(\frac{em}{k}\right)^k
	\]
	and that $k^k > k!$. 
	
	To finish the proof, here we have $r = (1 + \frac{1}{s})^{-2}$, 
	and take $x_1 = c_1(s)\log n$ where $c_1(s) = \frac{1}{\log (1 + \frac{1}{s})}$. 
	Note that 
	\[\log\left(1 + \frac{1}{s}\right) 
	= \frac{1}{s}\int_0^{1/s}\frac{1}{1 + x}dx
	\ge \frac{1}{s}\min_{0\le x\le 1/s} \frac{1}{1 + x}
	= \frac{1}{s}\cdot \frac{s}{s + 1}
	=\frac{1}{s + 1}
	\]
	and therefore $c_1(s)\le s + 1$. 
	Note also that $1 - r = 1 - \frac{s^2}{(1 + s)^2} = \frac{2s + 1}{(s + 1)^2}$. 
	Then we have the bound as 
	\[
	\sum_{x\ge x_1 + k}f_{\pi}(x)^2(x + 1)_k
	\le \sum_{x\ge x_1 + k}2r^x (x + 1)_k
	\le \frac{2(s+1)^{2(k + 1)}}{(2s + 1)^{k + 1}n^2}(e((s + 1)\log n + k))^k
	\]
	
\end{proof}
\begin{proof}[Proof of \prettyref{lmm:moment}]
	For the first bound, note that 
	\[\E_\pi[\theta^{q}] = q\int_0^\infty x^{q-1}\PP_\pi(\theta>x)dx \le 2q\int_0^\infty x^{q - 1}e^{-x/s}dx=2q(q-1)!s^{q}.\]
	For the second bound, using \cite[(44)]{JPW24}, we know $\PP_{f_\pi}(X\ge K)\le\frac32e^{-K\log(1+\frac1{2s})}$. Thus for any $L$, \begin{align*}
		\E[X_{\max}^{q}] &= q\int_0^\infty x^{q-1}\PP(X_{\max}>x)dx \nonumber\\
		&\le 4L^{q} + n\int_L^\infty x^{q-1}\PP(X>x)dx \nonumber\\
		&\le 4L^{q} + \frac{3n}2\int_L^\infty x^{q-1}e^{-x\log(1+\frac1{2s})}dx \nonumber\\
		&\stepa{\le} 4L^{q} + \frac{3n}{2(\log(1+\frac1{2s}))^{q}}\int_{L\log(1+\frac1{2s})}^\infty z^{q-1}e^{-z}dz \nonumber\\
		&\stepb{\le} 4L^{q} + \frac{3n}{2(\log(1+\frac1{2s}))^{q}} 
		q! e^{-L\log(1 + \frac{1}{2s})}\max_{0\le m\le q - 1} (L\log(1 + \frac{1}{2s}))^{m}
	\end{align*}
	where (a) follows from a change of variables with $z=x\log(1+\frac1{2s})$ and (b) follows from the following identity of the indefinite integral: 
	\[
	\int z^{q - 1}e^{-z}dz
	=-\left(\sum_{m = 0}^{q - 1}P(q - 1, m)z^{q - 1 - m}\right)e^{-z}
	\]
	Plugging in $L = \frac{\log n}{\log(1+\frac1{2s})}$ gives the following bound: 
	\[
	\frac{4(\log n)^{q} + \frac 32 q!\max(1, (\log n)^{q - 1})}{(\log (1 + \frac{1}{2s}))^{q}}
	\]
	
\end{proof}

\subsection{Proofs related to minimum distance estimator}\label{app:mindist}
\begin{proof}[Proof of \prettyref{lmm:npmle}]
	First, we note that $\mmse_k$($\pi$)$\le \E[\theta^{2k}] \le \sqrt{\E[\theta^{4k}]} \le \sqrt{M}$. 
	Denote the truncated prior $\pi_h$ as per \prettyref{eq:pi_h}. Then the MSE of $\hat{t}_{\hat{\pi}, k}$
	has the following form: 
	\begin{align*}
		\E_\pi[(\hat{t}_{\hat{\pi}, k}(X) - \theta^k)^2] 
		&= \E_\pi[\hat{t}_{\hat{\pi}, k}(X) - \theta^k)^2\mathbf{1}_{\theta\le h}] +\E_\pi[(\hat{t}_{\hat{\pi}, k}(X) - \theta^k)^2\mathbf{1}_{\theta> h}] \\
		&\stepa{\le} \E_\pi[(\hat{t}_{\hat{\pi}, k}(X) - \theta^k)^2|\theta\le h] 
		+ \sqrt{\E_\pi[(\hat{t}_{\hat{\pi}, k}(X) - \theta^k)^4]\PP_\pi(\theta> h)} \\
		&\stepb{\le} \E_{\pi_h}[(\hat{t}_{\hat{\pi}, k}(X) - \theta^k)^2] +\sqrt{8(\hat h^{4k} + \E_\pi[\theta^{4k}])\PP_\pi(\theta > h)} \\
		&\le \E_{\pi_h}[(\hat{t}_{\hat{\pi}, k}(X) - \theta^k)^2] +\sqrt{8(\hat h^{4k} + M)\PP_\pi(\theta> h)}
	\end{align*}
	where (a) follows from Cauchy-Schwarz and (b) follows from $(a+b)^4 \le 8(a^4+b^4)$ for any real $a,b$.
	Then
	\begin{align}\label{eq:npmle}
		&~\Regret_{\pi,k}(\hat{t}_{\hat{\pi}, k}) \nonumber\\
		=& ~\E_\pi[(\hat{t}_{\hat{\pi}, k}(X) - \theta^k)^2] - \mmse_k(\pi) \nonumber\\
		\le& ~\E_{\pi_h}[(\hat{t}_{\hat{\pi}, k}(X) - \theta^k)^2] - \mmse_k(\pi_h) + \mmse_k(\pi_h) -\mmse_k(\pi) +\sqrt{8(\hat h^{4k} + M)\PP_\pi(\theta> h)}\nonumber\\
		\le& ~\E_{\pi_h}[(\hat{t}_{\hat{\pi}, k}(X) - \hat{t}_{\pi_h, k}(X))^2] +\left(\frac1{\PP_\pi(\theta\le h)}-1\right)\mmse_k(\pi) +\sqrt{8(\hat h^{4k} + M)\PP_\pi(\theta> h)}\nonumber\\
		\stepa{\le}& ~\E_{\pi_h}[(\hat{t}_{\hat{\pi}, k}(X) - \hat{t}_{\pi_h, k}(X))^2] + \frac{\PP_\pi(\theta > h)}{\PP_\pi(\theta\le h)}\sqrt{M} +\sqrt{8(\hat h^{4k} + M)\PP_\pi(\theta> h)}\nonumber\\
		\le& ~\E_{\pi_h}[(\hat{t}_{\hat{\pi}, k}(X) - \hat{t}_{\pi_h, k}(X))^2] +\frac{(1+2\sqrt2)\sqrt{(\hat h^{4k} + M)\PP_\pi(\theta> h)}}{\PP_\pi(\theta\le h)}
	\end{align}
	where (a) follows from $\mmse_k(\pi)\le\sqrt M$.
	The second term is already in the right form, so we bound the first term. 
	We use the following identity on real numbers $a_1, a_2, b_1, b_2$:
	\[
	\frac{a_1}{b_1} - \frac{a_2}{b_2}
	=\frac{a_1(b_2-b_1)}{b_1(b_1+b_2)}
	+ \frac{2(a_1 - a_2)}{b_1+b_2}
	+ \frac{a_2(b_2-b_1)}{b_2(b_1+b_2)}
	\]
	and substitute $a_1 = f_{\hat{\pi}}(x + k)$, $b_1 = f_{\hat{\pi}}(x)$, 
	$a_2 = f_{\pi_h}(x + k)$, $b_2 = f_{\pi_h}(x)$. 
	Therefore, for $K\ge 1$, 
	\begin{align*}
		&~\E_{\pi_h}[(\hat{t}_{\hat{\pi}, k}(X) - \hat{t}_{\pi_h, k}(X))^2\mathbf{1}_{X\le K-k}]
		\\=&~\sum_{x=0}^{K-k}(x+1)_k^2f_{\pi_h}(x)\left(\frac{f_{\hat\pi}(x+k)}{f_{\hat\pi}(x)} - \frac{f_{\pi_h}(x+k)}{f_{\pi_h}(x)}\right)^2
		\\\stepa{\le}& ~3 \sum_{x=0}^{K-k} (x+1)_k^2f_{\pi_h}(x)
		\left(\frac{f_{\hat{\pi}}(x + k)^2(f_{\pi_h}(x)-f_{\hat{\pi}}(x))^2}
		{f_{\hat{\pi}}(x)^2(f_{\pi_h}(x) + f_{\hat{\pi}}(x))^2}
		+ 4\frac{(f_{\pi_h}(x + k)-f_{\hat{\pi}}(x + k))^2}{(f_{\pi_h}(x) + f_{\hat{\pi}}(x))^2}
		\right. 
		\\&\left.
		+ \frac{f_{\pi_h}(x + k)^2(f_{\pi_h}(x)-f_{\hat{\pi}}(x))^2}{f_{\pi_h}(x)^2(f_{\pi_h}(x) + f_{\hat{\pi}}(x))^2}
		\right)
		\\\stepb{\le }& ~3 \sum_{x=0}^{K-k} f_{\pi_h}(x)(\hat{t}_{\hat{\pi}, k}(x)^2 + \hat{t}_{\pi_h, k}(x)^2)
		\frac{(f_{\pi_h}(x)-f_{\hat{\pi}}(x))^2}{(f_{\pi_h}(x) + f_{\hat{\pi}}(x))^2}
		+12 \sum_{x=0}^{K-k} \frac{(x+1)_k^2 (f_{\pi_h}(x + k)-f_{\hat{\pi}}(x + k))^2}{f_{\pi_h}(x) + f_{\hat{\pi}}(x)}
		\\\stepc{\le} & ~3(h^{2k}+\hat{h}^{2k}) \sum_{x=0}^{K-k} f_{\pi_h}(x)
		\frac{(f_{\pi_h}(x)-f_{\hat{\pi}}(x))^2}{(f_{\pi_h}(x) + f_{\hat{\pi}}(x))^2}
		+12 \sum_{x=0}^{K-k} \frac{(x+1)_k^2 (f_{\pi_h}(x + k)-f_{\hat{\pi}}(x + k))^2}{f_{\pi_h}(x) + f_{\hat{\pi}}(x)}
	\end{align*}
	where (a) uses $(x+y+z)^2\le 3(x^2+y^2+z^2)$ for real numbers $x, y, z$, 
	(b) uses $\tpik(x) = (x+1)_k\frac{f_{\pi}(x + k)}{f_{\pi}(x)}$ and that 
	$\frac{f_{\pi_h}(x)}{f_{\pi_h}(x) + f_{\hat{\pi}}(x)} < 1$. 
	and (c) uses $\hat{t}_{\pi_h, k}\le h^k$ and $\hat{t}_{\hat{\pi}, k}\le \hat{h}^k$. 
	
	To relate this bound with $H^2$, 
	we use the fact that $(\sqrt{a}+\sqrt{b})^2\le 2(a + b)$ for all $a, b\ge 0$ to obtain the bound 
	\[(f_{\hat\pi}(x)-f_{\pi_h}(x))^2 \le 2(f_{\hat\pi}(x)+f_{\pi_h}(x))\left(\sqrt{f_{\hat\pi}(x)}-\sqrt{f_{\pi_h}(x)}\right)^2.\] 
	We also see that $(x+1)_k \le K^k$ for $x \le K-k$.
	Therefore, we have 
	\begin{align}\label{eq:npmle_term1}
		&\E_{\pi_h}[(\hat{t}_{\hat{\pi}, k}(X) - \hat{t}_{\pi_h, k}(X))^2\mathbf{1}_{X\le K-k}] \nonumber\\
		\le&6(h^{2k} + \hat h^{2k})\sum_{x=0}^{K-k}(\sqrt{f_{\hat\pi}(x)}-\sqrt{f_{\pi_h}(x)})^2\nonumber\\
		&+ 24K^k\max_{x\le K-k}\frac{(x+1)_kf_{\hat\pi}(x+k) + (x+1)_kf_{\pi_h}(x+k)}{f_{\hat\pi}(x) + f_{\pi_h}(x)}\sum_{x=0}^{K-k}\left(\sqrt{f_{\hat\pi}(x+k)}-\sqrt{f_{\pi_h}(x+k)}\right)^2 \nonumber\\
		\le&\left(6(h^{2k} + \hat h^{2k}) + 24K^k\max_{x\le K-k}\left(f_{\hat\pi}(x) + f_{\pi_h}(x)\right)\right)H^2(f_{\hat\pi}, f_{\pi_h})\nonumber\\
		\le&\left(6(h^{2k} + \hat h^{2k}) + 24K^k(h^k + \hat h^k)\right)H^2(f_{\hat\pi}, f_{\pi_h}).
	\end{align}
	Note that by triangle inequality on Hellinger distance, 
	\[H^2(f_{\hat\pi}, f_{\pi_h}) \le (H(f_{\hat\pi}, f_{\pi}) + H(f_{\pi}, f_{\pi_h}))^2 \le 2H^2(f_{\hat\pi}, f_{\pi}) + 2H(f_{\pi}, f_{\pi_h})^2.\] 
	But \[H(p_{\pi}, f_{\pi_h})^2 \le 2\TV(f_{\pi}, f_{\pi_h}) \le 2\TV(\pi, \pi_h) = 4\PP_\pi(\theta>h)\] where TV is the total variation, the middle inequality is from the data processing inequality, 
	and the last equality $\TV(\pi, \pi_h) = 2\PP_\pi(\theta>h)$ is justified in \cite[Appendix B]{JPW24}. 
	Combining this with \prettyref{eq:npmle_term1}, \begin{align*}
		\E_{\pi_h}[(\hat{t}_{\hat{\pi}, k}(X) - \hat{t}_{\pi_h, k}(X))^2\mathbf{1}_{X\le K-k}] \le
		\left(12(h^{2k} + \hat h^{2k}) + 48K^k(h^k + \hat h^k)\right)\left(H^2(f_{\hat\pi}, p_{\pi})+4\PP_\pi(\theta>h)\right).
	\end{align*}
	We can also bound \begin{align*}
		\E_{\pi_h}[(\hat{t}_{\hat{\pi}, k}(X) - \hat{t}_{\pi_h, k}(X))^2\mathbf{1}_{X> K-k}] \le (h^k+\hat h^k)^2\PP_{f_{\pi_h}}(X>K-k) = (h^k+\hat h^k)^2\frac{\PP_{f_{\pi}}(X>K-k)}{\PP_\pi(\theta\le h)}.
	\end{align*}
	Using $\PP_\pi(\theta\le h)\ge\frac12$, summing these two inequalities, and combining with \prettyref{eq:npmle} gives the desired bound.
\end{proof}

\subsection{Proofs related to ERM}\label{app:erm}
\begin{proof}[Proof of \prettyref{lmm:thetaerm_clipped}]
	We recall that $\Termk(X_1^n) = (\hat{t}_{\mathsf{erm}}(X_1), \cdots, \hat{t}_{\mathsf{erm}}(X_n))$ 
	where $\hat{t}_{\mathsf{erm}}$ is a minimizer of $\hat{R} := \hat{R}_{X_1, \cdots, X_n}$ as defined in \prettyref{eq:erm_obj_emp}. 
	Denote, now, $\Termclipped{k, [a, b]}$ as 
	$(\hat{t}_{[a, b]}(X_1), \cdots, \hat{t}_{[a, b]}(X_n))$, 
	where $\hat{t}_{[a, b]}$ minimizes $\hat{R}$ over $\mathcal{F}_{\uparrow, [a, b]}$.  
	
	We show that there exists a minimizer $\hat{u}$ of $\hat{R}$ among $\mathcal{F}_{\uparrow, [a, b]}$ such that $\hat{u}(x)=a$ for all $x$ such that $\hat{t}_{\mathsf{erm}}(x) \le a$. 
	Suppose this is not satisfied by $\hat{t}_{[a, b]}$. 
	By the monotonicity of both $\hat{t}_{[a, b]}$ and $\hat{t}_{\mathsf{erm}}$, 
	there exist disjoint intervals $I_1, I_2, I_3$ where $I_1 < I_2 < I_3$ such that 
	$\hat{t}_{[a, b]}(x)= a$ and $\hat{t}_{\mathsf{erm}}(x)\le a$ for $x\in I_1$, 
	$\hat{t}_{\mathsf{erm}}(x)\le a < \hat{t}_{[a, b]}(x)$ for $x\in I_2$, and 
	$\hat{t}_{[a, b]}(x)> a$ and $\hat{t}_{\mathsf{erm}}(x)> a$ for $x\in I_3$. 
	Now for $\epsilon\in [0, 1]$ we define the following two functions $\hat{t}_{1, \epsilon}, \hat{t}_{2, \epsilon}$ as follows: 
	\[
	\hat{t}_{1, \epsilon}(x) = 
	\begin{cases}
		\hat{t}_{[a, b]}(x) & x\not\in I_2\\
		(1-\epsilon)\hat{t}_{[a, b]}(x) + \epsilon\hat{t}_{\mathsf{erm}}(x) & x\in I_2\\
	\end{cases}
	\qquad 
	\hat{t}_{2, \epsilon}(x) = 
	\begin{cases}
		\hat{t}_{\mathsf{erm}}(x) & x\not\in I_2\\
		(1-\epsilon)\hat{t}_{\mathsf{erm}}(x) + \epsilon\hat{t}_{[a, b]}(x) & x\in I_2\\
	\end{cases}
	\]
	Note that $\hat{R}$ is separable and also convex in $\hat{t}$. 
	For this purpose, for each subset $I\subseteq \mathbb{R}$ we define $\hat{R}(\hat{t}; I)$ as 
	\[
	\hat{R}(\hat{t}; I)=\frac{1}{n} \sum_{i=1}^n \left(\hat{t}(X_i)^2-2(X_i - k + 1)_k \hat{t}(X_i - k)\right)\indc{X_i\in I}
	\]
	i.e. $\hat{R}=\hat{R}(\cdot, \mathbb{R})$. 
	Note that $\hat{t}_{[a, b]}(x) > \hat{t}_{\mathsf{erm}}(x)$ on $I_2$, 
	and by the definition of the intervals $I_1, I_2, I_3$, 
	there exists $\epsilon_1, \epsilon_2 > 0$
	such that $\hat{t}_{1, \epsilon}\in \mathcal{F}_{\uparrow, [a, b]}$, 
	$\hat{t}_{2, \epsilon}\in \mathcal{F}_{\uparrow}$. 
	Thus we have 
	\begin{flalign*}
		\hat{R}(\hat{t}_{[a, b]}; I_2)&\stepa{=}\hat{R}(\hat{t}_{[a, b]}) - \hat{R}(\hat{t}_{[a, b]}; \mathbb{R}\backslash I_2)
		\nonumber\\&\stepb{\le}\hat{R}(\hat{t}_{1, \epsilon_1}) - \hat{R}(\hat{t}_{1, \epsilon_1}; \mathbb{R}\backslash I_2)
		=\hat{R}(\hat{t}_{1, \epsilon_1}; I_2)
		\stepc{\le} (1-\epsilon_1)\hat{R}(\hat{t}_{[a, b]}; I_2) + \epsilon_1\hat{R}(\hat{t}_{\mathsf{erm}}; I_2)
	\end{flalign*}
	where (a) follows from the separability of $\hat{R}$, 
	(b) follows from that $\hat{t}_{([a, b])}$ 
	is a minimizer of $\hat{R}$ among functions in $\mathcal{F}_{\uparrow, [a, b]}$ and that $\hat{R}(\hat{t}_{[a, b]}; \mathbb{R}\backslash I_2) = \hat{R}(\hat{t}_{1, \epsilon_1}; \mathbb{R}\backslash I_2)$, 
	and (c) follows from the convexity of $\hat{R}$. 
	An immediate consequence of this is that $\hat{R}(\hat{t}_{[a, b]}; I_2)\le \hat{R}(\hat{t}_{\mathsf{erm}}; I_2)$. 
	Similarly, 
	considering the similar operation above on $\hat{R}(\hat{t}_{2, \epsilon_2}; I_2)$ and comparing against $\hat{R}(\hat{t}_{\mathsf{erm}, I_2})$ we have 
	$\hat{R}(\hat{t}_{\mathsf{erm}}; I_2)\le \hat{R}(\hat{t}_{[a, b]}; I_2)$. 
	Thus, in fact all the inequalities above are equality. 
	In particular, $\hat{R}(\hat{t}_{[a, b]}) = \hat{R}(\hat{t}_{1, \epsilon_1})$. 
	Such $\epsilon_1$ can be chosen such that there's some $X_i\in I_2$ where $\hat{t}_{1, \epsilon_1}(X_i) = a$ and $\hat{t}_{1, \epsilon}\in \mathcal{F}_{\uparrow, [a, b]}$. 
	We may then replace $\hat{t}_{[a, b]}$ with this $\hat{t}_{1, \epsilon_1}$, 
	and repeating this process would yield $\hat{t}_{[a, b]}$ with $\hat{t}_{[a, b]}(x)=a$ whenever 
	$\hat{t}_{\mathsf{erm}}(x)<a$. 
	In a similar way, we could also choose $\hat{t}_{[a, b]}$ such that 
	$\hat{t}_{[a, b]}(x) = b$ whenever $\hat{t}_{\mathsf{erm}}(x) > b$. 
	
	We are now left with $x_i$'s such that $a < \hat{t}_{\mathsf{erm}}(x_i) < b$. 
	Let the set of such $x_i$ be $I$, then 
	$\hat{t}_{\mathsf{erm}}|_I$ on this set $I$ is a minimizer of $\hat{R}$ among $\hat{t}\in\mathcal{F}_{\uparrow, [a, b]}$. Given also that $\hat{t}_{[a, b]}$ now takes only values $a$ or $b$ outside $I$, 
	it follows that we may also take $t_{[a, b]}|_I = \hat{t}_{\mathsf{erm}}|_I$. 
\end{proof}

Before proving \prettyref{lmm:erm_bound_T}, 
we introduce the following auxiliary step, which generalizes a key step in part of a proof of \cite[Lemma 4]{JPTW23}. 
The proof is deferred to the end of this subsection. 
\begin{lemma}\label{lmm:sos-rademacher}
	Let $h, \hat{h}, b > 0$ and $X_1, X_2, \cdots, X_n$ be fixed. 
	Denote $X_{\max} = \max\{X_1, \cdots, X_n\}$. 
	Let $\hat{t}_{\mathsf{oracle}}$ be a function taking values in $[0, h]$. 
	Let $\mathcal{F}$ be a class of monotone functions taking values in $[0, \hat{h}]$. 
	Then for $\epsilon_1, \cdots, \epsilon_n$ as $n$ independent Rademacher symbols, we have 
	\[
	\mathbb{E}\left[\sup_{\hat{t}\in\mathcal{F}}\sum_{i=1}^n (\epsilon_i - \frac 1b)(\hat{t}_{\mathsf{oracle}}(X_i) - \hat{t}(X_i))^2\right]
	\le c(\hat{h}^2 + h^2(1+X_{\max}))
	\]
	where the expectation is taken over $\epsilon_i$ and $c:=c(b) > 0$ is a constant that depends only on $b$. 
\end{lemma}

\begin{proof}[Proof of \prettyref{lmm:erm_bound_T}]
	Recall $N(x)$ is the sample frequency and define the quantity \begin{align}\label{eq:epsilon_x}
		\epsilon(x) = \sum_{i=1}^n\epsilon_i\mathbf{1}_{X_i=x}.
	\end{align}
	Recall also the expression $S(t_1, t_2, x)$ as defined in \prettyref{eq:s_erm}. 
	We first prove the bound on $R_2(b,n)$. 
	Defining $\hat{t}(x) = 0$ and $\tpik(x) = 0$ for $x<0$, we have
	\begin{align}\label{eq:T2_bn}
		&\sum_{i=1}^n2\epsilon_iS(\hat{t}, \tpik, X_i)-\frac1b(\tpik(X_i)-\hat{t}(X_i))^2\nonumber\\
		=&\sum_{x\ge0}2\epsilon(x)S(\hat{t}, \tpik, x) - \frac{N(x)}b(\tpik(x)-\hat{t}(x))^2\nonumber\\
		=&\sum_{x\ge0}2(\epsilon(x)\tpik(x)-(x+1)_k\epsilon(x+k))(\hat{t}(x) - \tpik(x))-\frac{N(x)}b(\tpik(x)-\hat{t}(x))^2
	\end{align}
	Note that due to the identity $2ax - bx^2\le \frac{a^2}{b}$ for all $b > 0$, 
	for any $x$ with $N(x) > 0$ we have 
	\[
	2(\epsilon(x)\tpik(x)-(x+1)_k\epsilon(x+k))(\hat{t}(x) - \tpik(x))-\frac{N(x)}b(\tpik(x)-\hat{t}(x))^2
	\ge b\cdot \frac{(\epsilon(x)\hat{t}_{\pi, k}(x) - (x + 1)_k\epsilon(x+k))^2}{N(x)}
	\]
	which is a bound independent of $\hat{t}$. 
	Therefore we may substitute \prettyref{eq:T2_bn} back into $R_2(b,n)$ to get $R_2(b, n)\le J_1(n) + J_0(n)$ where 
	\begin{align}\label{eq:j1}
		J_1(n) = \E\left[\sum_{x\ge 0}\frac{(\epsilon(x)\hat{t}_{\pi, k}(x) - (x + 1)_k\epsilon(x+k))^2}{N(x)}\indc{N(x) > 0}\right]
	\end{align}
	\begin{align}\label{eq:j0}
		J_0(n) = b\E\left\{\sup_{f\in\calF_*\cup\calF_*'}\left[\sum_{x\ge0}2(x+1)_k\epsilon(x+k)(\tpik(x)-\hat{t}(x))\mathbf{1}_{N(x)=0}\right]\right\}
	\end{align}
	We first bound $J_1(n)$. 
	Note that $\epsilon_1, \cdots, \epsilon_n$ are independent of $X_1, \cdots, X_n$. 
	Thus for each $x\ge 0$, by the definition of $\epsilon(x)$ defined in \prettyref{eq:epsilon_x}, 
	$\epsilon(x) | X_1, \cdots, X_n\sim 2\text{Binom}(N(x), \frac 12) - N(x)$. 
	Thus $\E[\epsilon(x)|X_1, \dots, X_n] = 0$, $\E[(\epsilon(x))^2|X_1,\dots,X_n] = N(x)$, 
	and $\E[\epsilon(x)\epsilon(x + k)|X_1, \dots, X_n] = 0$. 
	This gives 
	\begin{align}\label{eq:t1_v3}
		J_1(n)\le &b\cdot\E\left[\sum_{x\ge0}\frac{(\epsilon(x)\tpik(x)-(x+1)_k\epsilon(x+k))^2}{N(x)}\mathbf{1}_{N(x)>0}\right]\nonumber\\
		=&b\cdot\E\left[\sum_{x\ge0}\left((\tpik(x))^2+\frac{(x+1)_k^2N(x+k)}{N(x)}\right)\mathbf{1}_{N(x)>0}\right]
		\nonumber\\
		=&b\cdot \E\left[\sum_{x\ge0}(\tpik(x))^2\mathbf{1}_{N(x)>0}\right]
		+b\E\left[\sum_{x\ge0}\frac{(x+1)_k^2N(x+k)}{N(x)}\mathbf{1}_{N(x)>0}\right]
	\end{align}
	The first term can be bounded by observing that $\tpik(x)^2\le h^{2k}$ and 
	$\mathbf{1}_{N(x)>0} = 0$ for all but $1+X_{\max}$ of $x$'s, thus 
	$\E\left[\sum_{x\ge0}(\tpik(x))^2\mathbf{1}_{N(x)>0}\right]\le h^{2k}\mathbb{E}[1+X_{\max}]$. 
	For the second term, recall that conditioned on $N(x)$, $\E[N(x+k) | N(x)] = \frac{(n - N(x))f_{\pi}(x+k)}{1-f_{\pi}(x)}$, 
	so we may continue the expansion as 
	\begin{align}\label{eq:t1_v4}
		\E\left[\sum_{x\ge0}\frac{(x+1)_k^2N(x+k)}{N(x)}\mathbf{1}_{N(x)>0}\right]
		&=\sum_{x\ge 0} (x+1)_k^2\frac{(n-N(x))f_{\pi}(x + k)}{1 - f_{\pi}(x)}\E\left[\frac{\indc{N(x) > 0}}{N(x)}\right]
		\nonumber\\&\stepa{\le} c_1 (x+1)_k^2 \frac{f_{\pi}(x + k)}{(1 - f_{\pi}(x))f_{\pi}(x)}\min\{1, (nf_{\pi}(x))^2\}
		\nonumber\\&\stepb{\le} c_1 (k!)^2 \frac{f_{\pi}(k)}{(1-f_{\pi(0)})f_{\pi}(0)}
		+ c_2\sum_{x\ge 1}(x+1)_k^2 \frac{f_{\pi}(x + k)}{f_{\pi}(x)}\min\{1, (nf_{\pi}(x))^2\}
		\nonumber\\&\stepc{\le} c_3 k!(h^k + 1)
		+ c_2h^k\sum_{x\ge 1}(x+1)_k\min\{1, (nf_{\pi}(x))^2\}
	\end{align}
	where $c_1, c_2, c_3 > 0$ are absolute constants. 
	Here (a) is due to \prettyref{eq:binom_inverse_bound}, (b) due to \prettyref{eq:poi_sterling}, 
	(c) due to \prettyref{lmm:zero_term} and that 
	$(x+1)_k\frac{f_{\pi}(x + k)}{f_{\pi}(x)}=\tpik(x)\le h^k$. 
	We further bound the second term in \prettyref{eq:t1_v4} as follows: 
	\begin{align}\label{eq:t1_term3}
		&h^k\sum_{x\ge1}(x+1)_k\min\{(nf_{\pi}(x))^2,1\} \le h^kM^{k+1} + h^k\sum_{x\ge M}(x+1)_k\min\{(nf_{\pi}(x))^2,1\}\nonumber\\
		&\stepa{\le} h^kM^{k+1}+2^kn^2h^k\sum_{x\ge M}(x - k + 1)_k(f_{\pi}(x))^2\stepb{\le} h^kM^{k+1}+2^kn^2h^{2k}\PP_{X\sim f_{\pi}}[X>M]\nonumber\\
		&\le2^k\left(h^kM^{k+1}+\frac{h^{2k}}{n^5}\right)
		\le 2^k(h^k + M^{k + 1} + h^{2k})
	\end{align}
	where (a) we used the crude inequality $(x + 1)_k \le 2^k (x - k + 1)_k$ for $x\ge M\ge k$,  and
	(b) is because $$(x - k + 1)_kf_{\pi}(x) = \tpik(x-k)f_\pi(x-k) \le h^k.$$ The $\frac{h^{2k}}{n^5}$ term vanishes asymptotically, so substituting \prettyref{eq:t1_term3} back into \prettyref{eq:t1_v4}, we obtain 
	\begin{align}\label{eq:t1_v5}
		\frac1b J_1(n)\le h^{2k}(2^k + c_1(1)M) + k!(4h^k + e) + 2^kh^kM^{k+1}.
	\end{align}
	
	Now we bound $J_0(n)$. We know $|\epsilon(x+k)|\le N(x+k)=0$ for $x\ge X_{\max}-k+1$. We also use the fact that for $\hat{t}\in \mathcal{F}_r\cup \mathcal{F}'_r$, 
	$|\tpik(x) - \hat{t}(x)|\le \tpik(x) + r(x) + r(x') \le h^k + r(X_{\max}) 
	+ r(X_{\max}')$ for all $x\le X_{\max}$. 
	Thus
	\begin{align}\label{eq:t0_v2}
		J_0(n)\le&~\E\left[\sum_{x\ge0}2(x+1)_kN(x+k)\sup_{f\in\calF_*\cup\calF_*'}|\tpik(x)-\hat{t}(x)|\mathbf{1}_{N(x)=0}\right]\nonumber\\
		\le&~\E\left[\sum_{x=0}^{X_{\max} -k}2(x+1)_k(h^k + r(X_{\max}) 
		+ r(X_{\max}'))N(x+k)\mathbf{1}_{N(x)=0}\right]
	\end{align}
	Define $X_{\max}'$ as the $\max \{X_1', \cdots, X_n'\}$ 
		(i.e. the maximum element of the fresh samples). 
	Let $A=\{X_{\max}\le M,X_{\max}'\le M\}$. Then $\PP[A^C]\le\frac2{n^6}$ by the union bound.
	Thus, for some absolute constant $c > 0$: 
	\begin{align}\label{eq:t0_v3}
		&\E\left[\sum_{x=0}^{X_{\max}-k}2(x+1)_k(h^k + r(X_{\max}) 
		+ r(X_{\max}'))N(x+k)\mathbf{1}_{N(x)=0}\mathbf{1}_{A^C}\right]\nonumber\\
		\le&~\E\left[X_{\max}^k(h^k + r(X_{\max}) 
		+ r(X_{\max}'))\sum_{x=0}^{X_{\max}-k}N(x+k)\mathbf{1}_{N(x)=0}\mathbf{1}_{A^C}\right]\nonumber\\
		\stepa{\le}&~n\E\left[X_{\max}^k(h^k + r(X_{\max}) 
		+ r(X_{\max}'))\mathbf{1}_{A^C}\right]\nonumber\\
		\stepb{\le}&~n\sqrt{\E\left[X_{\max}^{2k} (h^k + r(X_{\max}) 
			+ r(X_{\max}'))^2\right]}\sqrt{\PP[A^C]}\nonumber\\
		\stepc{\le} & ~n\sqrt{3}
		(h^k\sqrt{\bbE[X_{\max}^{2k}]} + \sqrt{\bbE[X_{\max}^{2k} r(X_{\max})^2]} + \sqrt{\bbE[X_{\max}^{2k} r(X'_{\max})^2]})\cdot \frac{1}{n^3}\nonumber\\
		\stepd{\le} &~\frac{c(h^k\sqrt{c_1(2k)}M^k + \sqrt{\bbE[X_{\max}^{2k} r(X_{\max})^2]} + \sqrt{c_1(2k)c_2(k, 2)}M^kL)}{n^2}
	\end{align}
	where (a) follows from $\sum_{x = 0}^{X_{\max} - k} N(x + k) \le \sum_{x = 0}^{\infty} N(x) = n$, and 
	(b) follows from Cauchy-Schwarz,
	(c) follows from 
	\[
	\sqrt{\bbE[(A + B + C)^2]}
	\le \sqrt{3}\sqrt{\bbE[A^2 + B^2 + C^2]}
	\le \sqrt{3}(\sqrt{\bbE[A^2]} + \sqrt{\bbE[B^2]} + \sqrt{\bbE[C^2]})
	\]
	and (d) uses $\bbE[X^{2k}_{\max}r(X'_{\max})^2] = 
	\bbE[X^{2k}_{\max}]\bbE[r(X'_{\max})^2]$ due to the independence of 
	$X_{\max}$ and $X'_{\max}$. 
	
	For each $x\le M$, define $q_{\pi,M}(x)=\frac{f_{\pi}(x)}{\PP_{X\sim f_{\pi}}[X\le M]}$. Note that $\PP[N(x)=0|A]=(1-q_{\pi,M}(x))^n$ and conditioned on $A$ and $N(x)=0$, the random variable $N(x+k) \sim \text{Binom}\left(n,\frac{q_{\pi,M}(x+k)}{1-q_{\pi,M}(x)}\right)$. Then 
	\begin{align}\label{eq:t0_v4}
		&~\E\left[\sum_{x=0}^{X_{\max}-k}2(x+1)_k(h^k + r(X_{\max}) 
		+ r(X_{\max}'))N(x+k)\mathbf{1}_{N(x)=0}\mathbf{1}_{A}\right]\nonumber\\
		\le&~\E\left[\sum_{x=0}^{X_{\max}-k}2(x+1)_k(h^k + r(X_{\max}) 
		+ r(X_{\max}'))N(x+k)\mathbf{1}_{N(x)=0}\bigg|A\right]\nonumber\\
		\le&~\sum_{x=0}^{M-k}2(x+1)_k(h^k+2r(M))\E\left[N(x+k)\mathbf{1}_{N(x)=0}|A\right]\nonumber\\
		=&~\sum_{x=0}^{M-k}2(x+1)_k(h^k+2r(M))\E\left[N(x+k)|N(x)=0,A\right]\PP[N(x)=0|A]\nonumber\\
		\le &~\sum_{x=0}^{M-k}2(x+1)_k(h^k+2r(M))\frac{nq_{\pi,M}(x+k)}{1-q_{\pi,M}(x)}(1-q_{\pi,M}(x))^n\nonumber\\
		\stepa{=}&~\sum_{x=0}^{M-k}2(h^k+2r(M))\tpik(x)nq_{\pi,M}(x)(1-q_{\pi,M}(x))^{n-1}\nonumber\\
		\le&~2Mh^k(h^k+2r(M))
	\end{align}
	where (a) is because $\tpik(x)\le h^k$ for all $x$ and 
	$nw(1 - w)^{n - 1}\le (1- \frac{1}{n})^{n - 1} < 1$ for all $w\in [0, 1]$. 
	Summing \prettyref{eq:t0_v3} and \prettyref{eq:t0_v4} and continuing \prettyref{eq:t0_v2}, we have 
	$$J_0(n)\le\frac{c(h^k\sqrt{c_1(2k)}M^k + \sqrt{\bbE[X_{\max}^{2k} r(X_{\max})^2]} + \sqrt{c_1(2k)c_2(k, 2)}M^kL)}{n^2} + 2Mh^k(h^k+2r(M))$$ and combining with \prettyref{eq:t1_v5}, we obtain the desired bound 
	\begin{flalign*}
		R_2(b,n) &\le ~h^{2k}(1 + M) + k!(4h^k + e) + 2^kh^kM^{k+1}
		+2Mh^k(h^k+2r(M)) \nonumber\\
		&\quad 
		+\frac{c(h^k\sqrt{c_1(2k)}M^k + \sqrt{c_1(4k)}M^{2k} + \sqrt{c_2(k, 4)}L^2)}{n^2}
		\nonumber\\
		&\lesssim ~ h^{2k}(2^k + c_1(1)M) + (2h)^kM^{k + 1} + Mh^kr(M) + k!(h^k + 1) \nonumber\\
		& + \frac{h^k\sqrt{c_1(2k)}M^k + \sqrt{\bbE[X_{\max}^{2k} r(X_{\max})^2]} + \sqrt{c_1(2k)c_2(k, 2)}M^kL}{n^2}
	\end{flalign*}
	
	To bound $R_1(b, n)$, 
	we will utilize \prettyref{lmm:sos-rademacher} and note that we may use $h^k$ in place of $h$ and $r(X_{\max})$ in place of $\hat{h}$. Thus for some constant $c:= c(b)$
	\[
	R_1(b, n) \le c\mathbb{E}[h^{2k}(1+X_{\max}) + r(X_{\max})^2]
	=c'(h^{2k}(1+M) + c_2(k, 2)L^2)
	\]
	as desired. 
\end{proof}
\begin{proof}[Proof of \prettyref{lmm:sos-rademacher}]
	Set $m_b=b+1$, and recall the definition of $\epsilon(x)$ at \prettyref{eq:epsilon_x}. 
	Then for each $\hat{t}$, we have 
	\begin{align}
		&\sum_{i=1}^n \left(\epsilon_i -\frac 1b\right)(\hat{t}_{\mathsf{oracle}}(X_i) - \hat{t}(X_i))^2
		\nonumber\\=&~\sum_{x\ge 0, N(x)>0} \left(\epsilon(x) -\frac {N(x)}{b}\right)(\hat{t}_{\mathsf{oracle}}(x) - \hat{t}(x))^2
		\nonumber\\ 
		\le&~ m_b^2h^2\sum_{x=0}^{X_{\max}}\max\{\epsilon(x) - \frac 1b N(x), 0\}
		+\sum_{x=0}^{X_{\max}}\left(\epsilon(x) - \frac 1b N(x)\right)(\hat{t}_{\mathsf{oracle}}(x) - \hat{t}(x))^2\indc{\hat{t}(x) > m_bh}
	\end{align}
	By \cite[Lemma 5]{JPTW23}, we may bound the expectation of the first term by 
	\[
	\mathbb{E}\left[\sum_{x=0}^{X_{\max}} \max\{\epsilon(x) - \frac 1b N(x), 0\}\right]
	\le N_b(1 + X_{\max})
	\qquad N_b\triangleq \frac{1 - \frac{1}{b}}{e\cdot D(\frac{1 + \frac{1}{b}}{2} || \frac 12)}
	\]
	Now if $\hat{h}\le m_bh$ then the second term is 0 and we are done. 
	Otherwise, assume $\hat{h}> m_bh$. 
	To bound the second term, 
	for $x$ satisfying 
	$[m_bh, \hat{h}]$, we have $\frac{m_b-1}{m_b}\hat{t}(x) \le \hat{t}(x) - \hat{t}_{\mathsf{oracle}}(x) \le \hat{t}(x)$. 
	Therefore, we may define $g: [-1, 1]\to\mathbb{R}$ as 
	\[g(x) = \max\left\{\left(x-\frac{1}{b}\right), \left(\frac{m_b-1}{m_b}\right)^2\left(x-\frac{1}{b}\right)\right\},\]
	and have the following inequality 
	\[(\epsilon(x)-\frac{1}{b}N(x))(\hat{t}(x)-\tpik(x))^2 
	\le
	g\left(\frac{\epsilon(x)}{N(x)}\right)N(x)\hat{t}(x)^2
	\le \frac{2b^2-1}{2b(b+1)}\left(\frac{\epsilon(x)}{N(x)} - \frac{1}{2b^2-1}\right)N(x)\hat{t}(x)^2
	\]
	with the right inequality due to that $g$ is convex and $\frac{\epsilon(x)}{N(x)} \in [-1,1]$.
	Therefore, $g$ is bounded by the line connecting the two endpoints $(-1, -\frac{b}{b+1}), (1, \frac{b-1}{b})$. 
	In addition, given the monotonicity of $\hat{t}$, 
	there exists $v$ such that $\hat{t}(x)> m_bh^k$ if and only if $x \ge v$. Therefore, we continue from the display to get the bound 
	\[
	\sum_{x: N(x) > 0}\left(\epsilon(x) - \frac 1b N(x)\right)(\hat{t}_{\mathsf{oracle}}(x) - \hat{t}(x))^2\indc{\hat{t}(x) > m_bh}
	\le c_3(b)\sum_{x=v}^{X_{\max}}\left(\epsilon(x) - \frac{1}{2b^2-1}N(x)\right)\hat{t}(x)^2. 
	\]
	with $c_3(b) = \frac{2b^2-1}{2b(b+1)}$. 
	Fixing $v$, the optimization problem of taking the supremum can be viewed as an $X_{\max}-v+1$-th dimensional linear programming problem with unknowns being the values $\hat{t}(v)^2, \cdots, \hat{t}(X_{\max})^2$. The set of solutions $m_b^2h^{2} \le \hat{t}(v)^2 \le \dots \le \hat{t}(X_{\max})^2 \le \hat{h}^2$ is convex, so the optimum value must occur at one of the extreme points, 
	i.e. there is a $w$ with $\hat{t}(x)=m_bh$ for $x\le w$ and $\hat{t}(x)=\hat{h}$ for $x>w$. Therefore, 
	\begin{flalign*}
		&~\sup_{\hat{t}} \sum_{x=v}^{X_{\max}}\left(\epsilon(x) - \frac{1}{2b^2-1}N(x)\right)\hat{t}(x)^2
		\\=&~\sup_{w} 
		(m_bh)^2\sum_{x=v}^{w}\left(\epsilon(x) - \frac{1}{2b^2-1}N(x)\right) + \hat{h}^2\sum_{x=w+1}^{X_{\max}}\left(\epsilon(x) - \frac{1}{2b^2-1}N(x)\right)
		\\\stepa{\le} &~\sup_{w}  \hat{h}^2\sum_{x=w+1}^{X_{\max}}\left(\epsilon(x) - \frac{1}{2b^2-1}N(x)\right)
	\end{flalign*}
	where (a) is due to that $\sum_{x=v}^{w}\left(\epsilon(x) - \frac{1}{2b^2-1}N(x)\right)\le 0$, otherwise the objective value would increase by $\hat{t}(x)=\hat{h}$ for all $v$. 
	By \cite[Lemma 6]{JPTW23}, we have 
	\[
	\E\left[\sup_{w\ge 0}\left\{\sum_{x>w}^{X_{\max}}\left(\epsilon(x) - \frac{1}{2b^2-1}N(x)\right)\right\}\right]\le \sup_{w:0\le w\le n}(\epsilon_{w+1}+\dots+\epsilon_{n})-\frac{1}{2b^2-1}(n-w) \le c\] for some constant $c := c(b)$. 
	Thus combining the bounds for two terms give $c'(\hat{h}^2 + h^2(1+X_{\max}))$ for some constant $c':=c(b)$. 
\end{proof}

\section{Technical proofs for normal means model (upper bound)}\label{app:normal_means_ub_proof}
\begin{proof}[Proof of \prettyref{lmm:gsn_poly_bound}]
	Throughout the proof we use the shorthand $g := g_{\pi_1, \pi_2}$. 
	Denote the Hermite polynomial $H_j(x) = (-1)^j e^{-\frac{x^2}{2}}\frac{d^j}{dx^j} e^{-\frac{x^2}{2}}$. 
	We then have the following starting point (e.g. \cite[Lemma 9]{wu2020optimal}) : using the identity 
	$\varphi(x - u) = \varphi(x)\sum_{j\ge 0} H_j(x)\frac{u^j}{j!}$ \cite[22.9.17]{AS64}, we get 
	\begin{equation}\label{eq:g_hermite}
		g(x) = \frac{f_{\pi_1}(x) - f_{\pi_2}(x)}{\varphi(x)} = \sum_{j\ge 0} \frac{m_j(\pi_1) - m_j(\pi_2)}{j!} H_j(x)
	\end{equation}
	where $m_j(G)$ denotes the $j$-th moment of the prior $G$. 
	
	Next, we take $p$ as the degree-$q$ truncation of the Hermite polynomial expansion in \prettyref{eq:g_hermite} to obtain 
	$p(x) = \sum_{j = 0}^q \frac{m_j(\pi_1) - m_j(\pi_2)}{j!} H_j(x)$, which is a degree $q$ polynomial in $x$. 
	Then the residual $g - p$ can be written as 
	$\sum_{j > q} \frac{m_j(\pi_1) - m_j(\pi_2)}{j!} H_j(x)$. 
	We furthermore use the properties of the Hermite polynomial under the standard Gaussian norm \cite[7.374]{gradshteyn2014table}: 
	\[
	\mathbb{E}_{\mathcal{N}(0, 1)}[H_i(\theta)H_j(\theta)] = j! \boldsymbol{1}_{i = j}
	\qquad 
	H_j' = jH_{j - 1}
	\]
	In particular, the $k$-th derivative of $H_j$ for $j\ge k$ satisfies $\frac{H_j^{(k)}(x)}{j!}=\frac{j(j-1)\cdots (j-k+1)H_{j-k}(x)}{j!} = \frac{H_{j-k}(x)}{(j-k)!}$
	We therefore obtain the following: 
	\begin{align*}
		||g^{(k)} - p^{(k)}||_{L^2(\varphi)}^2
		& = ||\sum_{j > q} \frac{m_j(\pi_1) - m_j(\pi_2)}{(j-k)!} H_{j - k}||_{L^2(\varphi)}^2
		= \sum_{j > q} \frac{(m_j(\pi_1) - m_j(\pi_2))^2}{(j-k)!}
		\\& \le \sum_{j > q} \frac{4e^{j - k}h^{2j}}{(j - k)^{j - k}}
		\le 4h^{2k} \sum_{j > q} (\frac{eh^2}{q - k + 1})^{j - k}
		\le 4h^{2k}(\frac{eh^2}{q - k + 1})^{q - k}
	\end{align*}
	using the fact $q \ge 2eh^2 + k$, and that $|m_j(\pi_1) - m_j(\pi_2)|\le 2h^j$ for all $j\ge 1$. 
	
	Finally, we use \cite[(25)]{chen2026sharp} to obtain $w\le \exp(\frac{h^2}{2})\varphi$ using that the second moments of $\pi_1$ and $\pi_2$ are at most $h^2$. 
	This means $||g^{(k)} - p^{(k)}||^2_{L^2(w)}\le 4\exp(\frac{h^2}{2})h^{2k}(\frac{eh^2}{q - k + 1})^{q - k}$, as desired. 
\end{proof}

\begin{proof}[Proof of \prettyref{lmm:subgaussian-poly-bound}]
	Similar to \cite[Lemma 12]{chen2026sharp}, we consider a threshold $L$ and for each prior $G$, 
	consider $G = G_{\le L} + G_{ > L}$ where $G_{\le L}$ is restriction of $G$ to $[-L, L]$, 
	and define the convolution $f_{G_{\le L}}, f_{G_{> L}}$ similarly. 
	Note that $G_{\le L}$, $G_{ > L}$, $f_{G_{\le L}}, f_{G_{> L}}$ are not probability measures. 
	Nevertheless, the convolution and differentiation identities used below remain valid by linearity. 
	Concretely, denote 
	\[
	g_{\le L} := \frac{f_{(\pi_1)_{\le L}} - f_{(\pi_2)_{\le L}}}{\varphi}
	\qquad 
	g_{> L} := \frac{f_{(\pi_1)_{> L}} - f_{(\pi_2)_{> L}}}{\varphi}
	\]
	For each $q\ge 2eh^2+k$, 
	we can find a degree $q$ polynomial $p$ such that by \prettyref{lmm:gsn_poly_bound}, we have 
	\[
	||g_{\le L}^{(k)} - p^{(k)}||_{L^2(w)} \le 4\cdot e^{\frac{L^2}{2}}L^{2k}(\frac{eL^2}{q-k+1})^{q-k}
	\]
	Next, the proof of \prettyref{lmm:gsn-alternate-form} can also be used to establish 
	\[
	\varphi(x)\cdot \left(\frac{f_{G > L}}{\varphi}\right)^{(k)}(x) = \int_{\theta > L} \theta^k \varphi(x - \theta)dG(\theta)
	\]
	and also 
	\begin{flalign*}
		||\left(\frac{f_{\pi_1 > L}}{\varphi}\right)^{(k)}||_{L_2(w)}^2
		& = \int \left(\left(\frac{f_{\pi_1 > L}}{\varphi}\right)^{(k)}(x)\right)^2 w(x)dx
		\\& \stepa{\le} 2\int \left(\left(\frac{f_{\pi_1 > L}}{\varphi}\right)^{(k)}(x)\right)^2\frac{\varphi^2(x)}{f_{\pi_1}(x)}dx
		\\& = 2\int \frac{\left(\int_{\theta > L} \theta^k \varphi(x - \theta)d\pi_1(\theta)\right)^2}{f_{\pi_1}(x)}dx
		\\&\stepb{\le} 
		2\int \frac{\left(\int_{\theta > L} \theta^{2k} \varphi(x - \theta)d\pi_1(\theta)\right)\left(\int_{\theta > L} \varphi(x - \theta)d\pi_1(\theta)\right)}{f_{\pi_1}(x)}dx
		\\&\stepc{\le} 2\int \left(\int_{\theta > L} \theta^{2k} \varphi(x - \theta)d\pi_1(\theta)\right) dx
		\\& = 2\int_{|u| > L} u^{2k}d\pi_1(u)
	\end{flalign*}
	where (a) is due to $w(x) = \frac{\varphi^2}{(f_{\pi_1}+f_{\pi_2})/2}\le 2\frac{\varphi^2}{f_{\pi_1}}$, 
	(b) is using Cauchy-Schwarz, and (c) is using $f_{G_{> L}}\le f_G$. 
	Thus we get 
	\[
	||g^{(k)}_{>L}||_{L_2(w)}^2
	\le 2\left|\left|\left(\frac{f_{\pi_1 > L}}{\varphi}\right)^{(k)}\right|\right|_{L_2(w)}^2 
	+ 2\left|\left|\left(\frac{f_{\pi_2 > L}}{\varphi}\right)^{(k)}\right|\right|_{L_2(w)}^2
	\le 4\mathbb{E}_{\pi_1}[\theta^{2k} \indc{|\theta > L|}] + 4\mathbb{E}_{\pi_2}[\theta^{2k} \indc{|\theta > L|}]
	\]
	which is bounded by $O(L^{2k}\exp(-cL^2))$ for subgaussian priors. 
	Finally, choosing $L = c_0\sqrt{q-k}$ for some small $c_0 > 0$ yields both quantities $L^{2k}\exp(-cL^2)$ and $e^{\frac{L^2}{2}}L^{2k}(\frac{eL^2}{q-k+1})^{q-k}$ to be of order $\exp(-c(q-k))$ for some $c > 0$. 
	
\end{proof}

\begin{proof}[Proof of \prettyref{lmm:moment-bound}]
	We start with the proof of \cite{chen2026sharp} (in the Proof of Lemma 15, after (46)), that 
	\[
	f_{\pi}(x) = \mathbb{E}_\pi[\varphi(x-\theta)]\ge \frac{1}{2\sqrt{2\pi}}e^{-\frac{(|x|+a)^2}{2}}
	\]
	where $a$ satisfies $\pi([-a, a])\ge \frac{1}{2}$. 
	Given that $\pi$ is $s$-subgaussian, $a$ can be taken as $a = c(1+s)$ for some absolute constant $c$. 
	Now, w.l.o.g. consider $x \ge 0$. 
	Choose threshold $M = 2x + a + 2k$, we may bound $|\hat{t}_{\pi, k}(x)|$ as follows: 
	\begin{flalign*}
		|\hat{t}_{\pi, k}(x)|f_{\pi}(x)
		&\le \mathbb{E}_{\pi}[|\theta^k|\varphi(x-\theta)]
		\\&\le M^kf_{\pi}(x) + \mathbb{E}_{\pi}[|\theta^k|\varphi(x-\theta) \indc{\theta > M}]
		\\&\stepa{\le} M^kf_{\pi}(x) + \mathbb{E}_{\pi}[M^k\varphi(M-x) \indc{|\theta| > M}]
		\\&\le M^kf_{\pi}(x) + M^k\varphi(M-x)
		\\&\stepb{\le} M^kf_{\pi}(x) + 2M^kf_{\pi}(x)
	\end{flalign*}
	where (a) is because for $\theta > 0$, 
	$\theta^k\varphi(x-\theta)$ attains maximal value at $\theta = \frac{x+\sqrt{x^2+4k}}{2} < M$, 
	and (b) is because $\varphi(M-x)=\varphi(x+a+2k) < e^{-(\frac{(x+a)^2}{2})} \le 2f_{\pi}(x)$. 
	Thus we have $|\hat{t}_{\pi, k}(x)|\le 3M^k = 3(2x + a + 2k)^k\le 3(2x+1)^k(a+2k)^k$, and we may take $C = (a+2k)^k$. The case $x < 0$ can be dealt with analogously.  
\end{proof}

\section{Compound decision problem}\label{app:compound}
A related setting is the compound decision problem, 
where the hidden parameters $\theta_1^n := (\theta_1, \cdots, \theta_n)$ are not random. 
More specifically, one draws $(\nu_1, \cdots, \nu_n)$ according to $\pi_{\theta_1^n} := \frac 1n\sum_{i=1}^n \delta_{\theta_i}$. 
Each observation $X_i$ is then generated through $f(\cdot; \nu_i)$. 

Similarly to the mean estimation problem, there are two choices of oracle estimators for $\ell(\theta)$. Both are of the form 
\[
\argmin_{\hat{T}}\mathbb{E}[||\hat{T}(X_1^n) - \ell(\theta_1^n)||^2]
\]
The oracle based on \textit{simple} decision criterion has $\hat{T}$ optimized over the coordinate-wise rules: 
\[
\mathcal{D}^S_n := \{\hat{T}: \mathbb{R}^n\to\mathbb{R}^n; \hat{t}: \R\to \R; \hat{T}(X_1^n) = (\hat{t}(X_1), \cdots, \hat{t}(X_n))\}
\]
Note that such optimal $\hat{t}$ satisfies $\hat{t}(x) = \mathbb{E}[\ell(\theta)|X = x]$, i.e. is also the Bayes estimator. 
The oracle based on the \textit{permutation invariant} criterion, first introduced in \cite{greenshtein2009asymptotic}, has $\hat{T}$ optimized over the class 
\[
\mathcal{D}^{PI}_n := \{\hat{T}: \mathbb{R}^n\to\mathbb{R}^n; \forall \sigma\in S_n: \hat{T}(\sigma(X_1^n)) = \sigma(\hat{T}(X_1^n))\}
\]
i.e., the input and output under $\hat{T}$ are permuted in the same way. 

Denote $r^S_{\theta_1^n}$ and $r^{PI}_{\theta_1^n}$ as the risks (summed over $n$ coordinates) incurred by the simple and permutation invariant oracles, 
respectively. 
Then we may similarly define the total regret of a permutation invariant $\hat{T}$ as: 
\[
\mathsf{TotRegretComp}_{\ell, \theta_1^n}^S (\hat{T}) = \E[||\hat{T}(X_1^n) - \ell(\theta_1^n)||_2^2] - r^S(\theta)
\]
\[
\mathsf{TotRegretComp}_{\ell, \theta_1^n}^{PI} (\hat{T}) = \E[||\hat{T}(X_1^n) - \ell(\theta_1^n)||_2^2] - r^{PI}(\theta)
\]
Since simple decision oracles are also permutation invariant, $r^{S}\ge r^{PI}$ and therefore 
$\mathsf{TotRegretComp}_{\pi, \ell, \theta^n}^S\le \mathsf{TotRegretComp}_{\pi, \ell, \theta^n}^{PI}$. 
On the other hand, \cite[Corollary 5.2]{greenshtein2009asymptotic} shows that in the mean estimation problem of the normal means model, the gap between $r^{S}$ and $r^{PI}$ is $O(1)$ under suitable conditions. In our moment estimation problem under normal means, 
an explicit bound can be further quantified by considering the channel $\eta\to \mathcal{N}(\eta^{1/k}, 1)$, and utilizing a recent result in \cite[Lemma 2.5]{han2024approximate} to obtain a bound of 
\[
r^{S} - r^{PI}\le C\cdot h^{2k}\cdot (e^{4h^2} - 1)(1+h)
\]
for some absolute constant $C$ whenever $\theta_1^n\in [0, h]^{n}$. 

In a similar way, we may define the minimax total regret over a class $\mathcal{G}$ of priors as 
\[
\mathsf{TotRegretComp}_{\ell, n}^S (\mathcal{G}) = \inf_{\hat{T}}\sup_{\theta_1^n\in \mathcal{G}^{\otimes n}}\mathsf{TotRegretComp}_{\ell, \theta^n}^S (\hat{T})
\]\[
\mathsf{TotRegretComp}_{\ell, n}^{PI} (\mathcal{G}) = \inf_{\hat{T}}\sup_{\theta_1^n\in \mathcal{G}^{\otimes n}}\mathsf{TotRegretComp}_{\ell, \theta^n}^{PI} (\hat{T})
\]
Unlike empirical Bayes minimax total regret, much less is known about the upper bounds on total regret in the compound decision problem. 
For the normal means model, the best known upper bound of $\mathsf{TotRegretComp}$ is $O(\frac{\log^{4.5} n}{\log \log^{1.5} n})$, achieved by \cite[Lemma 1]{JZ09} and by bounding the Hellinger rate of estimating Gaussian mixtures, 
which is of polylog factor away from the lower bound. 
For the Poisson model, the only upper bound established is for heavy-tailed setting in \cite[Appendix F]{SW22}, 
which matches the corresponding lower bound up to polylog factors. 
Although one may adapt it to the setting where $\theta_1^n\in [0, h]^{n}$, 
the bound is likely suboptimal, with extra logarithmic factors. 

On the other hand, \cite[Proposition 3]{polyanskiy2021sharp} provides a way to establish a compound decision minimax regret lower bound using an empirical Bayes minimax regret lower bound. We now establish an important corollary for moment estimation in the compound estimation setting. 

\begin{corollary}\label{cor:totcompregret}
	When $\ell(\theta) := \theta^k$, the minimax total regret over $\mathcal{P}([0, h])$ has a lower bound of $\Omega\left((\frac{\log n}{\log\log n})^{k+1}\right)$ in both the Poisson and normal means models. 
\end{corollary}

\begin{proof}[Proof of \prettyref{cor:totcompregret}]
	We mimic the idea of \cite[Proposition 3]{polyanskiy2021sharp} that establishes 
	\[
	\mathsf{TotRegretComp}_{\ell, n}^S (\mathcal{G})\ge \mathsf{TotRegret}_{\ell, n} (\mathcal{G})
	\]
	whenever $\ell := \theta^k$.
	Next, \cite[Corollary 1]{wu2010functional} says that the function $G\to \mathsf{mmse}(G)$ is concave, when applied to channels that are either of the form 
	$\eta \to \mathcal{N}(\eta^{1/k}, 1)$ or $\eta \to \mathsf{Poi}(\eta^{1/k})$. 
	This gives, for each $G$, $\E_{\theta_1^n \simiid G}\mmse(G_{\theta_1^n}) \le \mmse(G)$, 
	and therefore for each $\hat{T}$ and each $G$ we have 
	\[
	\mathsf{TotRegret}_{\pi, \ell, n}(\hat{T}) \le \E_{\theta_1^n \simiid \pi}[\mathsf{TotRegretComp}_{\ell, \theta_1^n}^S(\hat{T})]
	\]
	Therefore, the claim is established by taking $\inf_{\hat{T}}\sup_{\pi}$ on both sides. 
\end{proof}
\end{document}

%% file: pkgs_and_defns.tex
\usepackage{xspace,exscale,relsize}
\usepackage{fancybox,shadow}

\usepackage{color}
\usepackage{amsfonts}

\usepackage{mathtools}

\makeatletter
\let\over=\@@over \let\overwithdelims=\@@overwithdelims
\let\atop=\@@atop \let\atopwithdelims=\@@atopwithdelims
\let\above=\@@above \let\abovewithdelims=\@@abovewithdelims
\makeatother
\usepackage{rotating}

\usepackage{float}
\usepackage{algorithm}
\usepackage{algpseudocode}

\usepackage{ifpdf}

\usepackage{psfrag}

\usepackage{prettyref,enumerate}

\usepackage{tikz}
\usetikzlibrary{arrows}
\tikzstyle{int}=[draw, fill=blue!20, minimum size=2em]
\tikzstyle{dot}=[circle, draw, fill=blue!20, minimum size=2em]
\tikzstyle{init} = [pin edge={to-,thin,black}]

\usepackage[all]{xy}

\ifx\eqref\undefined
\newcommand{\eqref}[1]{~(\ref{#1})}
\fi
\ifx\mod\undefined
\def\mod{\mathop{\rm mod}}
\fi

\usepackage{bm}

\newcommand{\norm}[1]{{\left\Vert #1 \right\Vert}}

\newcommand{\argmin}{\mathop{\rm argmin}}
\newcommand{\argmax}{\mathop{\rm argmax}}

\def\exp{\mathop{\rm exp}}

\def\EE{\Expect}

\def\PP{\mathbb{P}}

\def\simiid{\stackrel{\text{i.i.d.}}\sim}

\makeatletter
\def\bbordermatrix#1{\begingroup \m@th
	\@tempdima 4.75\p@
	\setbox\z@\vbox{%
		\def\cr{\crcr\noalign{\kern2\p@\global\let\cr\endline}}%
		\ialign{$##$\hfil\kern2\p@\kern\@tempdima&\thinspace\hfil$##$\hfil
			&&\quad\hfil$##$\hfil\crcr
			\omit\strut\hfil\crcr\noalign{\kern-\baselineskip}%
			#1\crcr\omit\strut\cr}}%
	\setbox\tw@\vbox{\unvcopy\z@\global\setbox\@ne\lastbox}%
	\setbox\tw@\hbox{\unhbox\@ne\unskip\global\setbox\@ne\lastbox}%
	\setbox\tw@\hbox{$\kern\wd\@ne\kern-\@tempdima\left[\kern-\wd\@ne
		\global\setbox\@ne\vbox{\box\@ne\kern2\p@}%
		\vcenter{\kern-\ht\@ne\unvbox\z@\kern-\baselineskip}\,\right]$}%
	\null\;\vbox{\kern\ht\@ne\box\tw@}\endgroup}
\makeatother

\newcommand{\stepa}[1]{\overset{\rm (a)}{#1}}
\newcommand{\stepb}[1]{\overset{\rm (b)}{#1}}
\newcommand{\stepc}[1]{\overset{\rm (c)}{#1}}
\newcommand{\stepd}[1]{\overset{\rm (d)}{#1}}

\newcommand{\Poi}{\mathrm{Poi}}

\newcommand{\Expect}{\mathbb{E}}

\newcommand{\TV}{{\rm TV}}

\newcommand{\indc}[1]{{\mathbf{1}_{\left\{{#1}\right\}}}}

\definecolor{myblue}{rgb}{.8, .8, 1}
\definecolor{mathblue}{rgb}{0.2472, 0.24, 0.6} 
\definecolor{mathred}{rgb}{0.6, 0.24, 0.442893}
\definecolor{mathyellow}{rgb}{0.6, 0.547014, 0.24}

\newcommand{\calF}{{\mathcal{F}}}
\newcommand{\calG}{{\mathcal{G}}}

\newcommand{\calN}{{\mathcal{N}}}

\newcommand{\calP}{{\mathcal{P}}}
\newcommand{\calQ}{{\mathcal{Q}}}

\newcommand{\mmse}{\mathsf{mmse}}

\def\R{\mathbb{R}}

\def\Z{\mathbb{Z}}
\def\E{\mathbb{E}}

\def\PP{\mathbb{P}}

\newrefformat{eq}{(\ref{#1})}
\newrefformat{thm}{Theorem~\ref{#1}}
\newrefformat{th}{Theorem~\ref{#1}}
\newrefformat{chap}{Chapter~\ref{#1}}
\newrefformat{defn}{Definition~\ref{#1}}
\newrefformat{sec}{Section~\ref{#1}}
\newrefformat{seca}{Section~\ref{#1}}
\newrefformat{algo}{Algorithm~\ref{#1}}
\newrefformat{fig}{Fig.~\ref{#1}}
\newrefformat{tab}{Table~\ref{#1}}
\newrefformat{rmk}{Remark~\ref{#1}}
\newrefformat{clm}{Claim~\ref{#1}}
\newrefformat{def}{Definition~\ref{#1}}
\newrefformat{cor}{Corollary~\ref{#1}}
\newrefformat{lmm}{Lemma~\ref{#1}}
\newrefformat{prop}{Proposition~\ref{#1}}
\newrefformat{pr}{Proposition~\ref{#1}}
\newrefformat{property}{Property~\ref{#1}}
\newrefformat{app}{Appendix~\ref{#1}}
\newrefformat{apx}{Appendix~\ref{#1}}
\newrefformat{ex}{Example~\ref{#1}}
\newrefformat{exer}{Exercise~\ref{#1}}
\newrefformat{soln}{Solution~\ref{#1}}

\def\unifto{\mathop{{\mskip 3mu plus 2mu minus 1mu%
			\setbox0=\hbox{$\mathchar"3221$}%
			\raise.6ex\copy0\kern-\wd0%
			\lower0.5ex\hbox{$\mathchar"3221$}}\mskip 3mu plus 2mu minus 1mu}}

\ifx\lesssim\undefined
\def\simleq{{{\mskip 3mu plus 2mu minus 1mu%
			\setbox0=\hbox{$\mathchar"013C$}%
			\raise.2ex\copy0\kern-\wd0%
			\lower0.9ex\hbox{$\mathchar"0218$}}\mskip 3mu plus 2mu minus 1mu}}
\else
\def\simleq{\lesssim}
\fi

\ifx\gtrsim\undefined
\def\simgeq{{{\mskip 3mu plus 2mu minus 1mu%
			\setbox0=\hbox{$\mathchar"013E$}%
			\raise.2ex\copy0\kern-\wd0%
			\lower0.9ex\hbox{$\mathchar"0218$}}\mskip 3mu plus 2mu minus 1mu}}
\else
\def\simgeq{\gtrsim}
\fi

\newif\ifmapx
{\catcode`/=0 \catcode`\\=12/gdef/mkillslash\#1{#1}}
\edef\jobnametmp{\expandafter\string\csname ic_apx\endcsname}
\edef\jobnameapx{\expandafter\mkillslash\jobnametmp}
\edef\jobnameexpand{\jobname}
\ifx\jobnameexpand\jobnameapx
\mapxtrue
\else
\mapxfalse
\fi

\newcommand{\Tpiell}{\hat{T}_{\pi, \ell}}
\newcommand{\Tpik}{\hat{T}_{\pi, k}}
\newcommand{\tpiell}{\hat{t}_{\pi, \ell}}
\newcommand{\tpik}{\hat{t}_{\pi, k}}
\newcommand{\thetaerm}{\hat{T}_\mathsf{erm}}
\newcommand{\Termk}{\hat{T}_{\mathsf{erm}, k}}
\newcommand{\Termell}{\hat{T}_{\mathsf{erm}, \ell}}

\newcommand{\Tnpmlek}{\hat{T}_{\mathsf{NPMLE}, k}}

\newcommand{\Termclipped}[1]{\hat{T}_{\mathsf{erm}, #1}}

\newcommand{\Trobk}{\hat {T}_{\mathsf{Rob}, k}}
\newcommand{\thetarobell}{\hat {T}_{\mathsf{Rob}, \ell}}
\newcommand{\sfem}{\mathsf{emp}}
\newcommand{\subexpo}{\mathsf{SubE}}
\newcommand{\Regret}{\mathsf{Regret}}
\newcommand{\TotRegret}{\mathsf{TotRegret}}

\newcommand{\SubG}{\mathsf{SubG}}

\renewcommand{\hat}{\widehat}
\renewcommand{\tilde}{\widetilde}

\newcommand{\bbR}{\mathbb R}

\newcommand{\bbZ}{\mathbb Z}

\newcommand{\bbE}{\mathbb E}
\newcommand{\bbP}{\mathbb P}